\documentclass[12pt,a4paper]{amsart}
\usepackage{amsmath,amsfonts,amssymb,amscd,amsthm}
\usepackage[shortlabels]{enumitem}

\usepackage{color, soul}
\definecolor{gr}{rgb}{0.7, 1, 0.7}
\definecolor{rr}{rgb}{1, 0.7, 0.7}

\usepackage{latexsym}
\usepackage{graphicx}

\theoremstyle{plain} 
\newtheorem{theorem}{Theorem}[section]
\newtheorem{lemma}[theorem]{Lemma}
\newtheorem{corollary}[theorem]{Corollary}

\newtheorem{proposition}[theorem]{Proposition}

\theoremstyle{definition} 
\newtheorem{definition}[theorem]{Definition}
\theoremstyle{remark} 
\newtheorem{remark}[theorem]{Remark}

\renewcommand{\mathfrak}{\mathbf}

\renewcommand{\Im}{\,\mathrm{Im}\,}
\renewcommand{\Re}{\,\mathrm{Re}\,}

\newcommand{\ignore}[1]{}

\newcommand{\eps}{\epsilon}

\newcommand{\cA}{\mathcal{A}}

\newcommand{\bbC}{\mathbb{C}}
\newcommand{\bbH}{\mathbb{H}}

\newcommand{\bbN}{\mathbb{N}}
\newcommand{\bbR}{\mathbb{R}}
\newcommand{\bbZ}{\mathbb{Z}}
\newcommand{\bbQ}{\mathbb{Q}}
\newcommand{\bbT}{\mathbb{T}}
\newcommand{\bbD}{\mathbb{D}}

\newcommand{\mcE}{\mathbf{E}}
\newcommand{\bP}{\mathbf P}
\newcommand{\tl}{\tilde}
\newcommand{\NN}{\mathbb N}
\newcommand{\CC}{\mathbb C}

\newcommand{\RR}{\mathbb R}

\newcommand{\cR}{\mathcal R}
\newcommand{\cI}{\mathcal I}

\newcommand{\cU}{\mathcal U}

\newcommand{\cV}{\mathcal V}
\newcommand{\cW}{\mathcal W}

\newcommand{\cE}{\mathbf E}
\newcommand{\cH}{\mathcal H}

\newcommand{\kC}{\mathfrak C}

\newcommand{\ccm}{\mathbf W}
\newcommand{\kL}{\mathfrak L}
\newcommand{\dist}{\operatorname{dist}}
\newcommand{\diam}{\operatorname{diam}}
\newcommand{\hol}{\mathbf H}

\newcommand{\cren}{{\mathcal R}_{\text{cyl}}}
\newcommand{\hren}{{\hat{\mathcal R}}_{\text{cyl}}}
\renewcommand{\Im}{\operatorname{Im}}
\renewcommand{\mod}{\operatorname{mod}}

\begin{document}
\title[Critical circle maps with non-integer exponents]{Rigidity, universality, and hyperbolicity of renormalization for critical circle maps with non-integer exponents}

\author{Igors Gorbovickis}
\address{Uppsala University, Uppsala, Sweden}
\email{igors.gorbovickis@math.uu.se}

\author{Michael Yampolsky}
\address{University of Toronto, Toronto, Canada}
\email{yampol@math.toronto.edu}


\begin{abstract}
We construct a renormalization operator which acts on analytic circle maps whose critical exponent $\alpha$ is not necessarily an 
odd integer $2n+1$, $n\in\bbN$. When $\alpha=2n+1$, our definition generalizes
 cylinder renormalization of analytic
critical circle maps \cite{Ya3}. In the case when $\alpha$ is close to an odd integer, we prove hyperbolicity of renormalization 
for maps of bounded type. We use it to prove universality and $C^{1+\alpha}$-rigidity for such maps.
\end{abstract}
\date{\today}
\maketitle

\section{Introduction}
\label{sec:intro}
Renormalization theory of critical circle maps was developed in the early 1980's to explain  universality phenomena in smooth families 
of circle homeomorphisms with one critical point, the so-called \textit{critical circle maps}. 
By definition, a critical circle map is a $C^3$ homeomorphism of the circle $\bbT=\bbR/\bbZ$ with a single critical point of an odd integer order $\alpha>1$. A canonical example is given by the {\it Arnold's family} which is constructed as follows. For $\theta\in\bbR$ let
$$A_\theta(x)=x+\theta-\frac{1}{2\pi}\sin{2\pi x}.$$
The maps $A_\theta$ are analytic orientation-preserving  homeomorphisms $\bbR\to\bbR$ with critical points of cubic type at integer values of $x$.
Since $A_\theta(x+1)=A_\theta(x)+1$, it projects to a well-defined critical circle map $f_\theta$. 
Of course, values of $\theta$ which differ by an integer produce the same maps $f_\theta$, so it is natural to consider the
parametric family $\theta\mapsto f_\theta$ with $\theta\in\bbT$.

 To give an example of universality phenomena, let us describe the {\it golden-mean universality}. Let $g_\theta$ be a family of critical circle maps
with critical exponent $\alpha$, depending on a parameter $\theta\in\bbT$ which is smooth in $\theta$ and has the property $\partial \tilde g_\theta(x)/\partial\theta>0$ for all $x\in\bbR$, where $\tl g_\theta:\bbR\to\bbR$ is any lift of the family to a smooth family of homeomorphisms of the real line. The Arnold's family is an example of such a family with $\alpha=3$.
Fix $\rho_*=(\sqrt{5}-1)/2$, the inverse golden mean. This irrational number has a continued fraction expansion 
$$
\rho_*=\cfrac{1}{1+\cfrac{1}{1+\cfrac{1}{1+\dotsb}}}
$$
(further on we will abbreviate such expressions as $[1,1,1,\ldots]$
for typographical convenience). The $n$-th convergent of this continued fraction is a rational number 
$$p_n/q_n=\underbrace{[1,1,1,\ldots,1]}_n.$$
Suppose $\theta_*$ is a parameter for which the rotation number
$$\rho(f_{\theta_*})=\rho_*.$$
It is not hard to see that there is a sequence of closed intervals $I_n$ consisting of parameters $\theta$ for which $\rho(f_\theta)=p_n/q_n$,
which converges to the parameter $\theta_*$. The empirical observation in this case is that their lengths
 $$|I_n|\sim a\delta^{-n}$$
for some $a>0$ and $\delta>1$. Moreover, $\delta$ is a universal constant - its value is the same for all {families of critical circle maps $g_\theta$} with the same critical exponent
$\alpha$ (so $\delta=\delta(\alpha)$).

 In~\cite{ORSS} and~\cite{FKS} this  
universality phenomenon was translated into conjectural hyperbolicity of a renormalization transformation $\cR$, a certain non-linear operator acting on the space of commuting pairs, where commuting pairs are pairs of maps closely related to critical circle maps (see \S\ref{Comm_pairs_sec}, \S\ref{renorm_com_pairs_sec} for precise definitions). 
The conjectures were fully developed in the works of O.~Lanford \cite{Lan2,Lan3,Lan4} and became known as {\it Lanford's program}. 
 
Let us mention here the key steps of Lanford's program, while a detailed discussion can be found in  \cite{Ya3,Ya4}. The program consists of proving the following statements:

\begin{enumerate}
\item \textbf{Global horseshoe attractor.} The renormalization operator $\cR$ possesses a ``horseshoe'' attractor in the space of commuting pairs equipped with a $C^r$-metric. The action of $\cR$ on the attractor is topologically conjugated to the shift on the space of bi-infinite sequences of natural numbers. Moreover, the orbits of commuting pairs with the same irrational rotation number converge to the same orbit in the attractor.

\item \textbf{Hyperbolicity of the attractor.} There exists a neighborhood of the attractor that admits a structure of an infinite dimensional smooth manifold compatible with $C^r$-metric. The horseshoe attractor is a hyperbolic invariant set of the renormalization operator with respect to this manifold structure, and the unstable direction at each point of the attractor is one-dimensional.
\end{enumerate}

Another important consequence of Lanford's program is \textit{rigidity of critical circle mappings}. The question of rigidity in the context of circle homeomorphisms goes back Poincar\'e, who showed that every circle homeomorphism with an irrational rotation number is semi-conjugate to a rigid rotation. In the particular case of critical circle maps Yoccoz~\cite{YocDen} proved that any two such maps $f_1$, $f_2$ with the same irrational rotation number are topologically conjugated. This result can be thought of as an analog of the Denjoy's theorem for circle diffeomorphisms. Later Yoccoz showed that the conjugacy between critical circle maps $f_1$ and $f_2$ is quasisymmetric. Finally, Khanin and Teplinsky~\cite{Khanin_Teplinsky_2007} proved that exponential convergence of the orbits of $f_1$ and $f_2$ under renormalization implies $C^1$-smoothness of the conjugacy.
Furthermore, if $f_1$ and $f_2$ are of bounded combinatorial type (c.f. Definition~\ref{Bounded_type_def}), then according to~\cite{FM1}, this conjugacy is $C^{1+\beta}$-smooth, for some $\beta>0$.

A key breakthrough in the study of renormalization of one-dimensional dynamical systems in general and of critical circle maps in particular
was made by D.~Sullivan \cite{Sull-qc,Sullivan-bounds} who introduced 
methods of holomorphic dynamics and Teichm\"uller theory into the subject. Extending the ideas of Sullivan, McMullen~\cite{McM-ren2}, Lyubich~\cite{Lyubich-Hairiness}, de Faria~\cite{DeFar} and de Faria-de Melo~\cite{FM1,FM2},
the second author brought Lanford's program for analytic critical circle maps with $\alpha\in 2\bbN+1$ to a successful completion in a series of works \cite{Ya1,Yamp-towers,Ya3,Ya4, Ya6}. In particular,
in \cite{Ya3,Ya4} he introduced a new renormalization transformation, known as the {\it cylinder renormalization operator} $\cren$, and then showed that this operator has a hyperbolic horseshoe attractor with a one-dimensional unstable direction.

It is well-known, that successive renormalizations of a $C^3$-smooth map with a critical point of an odd order $\alpha>1$ converge to a certain space of analytic maps (cf. \cite{FM1}), and the above developments happened in the analytic realm. 
However, empirical observations of universality were made for families of differentiable homeomorphisms of the circle with the critical point at $0$ of the form
\begin{equation}
\label{sing1}
\psi\circ q_\alpha\circ\phi,
\end{equation}
where $\phi$ and $\psi$ are local $C^3$-diffeomorphisms, $\phi(0)=0$ and
\begin{equation}
\label{sing2}
q_\alpha(x)=x|x|^{\alpha-1}\text{ where }\alpha>1
\end{equation}
even when $\alpha$ is not an odd integer. It is clear that such maps cannot be analytic at zero, which does not allow to apply previously known analytic methods for maps of this type.

The above discussion naturally suggests the following definition: a homeomorphism $f\colon\bbT\to\bbT$ is called an \textit{analytic critical circle map with critical exponent} $\alpha$, if it is analytic everywhere except possibly the critical point, in the neighborhood of which $f$ can be represented in the form~(\ref{sing1}), where $\phi$ and $\psi$ are analytic.

In this paper we complete Lanford's program for analytic critical circle maps of bounded combinatorial type with values of the critical exponent sufficiently close to odd integers greater than $1$. In particular, we 
\begin{enumerate}
\item extend the definition of renormalization operator $\cren$ in a suitable fashion to a space of maps with critical points of the type~(\ref{sing1});
\item extend the hyperbolicity results of \cite{Ya3,Ya4} to maps with critical points of non-integer order, sufficiently close to odd integers greater than $1$ and with bounded combinatorics;
\item prove existence of a \textit{global} horseshoe attractor for renormalizations of such maps.
\end{enumerate}
\noindent As a corollary of the above results, we prove
\begin{enumerate}
\setcounter{enumi}{3}
\item universality for families of analytic critical circle maps with values of critical exponents sufficiently close to odd integers greater than $1$; 
\item $C^{1+\beta}$-rigidity for such maps of bounded combinatorial type.
\end{enumerate}

\noindent Let us mention, that a very different approach to  questions similar to (1)-(2) for renormalization of unimodal maps was suggested in \cite{SmCr}.

Before formulating our results, let us give a few useful definitions. 
Suppose, $\mathbf B$ is a complex Banach space whose elements are functions of the complex variable. Following the 
notation of \cite{Ya3}, let us say that
the {\it real slice} of $\mathbf B$ is the real Banach space $\mathbf B^\RR$ consisting of the real-symmetric elements of $\mathbf B$.
If $\mathbf X$ is a Banach manifold modelled on $\mathbf B$ with the atlas $\{\Psi_\gamma\}$
we shall say that $\mathbf X$ is {\it real-symmetric} if $\Psi_{\gamma_1}\circ\Psi_{\gamma_2}^{-1}(U)\subset \mathbf B^\RR$ for any pair of indices $\gamma_1$, $\gamma_2$ and any open set $U\subset\mathbf B^\RR$, for which the above composition is defined. The {\it real slice of $\mathbf X$} is then defined as the real 
Banach manifold $\mathbf X^\RR= \cup_\gamma\Psi_\gamma^{-1}(\mathbf B^\RR) \subset \mathbf X$ 
with an atlas $\{\Psi_\gamma\}$.
An operator $A$ defined on a {subset} $Y\subset\mathbf X$ is {\it real-symmetric} if $A(Y\cap\mathbf X^\RR)\subset \mathbf X^\RR$.

We prove the following theorems.
\begin{theorem}
\label{main1}
For every positive integer $k>0$, there exists a real-symmetric analytic Banach manifold $\mathfrak N^{2k+1}$ containing a disjoint union of sets $\mathfrak C^\alpha$, parameterized by $\alpha\in(2k+1-\eps,2k+1+\eps)$, $\eps=\eps(k)>0$, so that for each such $\alpha$, the set $\mathfrak C^\alpha\cap(\mathfrak N^{2k+1})^\bbR$ is non-empty and consists of analytic critical circle maps with critical exponent $\alpha$. 
The cylinder renormalization operator $\cren$ extends to a real-symmetric analytic operator on an open subset of $\mathfrak N^{2k+1}$, and 
if $f\in\mathfrak C^\alpha$, for some parameter $\alpha$, and $\cren(f)$ is defined, then $\cren(f)\in\mathfrak C^\alpha$.
\end{theorem}

\noindent
Furthermore,

\begin{theorem}
\label{main2}
For every positive integer $k>0$, there exists a real-symmetric analytic Banach manifold $\mathfrak M^{2k+1}$, an open interval $I_k\subset\bbR$, such that $2k+1\in I_k$ and a family of real-symmetric analytic maps 
{small change here}
$i_\alpha\colon\mathfrak M^{2k+1}\to\mathfrak C^\alpha\subset\mathfrak N^{2k+1}$ that analytically depend on $\alpha\in I_k$. In addition to that, there exists a family of real-symmetric analytic operators $\mathcal R_\alpha$ that analytically depend on $\alpha\in I_k$ and are defined on a certain open subset of $\mathfrak M^{2k+1}$, so that the diagram
$$
\begin{CD}
\mathfrak M^{2k+1}  @>\mathcal R_\alpha>>  \mathfrak M^{2k+1} \\
@VV i_\alpha V  @VV i_\alpha V\\
\mathfrak C^\alpha @>\cren>> \mathfrak C^\alpha
\end{CD}
$$
commutes, whenever all maps are defined. Moreover, for any positive integer $B\ge 1$, the operator $\mathcal R_{2k+1}$ has a hyperbolic horseshoe attractor of type bounded by $B$ with a one-dimensional unstable direction.
\end{theorem}

{The manifolds $\mathfrak M^{2k+1}$ are ``large'': the image $i_{2k+1}(\mathfrak M^{2k+1})$ contains an open neighborhood of the renormalization horseshoe attractor in the 
space of critical cylinder maps with the exponent $2k+1$ constructed in} \cite{Ya3}. Furthermore, in a forthcoming
paper, we will demonstrate that renormalizations of smooth circle homeomorpshisms with a unique singularity of the form
(\ref{sing1})  converge to $i_\alpha({\mathfrak M}^{2k+1})$.

It is worth pointing out that the maps  $i_\alpha\colon\mathfrak M^{2k+1}\to\mathfrak C^\alpha$ are not injective when $\alpha$ is rational. 

\noindent
As a consequence of Theorem~\ref{main2}, we derive:

\begin{theorem}[\bf{ Renormalization hyperbolicity}]
\label{main3}
For any pair of integers $B, k\ge 1$ there exists $\varepsilon >0$, such that for any $\alpha\in (2k+1-\varepsilon,2k+1+\varepsilon)$, the operator $\mathcal R_\alpha$ has a hyperbolic horseshoe attractor of type bounded by $B$ with a one-dimensional unstable direction.
\end{theorem}

{Hyperbolicity of the attractors in all of the above theorems is understood in the sense of Definition}~\ref{hyperb_def}, given in the next section.

In order to simplify the notation, we will provide proofs of Theorems~\ref{main1},~\ref{main2},~\ref{main3} only for the case when $k=1$. For $k>1$ the proofs are identical, so for the remaining part of the paper we assume that in the above theorems the parameter $k$ is always set to $k=1$.

The above theorems establish the local hyperbolic structure of the horseshoe attractor of renormalization. What follows next are the global results.

For every $\alpha>1$ we naturally define the space $\mathcal A^\alpha$ of analytic commuting pairs with the critical exponent $\alpha$, so that renormalizations of analytic critical circle maps with singularity of the form~(\ref{sing1}) belong to $\mathcal A^\alpha$ (c.f. Definitions~\ref{A_r_alpha_def},~\ref{A_alpha_def}). Our first result is the following:
\begin{theorem}[{\bf Global renormalization attractor}]
\label{main4}
For every $k,B\in\bbN$, there exists an open interval $J=J(k,B)\subset\bbR$, such that $2k+1\in J$, and for every $\alpha\in J$, the renormalization operator restricted to $\mathcal A^\alpha$ has a global horseshoe attractor $\mathcal I_B^\alpha\subset\mathcal A^\alpha$ of type bounded by $B$. Furthermore, renormalizations of any two commuting pairs $\zeta_1,\zeta_2\in\mathcal A^\alpha$ with the same irrational rotation number of type bounded by $B$, converge exponentially in $C^r$-metric, for every non-negative $r$.
\end{theorem}
In \S\ref{sec:global} we formulate and prove an expanded version of this theorem (c.f. Theorem~\ref{Global_attractor_theorem_expanded}).

We will say that $\{f_t\mid t\in(-1,1)\}$ is a $C^1$-smooth one-paremeter  family of analytic critical circle maps with critical exponent $\alpha$, if 
there exist neighborhoods $V_2\subset V_1\subset\bbC\slash\bbZ$, such that $\bbT\subset V_1$, $0\in V_2$ and for every $t$, the critical circle map $f_t$ extends to an analytic function in $V_1\setminus\{\Re z=0\}$ and to an analytic (multiple valued) function in $V_2\setminus\{0\}$, where it can be represented as $f_t=\psi_t\circ q_\alpha\circ\phi_t$, for some conformal maps $\phi_t$ and $\psi_t$ defined on $V_2$ and $q_\alpha(\phi_t(V_2))$ respectively. Furthermore $\phi_t$ and $\psi_t$ are required to be $C^1$-smooth in $t$ and $\phi_t(0)=0$.

As a first consequence of Theorem~\ref{main4}, we obtain a universality statement:
\begin{theorem}[{\bf Universality}]\label{main5}
We adopt the notation of Theorem \ref{main4}. 
Let $\rho$ be an irrational rotation number of type bounded by $B$ which is periodic under the Gauss map with period $p$: 
$$\rho=[r_0,r_1,\ldots,r_{p-1},r_0,\ldots,r_{p-1},\ldots] .$$ 
There exists a positive integer $N=N(k)$, with the property that for every $\alpha\in J(k,B)$, there exists $\delta=\delta(\rho,\alpha)>1$ such that the following holds. 
Suppose 
$\{f_t\mid t\in(-1,1)\}$ 
is a $C^1$-smooth one-paremeter  family of analytic critical circle maps with critical exponent $\alpha$ such that:
\begin{itemize}
\item for every $x\in\bbT$, the derivative $\frac{\partial }{\partial t}f_t(x)>0;$
\item the rotation number $\rho(f_0)=\rho$;
\end{itemize}
Let $I_m\to 0$ denote the sequence of closed intervals of parameters $t$ such that the first $Npm$ numbers in the continued fraction of $\rho(f_t)$ coincide with those of $\rho$. Then
$$|I_m|\sim a\delta^{-m},\text{ for some }a>0.$$ 
\end{theorem}
Furthermore,  we prove the following rigidity result:
\begin{theorem}[{\bf Rigidity}]
\label{main6}
Every two analytic critical circle maps with the same irrational rotation number of type bounded by $B$ and with the same critical exponent $\alpha\in\cup_{k\in\bbN} J(k,B)$, are $C^{1+\beta}$-conjugate, where $\beta>0$ depends on $B$ and $\alpha$.
\end{theorem}

Having stated the main results, let us also highlight several conceptual issues that we handled in the paper. In the case of analytic critical circle maps (that is, when $k\in\NN$) there are two complementary approaches for defining renormalization. The ``classical'' one, going back to \cite{ORSS,FKS} is done in the language of commuting pairs (see Definition~\ref{commut_pair_def} below). In contrast, cylinder renormalization defined in \cite{Ya3} acts on analytic maps of an annulus, which restrict to critical circle maps (critical cylinder maps). Neither the classical definition of an analytic critical commuting pair, nor the definition of a critical cylinder map makes sense for the case when the singularity is of type (\ref{sing1}) with $\alpha\notin 2\NN+1$, since in this situation the iterates of our circle map cannot be analytically continued to a neighborhood of the origin. We finesse this difficulty in \S\ref{sec:global} for the case of commuting pairs and in \S\ref{sec:spaces} for critical cylinder maps. Moreover, we give a new proof of the existence of fundamental domains for cylinder renormalization which works for locally analytic maps and does not require any global structure.
Finally, 
in \S\ref{sec:renorm2} we introduce a new framework which bridges the gap between the two definitions of renormalization, and allows us to extend renormalization hyperbolicity results to $\alpha\notin 2\NN+1$.

The structure of the paper is as follows. In the next section \S\ref{sec:prelim} we recall the relevant facts of renormalization theory of critical circle maps. The reader can find a more detailed introduction in \cite{Ya3}. In \S\ref{sec:spaces} we introduce the functional spaces that are used in the construction of Banach manifolds from Theorems~\ref{main1} and \ref{main2}, and discuss their  properties. In \S\ref{sec:cres} we lay the ground for extending cylinder renormalization of \cite{Ya3} to maps with non odd-integer critical exponents, and generalize the results of \cite{Ya3} on the existence of fundamental crescents. In \S\ref{sec:renorm1}, we give a generalized definition of  cylinder renormalization. Then  in \S\ref{sec:renorm2} we construct the family of renormalizations $\cR_\alpha$ acting on $\mathfrak M^{2k+1}$. In \S\ref{sec:hyperb} we extend the renormalization hyperbolicity results of \cite{Ya3,Ya4} to a convenient setting. In \S\ref{sec:main_proofs} we complete the proof of renormalization hyperbolicity for $\alpha$ close to odd integers and derive Theorems \ref{main1}-\ref{main3}. Finally, in \S\ref{sec:global} we prove the global renormalization convergence, universality, and rigidity Theorems \ref{main4}-\ref{main6}.

\medskip
\noindent
{\bf Acknowledgment.} The authors would like to thank Mikhail Lyubich for very helpful suggestions.

\section{Preliminaries}
\label{sec:prelim}
\subsection{Critical circle maps}
To fix our ideas, we will always assume that the critical point of a critical circle map $f$ is placed at $0\in\bbT$.

Being a homeomorphism of the circle, a critical circle map has a well defined rotation number $\rho(f)$. It is useful to represent $\rho(f)$ as a finite or infinite continued fraction
$$\rho(f)=[r_0,r_1,r_2,\ldots]\text{ with }r_i\in\bbN.$$
Note that
an irrational number has a unique expansion as a continued fraction, so the following definition makes sense:
\begin{definition}\label{Bounded_type_def}
We will say that an irrational rotation number $\rho(f)=[r_0,r_1,r_2,\dots]$ (or $f$ itself) is of a 
type bounded by a positive constant $B$, if $\sup_{n\ge 0} r_n \le B$.
\end{definition}
For convenience of representing rational rotation numbers, let us add the symbol $\infty$ with the convention $1/\infty=0$. Then
$0=[\infty]$. For a non-zero rational rotation number $\rho(f)$ we will use the unique finite continued fraction expansion
 $$\rho=[r_0,r_1,\dots,r_m,\infty]$$ specified by the requirement $r_m>1$. 
For a rotation number $\rho(f)=[r_0,r_1,r_2,\dots]$ whose continued fraction contains at least $m+1$ terms, we denote by $p_m/q_m$ the 
$m$-th convergent $p_m/q_m= [r_0,\dots,r_{m-1}]$ of $\rho(f)$, written in the irreducible form.

If $f^{q_m}(0)\neq 0$, then we let
\begin{equation}\label{I_m_def}
I_m=[0,f^{q_m}(0)]
\end{equation}
denote the arc of the circle, which does not contain $f^{q_{m}+1}(0)$.
We note that $f^{q_m}(0)$ is a \textit{closest return of $0$}, that is, $I_m$ contains no iterates
$f^k(0)$ with $k<q_m$.

\subsection{Commuting pairs}\label{Comm_pairs_sec}
For two points $a,b\in\bbR$, by $[a,b]$ we will denote the closed interval with endpoints $a,b$ without specifying their order.

\begin{definition}\label{commut_pair_def}
Let $a, b\in\bbR$ be two real numbers, one of which is positive and another negative, and consider two intervals $I_\eta=[a,0]$ and $I_\xi=[b,0]$.
A commuting pair of class $C^r$, $r\geq 3$ ($C^\infty$, analytic) acting on the intervals $I_\eta$, $I_\xi$ is an ordered pair of maps $\zeta=(\eta,\xi)$, $\eta\colon I_\eta\to[a,b]$, $\xi\colon I_\xi\to[a,b]$ with the corresponding smoothness, such that the following properties hold:

\begin{enumerate}[(i)]
 \item $a=\xi(0)$ and $b=\eta(0)$;
 \item\label{com_pairs_ext_property} there exist $C^3$-smooth ($C^\infty$, analytic) extensions $\hat{\eta}\colon \hat I_\eta\to\bbR$ and $\hat{\xi}\colon \hat I_\xi\to\bbR$ of $\eta$ and $\xi$ respectively to some intervals $\hat I_\eta\Supset I_\eta$ and $\hat I_\xi\Supset I_\xi$, such that $\hat{\eta}$ and $\hat{\xi}$ are orientation-preserving homeomorphisms of $\hat I_\eta$ and $\hat I_\xi$ respectively, onto their images, and
$$
\hat{\eta}\circ \hat{\xi}=\hat{\xi}\circ\hat{\eta}
$$
where both compositions are defined;
 \item $\hat{\eta}'(x)\neq 0$ for all $x\in\hat I_\eta\setminus\{0\}$, and  $\hat{\xi}'(y)\neq 0$ for all $y\in\hat I_\xi\setminus\{0\}$; and $0$ is a cubic critical point for both maps;
 \item $\xi\circ\eta(0)\in I_\eta$.
\end{enumerate}
\end{definition}

\begin{definition}\label{ccp_def}
An analytic commuting pair $\zeta=(\eta,\xi)$ will be  called a \textit{critical commuting pair}.
\end{definition}

Given a commuting pair $\zeta=(\eta,\xi)$, we can identify the neighborhoods of the point $\eta(0)$ and $\xi\circ\eta(0)$ by the map $\hat{\xi}$ from property~\ref{com_pairs_ext_property} of Definition~\ref{commut_pair_def}. As a result of this identification, the interval $[\eta(0),\xi\circ\eta(0)]$ is glued into a smooth one dimensional compact manifold $M$ diffeomorphic to a circle. It aso follows from property~\ref{com_pairs_ext_property} of Definition~\ref{commut_pair_def} that the map
$$
f_\zeta(x)=\begin{cases}
\eta\circ\xi(x), &\text{if }x\in[\eta(0),0],\\
\eta(x), &\text{if }x\in[\xi\circ\eta(0),0]
\end{cases}
$$
projects to a smooth homeomorphism $F_\zeta$ of $M$ to itself that can have a unique critical point at the origin. 
Notice that there are many ways of putting an affine structure on the manifold $M$, which gives rise to the whole conjugacy class (the smoothness of the conjugacy is the same as that of the commuting pair) 
of critical circle maps that are conjugate to $F_\zeta$.

\subsection{Renormalization of commuting pairs}\label{renorm_com_pairs_sec}

\begin{definition}
The \textit{height} $\chi(\zeta)$ of a commuting pair $\zeta=(\eta,\xi)$ is the positive integer $r$, such that
$$
0\in [\eta^r(\xi(0)),\eta^{r+1}(\xi(0))).
$$
If no such $r$ exists, we set $\chi(\zeta)=\infty$.
\end{definition}

If $\chi(\zeta)=r\neq \infty$ and $\eta^r(\xi(0))\neq 0$, one can verify that
$$
\hat\zeta=(\eta^r\circ\xi,\eta)
$$
is a commuting pair acting on the intervals $I_\xi$, $[\eta^r(\xi(0)),0]$. It is known as the {\it pre-renormalization} of $\zeta$ and denoted by $p\mathcal R\zeta$.

\begin{definition}\label{renormalization_def}
If $\chi(\zeta)=r\neq \infty$ and $\eta^r(\xi(0))\neq 0$, then the renormalization $\mathcal R\zeta$ of a commuting pair $\zeta=(\eta,\xi)$ is defined to be the affine rescaling of the commuting pair $\hat\zeta=p\mathcal R\zeta$: 
$$
\mathcal R\zeta = (h\circ\eta^r\circ\xi\circ h^{-1},h\circ\eta\circ h^{-1}),
$$
where $h(x)=x/\eta(0)$. 
It is a commuting pair acting on the intervals $[0,1]$, $h([\eta^r(\xi(0)),0])$.
\end{definition}

\begin{definition}
For a commuting pair $\zeta$ define its \textit{rotation number} $\rho(\zeta)\in[0,1]$ to be equal to the continued fraction $[r_0,r_1,\dots]$, where $r_n=\chi(\mathcal R^n\zeta)$. Here, as before, $1/\infty$ is understood as~$0$.
\end{definition}
It is easy to see that the renormalization operator acts as a Gauss map on rotation numbers:
if $\rho(\zeta)=[r,r_1,r_2,\dots]$, then $\rho(\mathcal R\zeta)=[r_1,r_2,\dots]$.

The following proposition can be easily verified:
\begin{proposition}
For a commuting pair $\zeta$ the rotation number of the map $F_\zeta$
from \S~\ref{Comm_pairs_sec} 
is equal to $\rho(\zeta)$.
\end{proposition}

\subsection{From critical circle maps to commuting pairs}
Consider a critical circle map $f$ with rotation number $\rho(f)=[r_0,r_1,\dots]$, whose continued fraction expansion contains at 
least $m+2$ terms.  As before, set $p_m/q_m=[r_0,\dots,r_{m-1}]$ and $p_{m+1}/q_{m+1}=[r_0,\dots,r_m]$ and let $I_m=[0,f^{q_m}]$ and 
$I_{m+1}=[0,f^{q_{m+1}}]$ be as in~(\ref{I_m_def}). 
At the cost of a minor abuse of the notation, identify these intervals with their lifts to $(-1,1)\subset\bbR$. The pair of maps
$$
\zeta_m=(f^{q_{m+1}}, f^{q_m})
$$
is a commuting pair acting on the intervals $I_m$, $I_{m+1}$.
Note that,
if for a given critical circle map $f$, the corresponding commuting pairs $\zeta_m$ 
and $\zeta_{m+1}$ 
are defined, then $\zeta_{m+1}$ is the pre-renormalization of $\zeta_m$. This motivates the following definition:
\begin{definition}\label{renorm_f_def}
For a critical circle map $f$ and a positive integer $m$ as above, we denote
$$
p\cR^m f\equiv \zeta_m,
$$
and we define $\mathcal R^m f$ to be the affine rescaling of the commuting pair $\zeta_m$:
\begin{equation*}
\mathcal R^m f = (h\circ f^{q_{m+1}}\circ h^{-1}, h\circ f^{q_m}\circ h^{-1}),
\end{equation*}
where $h(x)=x/f^{q_m}(0)$.
\end{definition}

\subsection{Epstein class}
Given an open (possibly unbounded) interval $J\subset\bbR$, let $\bbC_J= \bbC\setminus(\bbR\setminus J)$ denote the plane slit along two rays. We will say that a \textit{topological disk} is an open simply-connected (not necessarily bounded) region in $\bbC$.
\begin{definition}
For $a\in\bbR$, $a\neq 0$, let $I$ denote the interval $I=[0,a]$. The \textit{Epstein class} $\mcE_I$ consists of all maps $g\colon I\to\bbR$ such that $g$ is an orientation preserving homeomorphisms of $I$ onto its image $g(I)=J$, and there exists an open interval $\tl J\supset {J}$ with the property that 
$g$ extends to an
analytic three-fold branched covering map of a topological disk
 $G\supset I$ 
onto the double-slit plane ${\bbC}_{\tl J}$ with a single critical point at $0$.
Let $\mcE$ denote the union $\mcE=\bigcup_I\mcE_I$ over all non-degenerate intervals $I$.
\end{definition}

Any map $g\in\mcE$ in the Epstein class can be decomposed as
\begin{equation}
\label{Epstein_decomp}
g=Q_c\circ h,
\end{equation}
where $Q_c(z)=z^3+c$, and $h:I\to \bbR$ is an analytic map with $h(0)=0$ that extends to a biholomorphism between $G$ and the complex plane with six slits,
which triple covers ${\bbC}_{\tl J}$ under the cubic map $Q_c(z)$. In what follows, we will always assume that for an Epstein map $g\colon I\to \bbR$, and $J=g(I)$, the interval $\tl J\supset J$ is the maximal (possibly unbounded) open interval, such that $g$ extends to a triple branched covering of a topological disk $G$ onto $\bbC_{\tl J}$.

We say that two positive constants $\alpha$ and $\beta$ are {\it $K$-commensurable } for some $K>1$ if
$$K^{-1}\alpha\leq\beta\leq K\alpha .$$

\begin{definition}\label{E_I_s_def}
For any $s\in (0,1)$ and an interval $I=[0,a]\subset\bbR$, let $\mcE_{I,s}\subset\mcE_I$ be the set that consists of all $g\in\mcE_I$, such that 
both $|I|$ and $\dist(I,J)$ are $s^{-1}$-commensurable with $|J|$, the length of each component of $\tl J\setminus J$ is at least $s|J|$, and $g'(a)\ge s$.
\end{definition}

We recall that the distortion of a conformal map $f$ in a domain $U\subset\bbC$ is defined as 
$$
\operatorname{Distortion}_U(f)=\sup_{z,w\in U}\frac{|f'(z)|}{|f'(w)|}.
$$
The following lemma was proved in~\cite{Yamp-towers}:
\begin{lemma}[Lemma 2.13 \cite{Yamp-towers}]\label{E_s_bounded_distortion_lemma}
For any $s\in(0,1)$ and an interval $I=[0,a]\subset\bbR$, there exists a domain $O_{I,s}\supset I$, such that for any $g\in \mcE_{I,s}$, the univalent map $h$ in~(\ref{Epstein_decomp}) is well-defined in $O_{I,s}$ and has $K(s)$-bounded distortion in $O_{I,s}$.
\end{lemma}

We will often refer to the space $\mcE$ as {\it the} Epstein class,
and to each ${\mcE}_{I,s}$ as {\it an} Epstein class.

We say that a critical commuting pair $\zeta=(\eta,\xi)$
belongs to the (an) Epstein class if both of its maps do. 
Allowing some abuse of notation, we will denote the space of all commuting pairs from the Epstein class by $\mcE$. Similarly, for any $s\in(0,1)$, by $\mcE_s$ we will denote the space of all commuting pairs $(\eta,\xi)$, such that $\eta\in\mcE_{[0,1],s}$ and $\xi\in\mcE_{I,s}$, where $I=[0,\eta(0)]$. 
It immediately follows from the definitions that:
\begin{lemma} 
If a renormalizable commuting pair $\zeta$ is in the Epstein class, then the same is true for
 $\cR \zeta$.
\end{lemma}

The following lemma is a consequence of real a priori bounds (c.f. с\cite{FM1}).

\begin{lemma}\label{E_s_bounded_size_lemma}
For any $s\in(0,1)$, there exist positive constants $C>1$ and $\lambda<1$, such that if $\zeta\in\mcE_s$ is $k$-times renormalizable and $I_\zeta^k=I_\eta^k\cup I_\xi^k$ denotes the domain of the commuting pair $p\cR^k\zeta$, then
$$
|I_\zeta^k|<C\lambda^k. 
$$
\end{lemma}

\subsection{Carath{\'e}odory convergence}
Consider 
the set of all triples $(U,u,f)$, where $U\subset \bbC$ is a topological
disk, $u\in U$ is a marked point in $U$, and $f:U\to \bbC$ is an analytic map. 
Recall (cf. \cite{McM-ren1}) that the \textit{Carath\'edory convergence} on 
this set
is defined in the following way:
$(U_n,u_n,f_n)\to (U,u,f)$, if and only if
\begin{itemize}
\item $u_n\to u\in  U$,
\item for any Hausdorff limit point $K$ of the sequence $\hat{\bbC}\setminus U_n$, $U$ is the component of 
$\hat{\bbC}\setminus K$ containing $u$, and
\item for all $n$ sufficiently large, $f_n$ converges to $f$ uniformly on compact subsets of $U$.
\end{itemize}

For a topological disk $U\subset \bbC$, $U\neq\bbC$ and $u\in U$ let $R_{(U,u)}:\bbD\to U$ denote
the inverse Riemann mapping with normalization $R_{(U,u)}(0)=u$, $R'_{(U,u)}(0)>0$.
By a classical result of Carath\'eodory, if neither $U$, nor $U_n$ are equal to $\bbC$, then the Carath\'edory convergence
$(U_n,u_n, f_n)\to (U,u,f)$ is equivalent to the simultaneous convergence $R_{(U_n,u_n)}\to R_{(U,u)}$ and $f_n\to f$ uniformly on compact subsets of $\bbD$ and $U$ respectively.

As before, for an Epstein map $g\in\mcE$, let $\tl J$ be the maximal (possibly unbounded) open interval, such that $g$ extends to a triple branched covering of a topological disk $G$ onto $\bbC_{\tl J}$. Notice that the pair $(g,\tl J)$ uniquely 
determines the topological disk $G$, hence we can associate to any $g\in{\mcE}$ the triple $(G,0,g)$.
In this way, the space $\mcE$ of all Epstein maps can be equipped with Carath\'eodory convergence.

Let us make a note of an important compactness property of $\mcE_{I,s}$:
\begin{lemma}[Lemma 2.12 \cite{Yamp-towers}]
\label{compactness}
For any $s\in (0,1)$ and an interval $I=[0,a]$, where $a\in\bbR$, $a\neq 0$, the set $\mcE_{I,s}$ is sequentially compact with respect to Carath\'eodory
convergence.
\end{lemma}

The space of all commuting pairs from the Epstein class can be equipped with Carath\'eodory convergence defined as Carath\'eodory convergence of each of the maps forming the commuting pair. Now Lemma~\ref{compactness} implies the following statement:

\begin{proposition}\label{E_s_compactness_proposition}
 For any $s\in (0,1)$, the set $\mcE_s$ of critical commuting pairs is sequentially compact with respect to Carath\'eodory convergence.
\end{proposition}
\begin{proof}
By definition, the set $\mcE_s$ consists of all critical commuting pairs $(\eta,\xi)$, such that $\eta\in\mcE_{[0,1],s}$ and $\xi\in\mcE_{I,s}$, for some interval $I$. According to Lemma~\ref{compactness}, the set $\mcE_{[0,1],s}$ is sequentially compact. Since $\xi(0)=1$, it follows from the definition of the class $\mcE_{I,s}$ that the number $|I|$ is $K(s)$-commensurable with $1$, where $K(s)$ is some constant that depends on $s$. This implies that the set of all maps $\xi$, such that $(\eta,\xi)\in\mcE_s$ is also sequentially compact. The rest of the proof follows easily.
\end{proof}

\subsection{Renormalization horseshoe}
The following theorem was proved by the second author in \cite{Yamp-towers}. It generalizes the results of de~Faria \cite{DeFar} and 
de~Faria \& de~Melo \cite{FM2} for pairs of a bounded type.
Consider the space of bi-infinite sequences 
$$\Sigma=\{(\ldots,r_{-k},\ldots,r_{-1},r_0,r_1,\ldots,r_k,\ldots)\text{ with }r_i\in\NN\}$$ 
equipped with the weak topology -- the coarsest topology in which the coordinate projections are continuous.
Denote by $\sigma:\Sigma\to\Sigma$ the right shift on this space:
$$
\sigma\colon (r_k)_{-\infty}^\infty\mapsto (r_{k+1})_{-\infty}^\infty.
$$
\begin{theorem}\label{Comm_pair_attractor_theorem}
There exists $s>0$ and an
$\mathcal R$-invariant set $\mathcal I\subset\mathbf E_s$ consisting of commuting pairs with irrational rotation numbers with the following properties. The action of $\mathcal R$ on $\mathcal I$ is topologically conjugate to the shift $\sigma\colon\Sigma\to\Sigma$:
$$
\iota\circ\mathcal R\circ \iota^{-1}=\sigma,
$$
and if $$\zeta=\iota^{-1}(\dots,r_{-k},\dots,r_{-1},r_0,r_1,\dots,r_k,\dots),$$ then 
$$\rho(\zeta)=[r_0,r_1,\dots,r_k,\dots].$$ The set $\mathcal I$ is sequentially pre-compact in the sense of Carath\'eodory convergence, its closure $\overline{\mathcal I}\subset\mathbf E$ is the attractor for renormalization. That is, for any analytic commuting pair $\zeta\in\mathbf E$ with irrational rotation number we have
$$
\mathcal R^n\zeta\to\overline{\mathcal I}
$$
in the sense of Carath\'eodory convergence. Moreover, for any two commuting pairs $\zeta,\zeta'\in\mathbf E$ with equal irrational rotation numbers $\rho(\zeta)=\rho(\zeta')$ we have
\begin{equation}\label{convergence_to_attractor_eq}
\dist(\mathcal R^n\zeta,\mathcal R^n\zeta')\to 0
\end{equation}
for the uniform distance between analytic extensions of commuting pairs on compact sets.
\end{theorem}

\begin{remark}\label{global_attractor_remark}
In~\cite{FM2} de Faria and de Melo extended the result of Theorem~\ref{Comm_pair_attractor_theorem} by proving the convergence~(\ref{convergence_to_attractor_eq}) for all (not necessarily Epstein) critical commuting pairs $\zeta$, $\zeta'$ with equal irrational rotation numbers $\rho(\zeta)=\rho(\zeta')$. 
\end{remark}

\noindent
Furthermore, in  \cite{Ya3,Ya4} it was shown that the above attractor is hyperbolic in a suitable sense.
 More specifically, in \cite{Ya3} the second author constructed a real-symmetric Banach manifold $\ccm$ of {\it critical cylinder maps}, such that the 
elements of $\mathbf W^\bbR$ are analytic critical circle maps. Furthermore he constructed a continuous 
projection $\pi_{\mathbf W}$ from an open neighborhood of $\mathcal I$ in $\cE$ to 
$\ccm$,
which, on the image, semi-conjugates $\cR^k$  for a fixed $k\in\NN$ to a real-symmetric analytic operator $\hren\colon\cV\to\mathbf W$ called the {\it cylinder renormalization} and defined on an open set $\cV$ in $\mathbf W$:
$$\pi_{\mathbf W}\circ \cR^k=\hren\circ\pi_{\mathbf W}. $$

The projection $\pi_{\mathbf W}$ also satisfies the following properties:
\begin{proposition}
\label{attr-injective}
\begin{enumerate}[(i)]
\item\label{conf_equv_part} If $\pi_{\mathbf W}(\zeta_1)=\pi_{\mathbf W}(\zeta_2)$, for some $\zeta_1,\zeta_2$ from a neighborhood of $\mathcal I$ in $\mcE$, then the Epstein pairs $\cR\zeta_1$ and $\cR\zeta_2$ are conformally conjugate in a neighborhood of their intervals of definition.
\item Conversely, suppose $\zeta_1$ and $\zeta_2$ are conformally conjugate in some neighborhoods of their intervals of definition and and $\rho(\zeta_1)=\rho(\zeta_2)\not\in\bbQ$. Then there exists $n\ge 0$, such that $\pi_{\mathbf W}(\cR^n\zeta_1)=\pi_{\mathbf W}(\cR^n\zeta_2).$

\item The rotation number of the critical circle map $\pi_{\mathbf W}(\zeta)$ is equal to the image of $\rho(\zeta)$ under the Gauss map.
\end{enumerate}
\end{proposition}

Finally, we have the following:
\begin{theorem}\cite{Ya4}\label{cren_hyperb_theorem}
Denote $\hat{\mathcal I}=\pi_{\mathbf W}(\mathcal I)$. Then the attractor $\hat{\mathcal I}\subset\cV$ is a uniformly hyperbolic set for $\hren$ with a complex one-dimensional unstable direction. 
\end{theorem}

It is important to note that the operator $\hren$ is not injective on $\cV$: the images under $\hren$ of two analytic maps which are conformally conjugate in a sufficiently large neighborhood (see Proposition~\ref{attr-injective} or~\cite{Ya3} for details) coincide. Uniform hyperbolicity of $\hat{\mathcal I}$ is understood in the sense of Definition~\ref{hyperb_def} below. 

Before proceeding to Definition~\ref{hyperb_def}, we need to introduce the following notation: let $U$ be an open set in a smooth Banach manifold $\mathbf M$. Consider a $C^\infty$-smooth (not necessarily injective) mapping $f\colon U\to\mathbf M$ that possesses a forward invariant set $\Lambda\subset U$. Let $\hat{\Lambda}$ be the subset of the direct product ${\Lambda}^\bbN$ that consists of all histories in ${\Lambda}$, i.e.
$$
\hat{{\Lambda}}=\{(x_i)_{i\le 0}\colon x_i\in {\Lambda}; f(x_i)=x_{i+1}\}.
$$
The metric on ${{\Lambda}}^\bbN$ is defined by
$$
d((x_i), (y_i))=\sum_{i\le 0}2^i\|x_i-y_i\|.
$$
The restriction $f|_{\Lambda}$ lifts to a homeomorphism $\hat f\colon\hat{{\Lambda}}\to\hat{{\Lambda}}$ defined by $\hat f((x_i))=(x_{i+1})$. The natural projection from $\hat{{\Lambda}}$ to ${\Lambda}$ sends $(x_i)$ to $x_0$, and the pullback under this projection of the tangent bundle $T_{{\Lambda}}\mathbf M$ is a tangent bundle on $\hat{{\Lambda}}$. We denote this tangent bundle by $T_{\hat{{\Lambda}}}$. Explicitly, an element of $T_{\hat{{\Lambda}}}$ is of the form $((x_i),v)$, where $(x_i)\in\hat{{\Lambda}}$ and $v\in T_{x_0}\mathbf M$. The differential $Df$ naturally lifts to a map $D\hat f\colon T_{\hat{{\Lambda}}}\to T_{\hat{{\Lambda}}}$. 

\begin{definition}\label{hyperb_def}
We say that the map $f$ is hyperbolic on ${\Lambda}$, if there exists a continuous splitting of the tangent bundle $T_{\hat{{\Lambda}}}=E^s\oplus E^u$, such that $D\hat f(E^{u/s})=E^{u/s}$, and there exist constants $c>0$ and $\lambda>1$, such that for all $n\ge 1$,
$$
\|D\hat f^n(v)\|\ge c\lambda^n\|v\|,\qquad v\in E^u,
$$
$$
\|D\hat f^n(v)\|\le c^{-1}\lambda^{-n}\|v\|,\qquad v\in E^s.
$$
\end{definition}
Note that the fiber of the stable bundle $E^s$ at a point $(x_i)\subset\hat{\Lambda}$ depends only on the point $x_0\subset\Lambda$.

We also notice that the set $\hat{\mathcal I}$ is not compact in $\mathbf W$ and its closure is not contained in $\cV$. In particular, this means that a priori, the existence of local stable/unstable manifolds is not guaranteed at every point of $\hat{\mathcal I}$. However, the following was shown in~\cite{Ya4}:

\begin{theorem}\label{hyperb_1_theorem}
For any $f\in\hat{\mathcal I}$ there exists a local stable manifold $W^s(f)\subset\cV$ passing through the point $f$. Moreover, $W^s(f)\cap\mathbf W^\bbR$ consists of all critical circle maps $g$ that are sufficiently close to $f$ and such that $\rho(g)=\rho(f)$.
\end{theorem}

\subsection{Holomorphic commuting pairs and complex {\it a priori} bounds}\label{holo_pairs_sec}
De~Faria \cite{DeFar} introduced holomorphic commuting pairs to apply 
Sullivan's Riemann surface laminations technique to the renormalization of 
critical circle maps. They are suitably defined holomorphic extensions 
of critical commuting pairs which replace Douady-Hubbard polynomial-like maps \cite{DH85}.
A critical commuting pair $\zeta=(\eta|_{I_\eta},\xi|_{I_\xi})$  extends to a {\it
holomorphic commuting
pair} $\cH$ if there exist 
three simply-connected $\RR$-symmetric domains $D, U, V\subset\bbC$ whose intersections with the real line are denoted by $I_U=U\cap\bbR$, $I_V=V\cap\bbR$, $I_D=D\cap\bbR$, and a simply connected $\bbR$-symmetric Jordan domain $\Delta$, such that

\begin{itemize}
\item $\bar D,\; \bar U,\; \bar V\subset \Delta$;
 $\bar U\cap \bar V=\{ 0\}\subset D$; the sets
  $U\setminus D$,  $V\setminus D$, $D\setminus U$, and $D\setminus V$ 
  are nonempty, connected, and simply-connected; 
  $I_\eta\subset I_U\cup\{0\}$, $I_\xi\subset I_V\cup\{0\}$;
\item  the sets $U\cap\bbH, V\cap\bbH, D\cap\bbH$ are Jordan domains;

\item the maps $\eta$ and $\xi$ have analytic extensions to $U$ and $V$ respectively, so that $\eta$ is a conformal diffeomorphism of $U$ onto $(\Delta\setminus\bbR)\cup\eta(I_U)$, and $\xi$ is a conformal diffeomorphism of $V$ onto $(\Delta\setminus\bbR)\cup\xi(I_V)$;

\item the maps $\eta\colon U\to\Delta$ and $\xi\colon V\to\Delta$ can be further extended to analytic maps $\hat{\eta}\colon U\cup D\to\Delta$ and $\hat{\xi}\colon V\cup D\to\Delta$, so that the map $\nu=\hat{\eta}\circ\hat{\xi}=\hat{\xi}\circ\hat{\eta}$ is defined in $D$ and is a three-fold branched covering of $D$ onto $(\Delta\setminus \bbR)\cup\nu(I_D)$ with a unique critical point at zero.
\end{itemize}
\begin{figure}[ht]
\centerline{\includegraphics[width=0.9\textwidth]{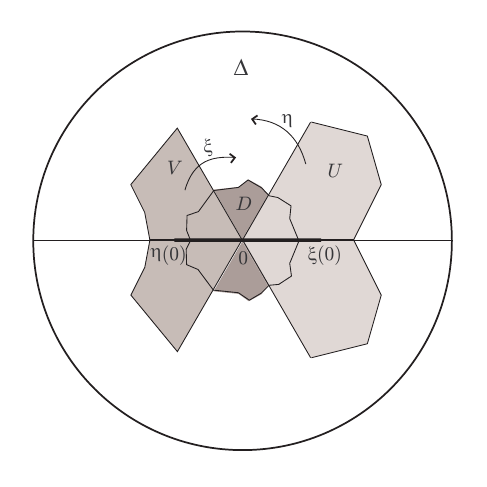}}
\caption{\label{fig:holpair}A holomorphic commuting pair.}
\end{figure}
\noindent
We shall identify a holomorphic pair $\cH$ with a triple of maps $\cH=(\eta,\xi, \nu)$, where $\eta\colon U\to\Delta$, $\xi\colon V\to\Delta$ and $\nu\colon D\to\Delta$ (c.f. Figure~\ref{fig:holpair}). We shall also call $\zeta$ the {\it commuting pair underlying $\cH$}, and write $\zeta\equiv \zeta_\cH$. When no confusion is possible, we will use the same letters $\eta$ and $\xi$ to denote both the maps of the commuting pair $\zeta_\cH$ and their analytic extensions to the corresponding domains $U$ and $V$.

The sets $\Omega_\cH=D\cup U\cup V$ and $\Delta\equiv\Delta_\cH$ will be called \textit{the domain} and \textit{the range} of a holomorphic pair $\cH$. We will sometimes write $\Omega$ instead of $\Omega_\cH$, when this does not cause any confusion. 

We can associate to a holomorphic pair $\cH$  a piecewise defined map $S_\cH\colon\Omega\to\Delta$:
\begin{equation*}
S_\cH(z)=\begin{cases}
\eta(z),&\text{ if } z\in U,\\
\xi(z),&\text{ if } z\in V,\\
\nu(z),&\text{ if } z\in\Omega\setminus(U\cup V).
\end{cases}
\end{equation*}
De~Faria \cite{DeFar} calls $S_\cH$ the {\it shadow} of the holomorphic pair $\cH$.

If $\zeta_\cH$ is renormalizable with height $r$, then the pre-renormalization of the corresponding holomorphic pair $\cH$ is defined in an obvious fashion,
as a holomorphic commuting pair $p\cR(\cH)=(\eta^r\circ \xi,\eta,\eta^r\circ\nu)$ with range $\Delta$ and the following domains $U',V',D'$:
\begin{itemize}
\item $V'=U$;
\item $U'= (\xi^{-1}\circ\eta^{-(r-1)})(\tl U)$, where $\tl U= (U\setminus\bbR)\cup(0,(\eta^{r-1}\circ\xi)(0))$;
\item $D'=\hat{\eta}^{-1}(U')$, where $\hat{\eta}$ is the analytic extension of $\eta$ to the domain $D\cup U$, and $\hat{\eta}^{-1}$ denotes the full preimage.
\end{itemize} 
The renormalization $\cR(\cH)$ is, as usual, the linear rescaling of $p\cR(\cH)$ sending $\eta(0)$ to $1$. Clearly, $\cH$ is renormalizable if and only if $\zeta_\cH$ is renormalizable and we have
$$
\zeta_{\cR(\cH)}=\cR(\zeta_\cH).
$$

We can naturally view a holomorphic pair $\cH$ as three triples
$$(U,\xi(0),\eta),\;(V,\eta(0),\xi),\;(D,0,\nu).$$
We say that a sequence of holomorphic pairs converges in the sense of Carath{\'e}odory convergence, if the corresponding triples do.
We denote the space of triples equipped with this notion of convergence by $\hol$.

We let the modulus of a holomorphic commuting pair $\cH$, which we denote by $\mod(\cH)$ to be the modulus of the largest annulus $A\subset \Delta$,
which separates $\bbC\setminus\Delta$ from $\overline\Omega$.

\begin{definition}\label{H_mu_def}
For $\mu\in(0,1)$ let $\hol(\mu)\subset\hol$ denote the space of holomorphic commuting pairs
${\cH}:\Omega_{{\cH}}\to \Delta_{\cH}$, with the following properties:
\begin{enumerate}[(i)]
\item\label{hol_mu_1} $\mod (\cH)\ge\mu$;
\item\label{hol_mu_2} {$|I_\eta|=1$, $|I_\xi|\ge\mu$} and $|\eta^{-1}(0)|\ge\mu$; 
\item $\dist(\eta(0),\partial V_\cH)/\diam V_\cH\ge\mu$ and $\dist(\xi(0),\partial U_\cH)/\diam U_\cH\ge\mu$;
\item\label{hol_mu_4} {the domains $\Delta_\cH$, $U_\cH\cap\bbH$, $V_\cH\cap\bbH$ and $D_\cH\cap\bbH$ are $(1/\mu)$-quasidisks.}
\item\label{hol_mu_5} $\diam(\Delta_{\cH})\le 1/\mu$;
\end{enumerate}
\end{definition}

\begin{lemma}[Lemma 2.17 \cite{Yamp-towers}]
\label{bounds compactness}
For each $\mu\in(0,1)$ the space $\hol(\mu)$ is sequentially compact.
\end{lemma}

\noindent
We say that a real commuting pair $\zeta=(\eta,\xi)$ with
an irrational rotation number has
{\it complex {\rm a priori} bounds}, if there exists $\mu>0$ such that all renormalizations of $\zeta=(\eta,\xi)$ extend to 
holomorphic commuting pairs in $\hol(\mu)$.
The existense of complex {\it a priori} bounds is a key analytic issue 
 of renormalization theory. 
{Before proceeding with the theorem on complex bounds for critical circle maps, we need to give the following definition:}

\begin{definition}\label{A_r_com_pair_def}
For a set $S\subset\bbC$, and $r>0$, we let $N_r(S)$ stand for the $r$-neighborhood of $S$ in $\bbC$. 
For each $r>0$ we introduce a class $\cA_r$ consisting of pairs   
$(\eta,\xi)$ such that the following holds:
\begin{itemize}
\item $\eta$, $\xi$ are real-symmetric analytic maps defined in the domains 
$$N_r([0,1])\text{ and }N_{r|\eta(0)|}([0,\eta(0)])$$
respectively, and continuous up to the boundary of the corresponding domains;
\item the pair $$\zeta\equiv (\eta|_{[0,1]},\xi|_{[0,\eta(0)]})$$
is a critical commuting pair.
\end{itemize}
\end{definition}

\noindent
For simplicity, if $\zeta$ is as above, we will write $\zeta\in\cA_r$. But it is important to note that viewing our critical commuting pair $\zeta$ as an element of $\cA_r$ imposes restrictions on where we are allowed to iterate it. Specifically, we view such $\zeta$ as undefined at any point $z\notin N_r([0,\xi(0)])\cup N_r([0,\eta(0)])$ (even if $\zeta$ can be analytically continued to $z$). Similarly, when we talk about iterates of $\zeta\in\cA_r$ we iterate the restrictions $\eta|_{N_r([0,\xi(0)])}$ and $\xi|_{N_r([0,\eta(0)]}$.
In particular, we say that the first and second elements of $p\cR\zeta=(\eta^r\circ\xi,\eta)$ are defined in the maximal domains, where the corresponding iterates are defined in the above sense.

For a domain $\Omega\subset\bbC$, we denote by $\mathfrak D(\Omega)$ the complex Banach space of analytic functions in $\Omega$, continuous up to the boundary of $\Omega$ and equipped with the sup norm. Consider the mapping $i\colon\cA_r\to\mathfrak D(N_r([0,1]))\times\mathfrak D(N_r([0,1]))$ defined by 
\begin{equation}\label{norm_inclusion_eq}
i(\eta,\xi)=(\eta, h\circ\xi\circ h^{-1}),\quad\text{ where }\quad h(z)=z/\eta(0).
\end{equation}
The map $i$ is injective, since $\eta(0)$ completely determins the rescaling $h$. Hence, this map induces a metric on $\cA_r$ from the direct product of sup-norms on $\mathfrak D(N_r([0,1]))\times\mathfrak D(N_r([0,1]))$. This metric on $\cA_r$ will be denoted by $\dist_r(\cdot,\cdot)$.

\begin{theorem}[{\bf Complex bounds}]
\label{complex bounds}There exists a universal constant $\mu>0$  such that
the following holds. For every positive real number $r>0$ and every pre-compact family $S\subset\mathcal A_r$ of critical commuting pairs, there exists $N=N(r, S)\in\bbN$ such that if $\zeta\in S$ is
a $2n$ times renormalizable commuting pair, where $n\geq N$, then
$p\cR^n\zeta$ restricts to a holomorphic commuting pair $\cH_n:\Omega_n\to\Delta_n$ with $\Delta_n\subset N_r(I_\eta)\cup N_r(I_\xi)$. 
Furthermore, the range $\Delta_n$ is a Euclidean disk,  and the appropriate affine rescaling of $\cH_n$ is in $\hol(\mu)$.
\end{theorem}

Theorem~\ref{complex bounds} was first proved in~\cite{Ya1} for critical commuting pairs from the Epstein class. This proof was later adapted by de~Faria and de~Melo \cite{FM2} 
to the case of non-Epstein critical commuting pairs.
We note that both in~\cite{Ya1} and in~\cite{FM2} this theorem was formulated for a single critical circle map (or commuting pair) with an irrational rotation number, however the uniformity of the estimates in a pre-compact family stated above is evident from the proofs.

\section{The functional spaces}
\label{sec:spaces}

\subsection{The spaces of generalized analytic critical cylinder maps $\mathfrak C_r$ and $\mathfrak C_r^\alpha$}

According to the standard terminology, an analytic critical circle map is a critical circle map which is analytic at every point of the circle including the critical point. This implies that the critical exponent is an odd integer not smaller than~$3$. In this section we generalize the class of analytic critical circle maps to include maps with other critical exponents.

For a fixed positive real number $0<r<1/2$, let $A_r\subset\bbC$ be the $r$-neighborhood of the interval $[0,1]$ with the points $1$ and $0$ removed:
$$
A_r=\{z\in\bbC\mid\dist(z,[0,1])<r\}\setminus\{0,1\}.
$$
The universal cover $\tl A_r$ of $A_r$ can be identified with the space of all pairs $(z,\gamma)$, such that $z\in A_r$ and $\gamma$ is a homotopy class of paths in $A_r$ which begin at $\frac{1}{2}$ and end at $z$. 
We define the surface $U_r$ as the subset of the universal cover $\tl A_r$, consisting of all pairs $(z,\gamma)\in\tl A_r$, such that $\gamma$ has a representative that enters the disks $\bbD_r(0)$ or $\bbD_r(1)$ no more than once and 
lies either entirely to the left or entirely to the right from the vertical line $\Re z=1/2$. 
The surface $U_r$ inherits a complex analytic structure as well as the Euclidean distance from $A_r$ via the projection of $U_r$ onto the first coordinate. If $f$ is an analytic function on $U_r$, then $f$ can be viewed as a multiple-valued analytic function on $A_r$. We will write $f(z,\gamma)$ for the value of $f$ at $(z,\gamma)\in U_r$, and we will write $f(z)$ when we view $f$ as a multiple-valued function. If $\gamma$ has a representative that lies in $A_r\setminus((-\infty,0)\cup(1,+\infty))$, we will shorten the notation for $f(z,\gamma)$ by writing $f(z,\gamma)=f(z)$, provided that this does not cause any confusion. 
By $\overline U_r$ we denote the completion of $U_r$ with respect to the Euclidean distance.

\begin{definition}
For a positive real number $r>0$, let $\mathfrak A_r$ denote the 
space of all functions $f\colon\overline U_r\to\bbC$ that are analytic on $U_r$, continuous on $\overline U_r$, 
and satisfy the identity $f(0)=f(1)$. Equipped with the sup-norm, the space $\mathfrak A_r$ is a complex Banach space. 
\end{definition}

Let $\pi\colon\bbC\to\bbC\slash\bbZ$ be the natural projection. 
Every map $f\colon\bbR\slash\bbZ\to\bbC\slash\bbZ$ naturally lifts to a (non-uniquely defined) map $\tilde f\colon [0,1]\to\bbC$, such that 
\begin{equation}\label{proj_rel}
\pi\circ\tilde f = f\circ\pi.
\end{equation}

\begin{definition}\label{C_r_def}
Given a positive real number $r>0$, we define the set of generalized cylinder maps $\mathfrak C_r$ to be the space of all continuous maps $f\colon\bbR\slash\bbZ\to\bbC\slash\bbZ$, such that any lift of $f$ to a map of $[0,1]\to\bbC$ has an analytic extension $\tilde f\colon U_r\to\bbC$ that is continuous on $\overline U_r$, 
and $\tilde f(1)=\tilde f(0)+1$.
\end{definition}

Let $V_r\subset\bbC\slash\bbZ$ be an equatorial neighborhood defined by
\begin{equation}
\label{v-annulus}
V_r\equiv \{z\in\bbC\colon|\Im z|<r\}/\bbZ\subset\bbC\slash\bbZ.
\end{equation}
Let  $f\in\mathfrak C_r$, and let $\tl f:U_r\to\bbC$ be as in Definition~\ref{C_r_def}. 
Pushing $\tl f$ down by the covering map  $\pi:\bbC\to \bbC/\bbZ$ we obtain a multiple-valued 
map from $V_r$ to $\bbC/\bbZ$,
which restricts to $f$ on the circle. In what follows, we will identify $f$ with this particular multiple-valued analytic continuation.
Whenever we iterate the multivalued map $f$, we will specify the choice of the branches.

\begin{proposition}\label{C_r_Banach_prop}
For any positive real number $r>0$, the set $\mathfrak C_r$ has a structure of an affine complex Banach manifold modeled on $\mathfrak A_r$.
\end{proposition}
\begin{proof}
For any map $\tilde f\colon U_r\to\bbC$ which is a lift of a map $f\in\mathfrak C_r$ via the relation~(\ref{proj_rel}), and for any map $g\in\mathfrak C_r$ from a sufficiently small neighborhood of $f$, let $\tilde g\colon U_r\to\bbC$ be the unique lift of $g$ via the same relation~(\ref{proj_rel}), such that the mapping $g\mapsto\tilde g$ is continuous in the uniform metric and takes $f$ to $\tilde f$. Then for any such $\tilde f$ we can consider a local chart
$\sigma_{\tilde f}\colon(\mathfrak C_r, f)\to\mathfrak A_r$ defined in a small open neighborhood of $f$ in $\mathfrak C_r$ by the relation
$$
\sigma_{\tilde f}\colon g\mapsto\tilde g-\tilde f.
$$
It is obvious that the transition maps $\sigma_{\tilde f_1}\circ\sigma_{\tilde f_2}^{-1}$ are affine, hence $\mathfrak C_r$ is an affine complex Banach manifold modeled on $\mathfrak A_r$.
\end{proof}

We denote $\mathfrak C_r^\bbR$ the real slice of the Banach manifold $\mathfrak C_r$. It consists of all
 maps from $\mathfrak C_r$ which map the circle $\bbR\slash\bbZ$ to itself.

\begin{definition}
Let $U\subset\bbC$ be a simply connected neighborhood of a point $z_0$, and consider a possibly multiple-valued analytic map $f\colon U\setminus \{z_0\}\to \bbC$. We say that $f$ has a critical exponent $\alpha\in\bbC$ at $z_0$, if in a neighborhood of $z_0$ the map $f$ can be represented as 
$$
f(z)=\psi((\phi(z))^\alpha),
$$
where $\phi$, $\psi$ are locally conformal maps, $\phi(z_0)=0$, and $z\mapsto z^\alpha$ is an appropriate branch of the power map.
\end{definition}

\begin{definition} 
Given a positive real number $r>0$ and a complex number $\alpha\in\bbC$, $\alpha\neq 1$,
we define the set of generalized critical cylinder maps $\mathfrak C_r^\alpha\subset\mathfrak C_r$ to be the set of all maps $f\in\mathfrak C_r$, whose lift to the interval $[0,1]$ has an analytic extension $\tilde f\colon U_r\to\bbC\slash\bbZ$ with critical  exponent $\alpha$ at $0$ and $1$.
\end{definition}

\begin{definition}\label{C_hat3_def}
When $\alpha=3$ (or any other odd integer greater than $3$), we define $\hat{\mathfrak C}_r^\alpha\subset\mathfrak C_r^\alpha$ to be the proper subset of $\mathfrak C_r^\alpha$ that consists of all maps $f\in\mathfrak C_r^\alpha$ which are analytic in $V_r$. 
In other words, all branches of the map $f$ in the disk of radius $r$ around zero coincide. 
\end{definition}
\begin{remark}
For any $\alpha\in 2\bbN+1$, the set $\hat{\mathfrak C}_r^\alpha$ is a submanifold of $\mathfrak C_r$ (c.f.~\cite{Ya3}).
\end{remark}

\subsection{The spaces of critical triples $\mathfrak P_{U,t,h}$ and $\mathfrak P_{U,t,h}^\alpha$}
In what follows, we will find it convenient when perturbing an analytic critical circle map $f$ to a map with a critical exponent $\alpha\neq 2n+1$ to first decompose $f$ into the form (\ref{sing1}) and then perturb each of the terms in the decomposition separately. To
lay the ground for this discussion, we introduce a suitable space of decompositions (\ref{sing1}) below. We note that the main difficulty here is to define these decompositions globally in a neighborhood of the circle (rather than just in a neighborhood of the critical point), so that the space of decompositions is invariant under renormalizations.

Given a positive real number $\alpha>1$, the function
$$
p_{\alpha+}\colon\bbC\setminus\bbR^-\to\bbC
$$
is defined as the branch of the map $z\mapsto z^\alpha$ which maps positive reals to positive reals. Similarly we define the function
$$
p_{\alpha-}\colon\bbC\setminus\bbR^+\to\bbC
$$
as the branch of the map $z\mapsto -(-z)^\alpha$ which maps negative reals to negative reals.
\begin{remark}
Functions $p_{\alpha+}$ and $p_{\alpha-}$ can be defined for any $\alpha\in\bbC$ by means of analytic continuation in $\alpha$-coordinate.
\end{remark}

Given a positive real number $t>0$, let $I_t^+$ denote the straight line segment connecting the points $\frac{1}{2}+it$ and $\frac{1}{2}-it$. Similarly, let $I_t^-$ denote the straight line segment connecting the points $-\frac{1}{2}+it$ and $-\frac{1}{2}-it$.

Consider a simply connected domain $U$ which contains both $I_t^+$ and $I_t^-$ as well as the interval $[-\frac{1}{2},\frac{1}{2}]$. 
For a complex number $\alpha\neq 0$ and a univalent analytic function $\phi\colon U\to\bbC$, let us assume 
that $\phi(I_t^+)$ and $\phi(I_t^-)$ lie in the domains of the functions $p_{\alpha+}$ and $p_{\alpha-}$ respectively. 
We define a closed curve $\gamma_{\phi,\alpha,t}\subset\bbC$ 
which consists of four pieces: $p_{\alpha+}(\phi(I_t^+))$, $p_{\alpha-}(\phi(I_t^-))$ and the two straight line segments $\ell_1$, $\ell_2$ connecting 
$p_{\alpha+}(\phi(\frac{1}{2}+it))$ with $p_{\alpha-}(\phi(-\frac{1}{2}+it))$ and $p_{\alpha+}(\phi(\frac{1}{2}-it))$ with $p_{\alpha-}(\phi(-\frac{1}{2}-it))$ 
respectively.

Let us make the following simple observation:
\begin{proposition}
\label{prop:quasidisk}
The curve $\gamma_{\phi,\alpha,t}$ is simple, closed and bounds a quasidisk if and only if the following four conditions hold:
\begin{itemize}
\item the curves $p_{\alpha+}(\phi(I_t^+))$, $p_{\alpha-}(\phi(I_t^-))$ are simple and disjoint;
\item the line segments $\ell_1$ and $\ell_2$ are disjoint;
\item the segments $\ell_1$, $\ell_2$ do not instersect with the curves $p_{\alpha+}(\phi(I_t^+))$, $p_{\alpha-}(\phi(I_t^-))$ except at the end points, and
\item the curves $p_{\alpha+}(\phi(I_t^+))$, $p_{\alpha-}(\phi(I_t^-))$ are not tangent to $\ell_1$, $\ell_2$ at their intersection points.
\end{itemize}

\end{proposition}
\begin{definition}
Suppose $\gamma_{\phi,\alpha,t}$ is simple, closed and bounds a quasidisk. Then we denote this quasidisk by $Q_{\phi,\alpha,t}$ (see Figure \ref{pic-quasidisk}).
\end{definition}

\begin{figure}[ht]
\centerline{\includegraphics[width=1.1\textwidth]{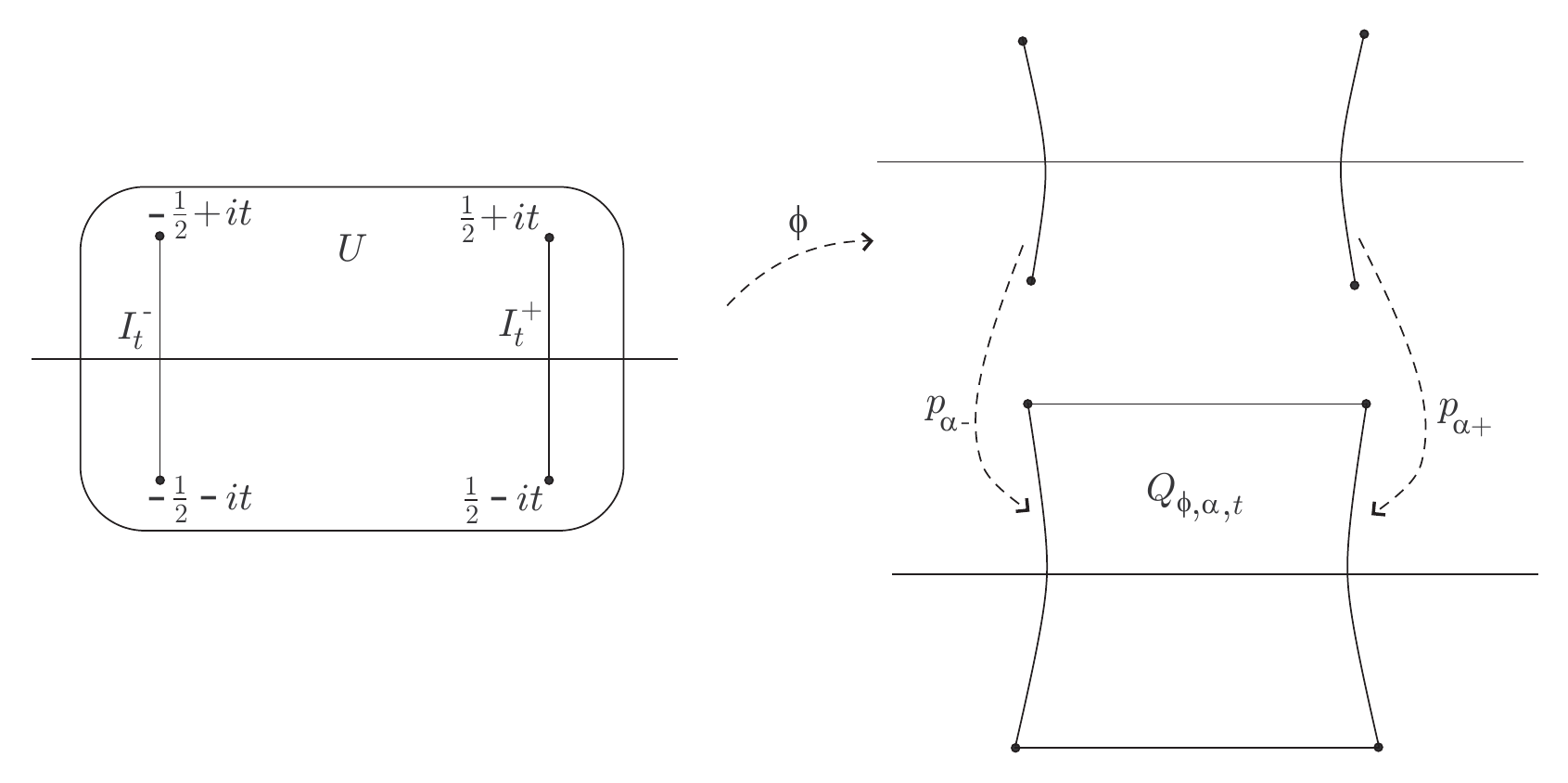}}
\caption{\label{pic-quasidisk}A domain $Q_{\phi,\alpha,t}$.}
\end{figure}

\begin{definition}\label{B_def}
For a positive real number $t>0$, a complex number $\alpha\neq 0$ and a Jordan domain $U\subset\bbC$ that contains the set $I_t^+\cup I_t^-\cup [-\frac{1}{2},\frac{1}{2}]$, by $\mathfrak B_{U,t}^\alpha$ we denote the set of all analytic maps $\phi\colon U\to\bbC$ that satisfy the following properties:

\begin{enumerate}[(i)]
 \item the map $\phi$ continuously extends to the closure $\overline U$, and $\phi(0)=0$, $\phi'(0)=1$;
 \item the map $\phi$ is univalent in a neighborhood of $I_t^+\cup I_t^-\cup[-\frac{1}{2},\frac{1}{2}]$;
 \item\label{B_def_item3} the quasidisk $Q_{\phi,\alpha,t}$ as well as the curves $p_{\alpha-}(\phi((-\frac{1}{2},0]))$ and $p_{\alpha+}(\phi([0,\frac{1}{2})))$ are defined, and $p_{\alpha-}(\phi((-\frac{1}{2},0])), p_{\alpha+}(\phi([0,\frac{1}{2}))) \subset Q_{\phi,\alpha,t}$;
 \item\label{B_def_item4} every straight line through the origin intersects the boundary of the quasidisk $Q_{\phi,\alpha,t}$ transversally at exactly two points.
\end{enumerate}
\end{definition}

Let $U\subset\bbC$ be a simply connected domain containing the origin. By $\mathfrak A_U$ we denote the space of all bounded analytic functions $\phi\colon U\to\bbC$ that are continuous up to the boundary, and satisfy the properties $\phi(0)=0$, and $\phi'(0)=1$. 
We note that $\mathfrak A_U$ equipped with the sup-norm forms an affine complex Banach space
 (the space $\mathfrak A_U-\mathrm{id}$ is a complex Banach space). 

The next statement follows from obvious continuity considerations:
\begin{proposition}\label{Banach_B_UAlphaSRMu}
For every $t$, $\alpha$ and $U$ as in Definition~\ref{B_def}, the set $\mathfrak B_{U,t}^\alpha$ is an open subset of the affine Banach space $\mathfrak A_U$, hence it is a complex Banach manifold, modeled on $\mathfrak A_U$.
\end{proposition}
\begin{proof}
It follows from Proposition~\ref{prop:quasidisk} that the quasidisk $Q_{\phi,\alpha,t}$ exists for all maps $\phi$ from a certain open subset of $\mathfrak A_U$. All other properties of Definition~\ref{B_def} also define open subsets in $\mathfrak A_U$. The set $\mathfrak B_{U,t}^\alpha$ is the intersection of these open subsets, hence it is also open.
\end{proof}

\begin{definition}
We define the set $\mathfrak B_{U,t}$ as the union of all pairs $(\alpha, \phi)$, such that $\alpha\in\bbC$, $\alpha\neq 0$, and $\phi\in \mathfrak B_{U,t}^\alpha$.
\end{definition}

\noindent
Similarly to Proposition~\ref{Banach_B_UAlphaSRMu}, considerations of continuity yield the following proposition:
\begin{proposition}\label{Banach_B_U_lemma}
The set $\mathfrak B_{U,t}$ is an open subset of the direct product of Banach spaces $\bbC\times\mathfrak A_U$, hence it is a complex Banach manifold, modeled on $\bbC\times\mathfrak A_U$.
\end{proposition}

For a positive real number $h>0$, we denote  by $\mathfrak U_h$ the set of all $1$-periodic functions that are defined and analytic in the strip $\{z\in\bbC\colon |\Im z|<h\}$ and continuous up to the boundary of this strip. The set $\mathfrak U_h$ equipped with the uniform norm is a complex Banach space.
As before, for a positive real number $h$, we  denote by $V_h$ the $h$-neighborhood of the equatorial circle in $\bbC\slash\bbZ$ defined in~(\ref{v-annulus}).

\begin{definition}\label{D_h_def}
For a fixed real number $h>0$, we denote  by $\mathfrak D_h$ the set of all analytic functions $\psi\colon V_h\to\bbC\slash\bbZ$ such that the following properties hold:
\begin{itemize}
\item $\psi$ continuously extends to the boundary of $V_h$;
\item the image of the unit circle under $\psi$ is homotopic to the equatorial circle  $\bbT\subset \bbC\slash\bbZ$.
\item $\psi$ is univalent in a neighborhood of $\bbT\subset\bbC\slash\bbZ$. 
\end{itemize}
\end{definition}

\noindent
We immediately notice that 
\begin{lemma}\label{D_h_Banach_lemma}
Each set $\mathfrak D_h$ has a structure of a  complex Banach manifold modeled on $\mathfrak U_h$.
\end{lemma}
\begin{proof}
For a map $g\in\mathfrak U_h$, the map 
$$\tilde f(z)=g(z)+z$$
is a lift of some $f\in\mathfrak D_h$ via the relation~(\ref{proj_rel}). The constructed correspondence $g\mapsto f$ is an affine covering map from $\mathfrak U_h$ onto the space of all analytic functions $\psi\colon V_h\to\bbC\slash\bbZ$ that satisfy the first two properties of Definition~\ref{D_h_def}. The last property of Definition~\ref{D_h_def} defines $\mathfrak D_h$ as an open subset of the former space, hence $\mathfrak D_h$ is an affine complex Banach manifold modeled on $\mathfrak U_h$.
\end{proof}

\begin{definition}
Given a neighborhood of the origin $U\subset\bbC$ with Jordan boundary 
and positive real numbers $t,h>0$, by $\mathfrak P_{U,t,h}$ we denote  $$\mathfrak P_{U,t,h}=\mathfrak B_{U,t}\times\mathfrak D_h.$$
\end{definition}

\noindent
In other words, $\mathfrak P_{U,t,h}$ is the set of all triples $(\alpha,\phi,\psi)$, such that $\alpha\in\bbC$, $\alpha\neq 0$, $\phi\in \mathfrak B_{U,t}^\alpha$ and $\psi\in\mathfrak D_h$.

\begin{definition}\label{P_U_alpha_s_r_mu_def}
For a fixed $\alpha\in\bbC$, such that $\alpha\neq 0$, we denote by $\mathfrak P_{U,t,h}^\alpha\subset \mathfrak P_{U,t,h}$ the set
of all elements from $\mathfrak P_{U,t,h}$ whose first coordinate is equal to $\alpha$. 
The set $\mathfrak P_{U,t,h}^\alpha$ can be identified with the direct product $\mathfrak B_{U,t}^\alpha\times \mathfrak D_h$.
\end{definition}

Finally, let us formulate an immediate corollary from Proposition~\ref{Banach_B_UAlphaSRMu}, Proposition~\ref{Banach_B_U_lemma} and Lemma~\ref{D_h_Banach_lemma}:
\begin{lemma}
The space $\mathfrak P_{U,t,h}$ has the structure of a complex Banach manifold. For each $\alpha\in\bbC$, $\alpha\neq 0$, the space $\mathfrak P_{U,t,h}^\alpha$ is a complex Banach submanifold of $\mathfrak P_{U,t,h}$ of complex codimension~$1$.
\end{lemma}

\subsection{Generalized critical cylinder maps in $\mathfrak P_{U,t,h}$ and $\mathfrak P_{U,t,h}^\alpha$}

In this subsection we construct a correspondence between elements of $\mathfrak P_{U,t,h}$ and generalized analytic critical cylinder maps, and
 show that this correspondence is an analytic map (c.f. Lemma~\ref{f_tau_analyt_lemma}).

Let $U\subset\bbC$ be a sufficiently large neighborhood of the origin and let $t>0$ be a sufficiently small real number such that the space $\mathfrak B_{U,t}$ is defined and non-empty. For every $(\alpha,\phi)\in\mathfrak B_{U,t}$, we notice that the function
$$
\nu(z)=p_{\alpha+}(\phi(\phi^{-1} (p_{\alpha-}^{-1}(z))+1))
$$
is analytic and univalent in a neighborhood of the curve $p_{\alpha-}(\phi(I_t^-))$. (Here $p_{\alpha-}^{-1}$ denotes the inverse branch, for which $\phi^{-1}\circ p_{\alpha-}^{-1}$ maps $p_{\alpha-}(\phi(I_t^-))$ to $I_t^-$). Because of that the set $Q_{\phi,\alpha,t}\cup p_{\alpha-}(\phi(I_t^-))\cup p_{\alpha+}(\phi(I_t^+))$ factored by the action of the map $\nu$, can be viewed as a Riemann surface
\begin{equation}\label{A_phi_alpha_s_def}
A_{\phi,\alpha,t}=Q_{\phi,\alpha,t}\cup p_{\alpha-}(\phi(I_t^-))\cup p_{\alpha+}(\phi(I_t^+))\slash\nu
\end{equation}
conformally isomorphic to an annulus.

\begin{lemma}\label{pi_phi_lemma}
Let $U\subset\bbC$ and $t>0$ be such that the space $\mathfrak B_{U,t}$ is defined and contains an element $(\alpha_0,\phi_0)\in\mathfrak B_{U,t}$, where $\alpha_0\in\bbR$ and $\phi_0$ is real-symmetric. Then there exists a family of conformal maps $\pi_{\phi,\alpha,t}\colon Q_{\phi,\alpha,t}\to\bbC\slash\bbZ$ parameterized by $(\alpha,\phi)\in\mathfrak B_{U,t}$, such that for every $(\alpha,\phi)\in\mathfrak B_{U,t}$, the map $\pi_{\phi,\alpha,t}$ projects to a conformal map of the annulus $A_{\phi,\alpha,t}$, $\pi_{\phi,\alpha,t}(0)=0$ and $\pi_{\phi,\alpha,t}$ is real-symmetric, if $\alpha\in\bbR$ and $\phi$ is real-symmetric. Moreover, for every $z\in Q_{\phi,\alpha,t}$, the dependence of $\pi_{\phi,\alpha,t}(z)$ on $(\alpha,\phi)$ is analytic.
\end{lemma}
\begin{proof}

For every $(\alpha,\phi)\in\mathfrak B_{U,t}$, two sides of the quasidisk $Q_{\phi,\alpha,t}$ -- the curves $p_{\alpha+}(\phi(I_t^+))$ and $p_{\alpha-}(\phi(I_t^-))$, are injective holomorphic images of fixed intervals $I_t^+$ and $I_t^-$, hence, there exists a holomorphic motion
$$
\chi_{\alpha,\phi}\colon p_{\alpha_0+}(\phi_0(I_t^+))\cup p_{\alpha_0-}(\phi_0(I_t^-))\to p_{\alpha+}(\phi(I_t^+))\cup p_{\alpha-}(\phi(I_t^-))
$$
over $(\alpha,\phi)\in \mathfrak B_{U,t}$, given by the formula
\begin{equation}\label{holo_motion_chi}
\chi_{\alpha,\phi}(z)=\begin{cases}
  (p_{\alpha+}\circ\phi\circ\phi_0^{-1}\circ p_{\alpha_0+}^{-1})(z) & \text{if } z\in 
     p_{\alpha_0+}(\phi_0(I_t^+)) \\
  (p_{\alpha-}\circ\phi\circ\phi_0^{-1}\circ p_{\alpha_0-}^{-1})(z) & \text{if } z\in
     p_{\alpha_0-}(\phi_0(I_t^-)). 
\end{cases}
\end{equation}

Since the other two sides of the quasidisk $Q_{\phi,\alpha,t}$ are straight line segments whose endpoints move holomorphically, the holomorphic motion~(\ref{holo_motion_chi}) extends to a holomorphic motion $\chi_{\alpha,\phi}\colon \partial Q_{\phi_0,\alpha_0,t}\to \partial Q_{\phi,\alpha,t}$. Finally, it follows from property~\ref{B_def_item3} of Definition~\ref{B_def} that zero belongs to $Q_{\phi,\alpha,t}$, and according to property~\ref{B_def_item4} from Definition~\ref{B_def}, the quasidisk $Q_{\phi,\alpha,t}$ is a star domain, hence the holomorphic motion $\chi_{\alpha,\phi}$ of the boundary of the quasidisk extends to a holomorphic motion 
\begin{equation}\label{holo_motion}
\chi_{\alpha,\phi}\colon \overline Q_{\phi_0,\alpha_0,t}\to \overline Q_{\phi,\alpha,t},\quad\text{with}\quad\chi_{\alpha,\phi}(0)=0, \forall \alpha,\phi
\end{equation}
of the closure of the quasidisk over the whole space $\mathfrak B_{U,t}$.

We construct the family of maps $\pi_{\phi,\alpha,t}$ in the following way: for $(\alpha,\phi)=(\alpha_0,\phi_0)$ we put $\pi_{\phi_0,\alpha_0,t}$ to be the unique conformal map with the property that $\pi_{\phi_0,\alpha_0,t}(0)=0$ and $\pi_{\phi_0,\alpha_0,t}$ induces a conformal diffeomorphism between $A_{\phi_0,\alpha_0,t}$ and an annulus $A_0\subset\bbC\slash\bbZ$ whose boundary consists of two circles parallel to the equator.

For an arbitrary $(\alpha,\phi)\in\mathfrak B_{U,t}$ we make the following construction: by the Lambda-lemma \cite{MSS,Lyubich-stability}, for each $(\alpha,\phi)\in\mathfrak B_{U,t}$, the map $\chi_{\alpha,\phi}$ is quasiconformal. Pulling back the standard conformal structure on $Q_{\phi,\alpha,t}$ by the map $\chi_{\alpha,\phi}\circ\pi_{\phi_0,\alpha_0,t}^{-1}$, we obtain a conformal structure $\mu_{\alpha,\phi}$ on $A_0$. We extend the conformal structure $\mu_{\alpha,\phi}$ beyond the annulus $A_0$ by the standard conformal structure. We obtain a conformal structure $\mu_{\alpha,\phi}$ on the whole cylinder $\bbC\slash\bbZ$, and the dependence of $\mu_{\alpha,\phi}$ on $(\alpha,\phi)$ is analytic. By the Measurable Riemann Mapping Theorem, there exists a unique map $f_{\alpha,\phi}\colon\bbC\slash\bbZ\to\bbC\slash\bbZ$ that straightens the conformal structure $\mu_{\alpha,\phi}$ and fixes the point $0$ and the two boundary points at infinity. Moreover, $f_{\alpha,\phi}$ depends analytically on $(\alpha,\phi)\in\mathfrak B_{U,t}$. We define the map $\pi_{\phi,\alpha,t}$ as the composition
$$
\pi_{\phi,\alpha,t}=f_{\alpha,\phi}\circ\pi_{\phi_0,\alpha_0,t}\circ\chi_{\alpha,\phi}^{-1}.
$$

It follows from the construction that if $\alpha\in\bbR$ and $\phi$ is real-symmetric, then $\pi_{\phi,\alpha,t}$ is also real-symmetric. Analytic dependence of $\pi_{\phi,\alpha,t}$ on $\phi$ and $\alpha$ follows from analytic dependence of $f_{\alpha,\phi}$ and $\chi_{\alpha,\phi}$ on the parameters (c.f.  Remark on page 345 of \cite{Lyubich-Hairiness}).
\end{proof}

\begin{definition}\label{U_r_pm_def}
For a positive real number $r>0$, let $U_r^+\subset U_r$ be the surface that consists of all pairs $(z,\gamma)\in U_r$, such that $z\in A_r\cap\{\Re z<0.5\}$. 
Similarly, let $U_r^-$ be the surface that consists of all pairs $(z,\gamma)$, such that $(z+1,\gamma)\in U_r$ and $z+1\in A_r\cap\{\Re z>0.5\}$.
\end{definition}

Let $U\subset\bbC$ be a neighborhood of the interval $[-1/2,1/2]$.
Assume that the analytic map $\phi\colon U\to\bbC$ is such that $\phi(0)=0$ and $\phi'(0)$ is a non-negative and nonzero complex number. Then for $\alpha\in\bbC$, the composition $p_{\alpha+}\circ\phi$ is defined on some interval $(0,\varepsilon)\subset \bbR$ to the right of the origin. Similarly, the composition $p_{\alpha-}\circ\phi$ is defined on some interval $(-\varepsilon,0)\subset \bbR$ to the left of the origin. Let a positive real number $r>0$ be sufficiently small, so that the projections of $U_r^+,U_r^-$ onto the first coordinate are contained in $U$. Then the first composition can be extended to an analytic map on $U_r^+$ and the second composition can be extended to an analytic map on $U_r^-$. Let us denote the first map by $p_{\phi,\alpha+}\colon U_r^+\to\bbC$, and the second map by $p_{\phi,\alpha-}\colon U_r^-\to\bbC$.

Now we fix a neighborhood $U\subset\bbC$ and a real number $t>0$ that satisfy Lemma~\ref{pi_phi_lemma}. We also fix a family of maps $\pi_{\phi,\alpha,t}$ whose existence is guaranteed by Lemma~\ref{pi_phi_lemma}. For a positive real number $h>0$, we consider the space $\mathfrak P_{U,t,h}$, and for an element $\tau=(\alpha,\phi,\psi)\in\mathfrak P_{U,t,h}$ we define the map $g_\tau\colon [-\frac{1}{2},\frac{1}{2}]\to\bbC\slash\bbZ$ according to the formula
\begin{equation}\label{g_tau_def}
g_\tau(z)=
\begin{cases}
  \psi(\pi_{\phi,\alpha,t}(p_{\phi,\alpha-}(z))) 
  & \text{if } z\in [-\frac{1}{2},\,0] \\
  \psi(\pi_{\phi,\alpha,t}(p_{\phi,\alpha+}(z))) 
  & \text{if } z\in [0,\,\frac{1}{2}],
\end{cases}
\end{equation}
provided that the above compositions are defined.

It follows from the construction of the space $\mathfrak P_{U,t,h}$ that the map $g_\tau$ projects to the continuous map of the circle $f_\tau\colon\bbR\slash\bbZ\to\bbC\slash\bbZ$ defined by the relation
\begin{equation}\label{f_tau_def}
f_\tau(\pi(z))=g_\tau(z).
\end{equation}

If for some positive real number $r>0$, the first and the second compositions in~(\ref{g_tau_def}) are defined in the domains $U_r^-$ and $U_r^+$ respectively, then according to Definition~\ref{C_r_def}, this implies that $f_\tau$ is a generalized critical cylinder map from $\mathfrak C_r$. It also follows from~(\ref{g_tau_def}) that the critical exponent of $f_\tau$ at zero is equal to $\alpha$, hence $f_\tau\in\mathfrak C_r^\alpha$.

\begin{lemma}\label{f_tau_analyt_lemma}
Assume that $U\subset\bbC$ and $t>0$ satisfy the conditions of Lemma~\ref{pi_phi_lemma}, and let $\pi_{\phi,\alpha,t}$ be a fixed family of maps from Lemma~\ref{pi_phi_lemma}. Assume that for some $r_0>0$, $h>0$ and $\tau_0\in\mathfrak P_{U,t,h}$, the first and the second compositions in~(\ref{g_tau_def}) are defined in the domains $U_{r_0}^-$ and $U_{r_0}^+$ respectively. Then for all positive real numbers $r<r_0$, there exists a neighborhood $\mathcal U_r\subset\mathfrak P_{U,t,h}$ of $\tau_0$, such that the correspondence $\tau\mapsto f_\tau$ is an analytic map from $\mathcal U_r$ to $\mathfrak C_r$.
\end{lemma}
\begin{proof}
It follows from continuous dependence of $f_\tau(z)$ on the parameters that since $f_{\tau_0}\in\mathfrak C_{r_0}$, then for all $r<r_0$ there exists a neighborhood of $\tau_0$, such that for every $\tau$ from that neighborhood, the map $f_\tau$ belongs to $\mathfrak C_r$. We denote this neighborhood by $\mathcal U_r$. Analytic dependence of $f_\tau$ on the parameters follows from analytic dependence of the map $\pi_{\phi,\alpha,t}$ on $\alpha$ and $\phi$ (c.f. Lemma~\ref{pi_phi_lemma}).
\end{proof}

\section{Fundamental crescents}
\label{sec:cres}

\subsection{Fundamental crescent of an analytic map}
The following discussion mirrors \cite{Ya3}. Let $\Omega\subset\bbC$ be a domain and let $h\colon\Omega\to\bbC$ be an analytic map.

\begin{definition}
\label{def-crescent1}
A simply connected Jordan domain $C\subset \bbC$ 
is called a \textit{crescent domain} for $h$, if the boundary of $C$ is the union of two piecewise smooth simple curves $l\subset \Omega$ and $h(l)\subset h(\Omega)$,
the endpoints $a$, $b$ of $l$ are two distinct repelling fixed points of $h$, and $l\cap h(l)=\{a,b\}$.
\end{definition}

\begin{definition}\label{Fattening_def}
A domain $C^o\subset\bbC$ is called a \textit{fattening} of a crescent domain $C$ as above, if $C^o$ is the union
$$
C^o=C\cup W_l\cup h(W_l),
$$
where $W_l\Subset\Omega$ is a neighborhood on which the map $h$ is univalent, the arc $l$ is contained in $\overline W_l$ so that $l\cap\partial W_l=\{a,b\}$, and $\overline W_l\cap h(\overline W_l)=\{a,b\}$.
\end{definition}

\noindent
Finally,
\begin{definition}
A crescent domain $C$ as above is called a \textit{fundamental crescent} {for $h$}, if it has a fattening $C^o$ and the quotient of $C^o$ by the action of the map $h$ is a Riemann surface conformally isomorphic to the cylinder $\bbC\slash\bbZ$ (see Figure \ref{fig-crescent}). We note that this Riemann surface is independent of the choice of a fattening $C^o$. 
\end{definition}

\noindent
If $C$ is a fundamental crescent for a map $h$, then the conformal isomorphism between the Riemann surface $C^o/h$ and the cylinder $\bbC\slash\bbZ$ is 
uniquely determined up to the post-composition with a conformal automorphism of the cylinder $\bbC\slash\bbZ$. Hence a lift of this isomorphism to the universal cover of the cylinder $\bbC\slash\bbZ$ is uniquely determined up to the post-composition with translations and multiplication by $-1$. 
This observation motivates the following definition:

\begin{definition}
Let $h\colon\Omega\to\bbC$ be an analytic map with a fundamental crescent $C$, and its fattening $C^o=C\cup W_l\cup h(W_l)$. For a fixed base point $\omega\in C$, we denote by
$$
\pi_{C^o}^\omega\colon C^o \to\bbC
$$
the unique conformal map which conjugates $h$ with the unit shift $z\mapsto z+1$ in $W_l$, and such that $\pi_{C^o}^\omega(\omega)=0$. We also denote by $\tilde\pi_{C^o}^\omega\colon C \to\bbC\slash\bbZ$ the composition
\begin{equation}\label{tilde_pi_to_pi_eq}
\tilde\pi_{C^o}^\omega=\pi\circ\pi_{C^o}^\omega .
\end{equation}
Finally, by $\pi_C^\omega$ and $\tilde{\pi}_C^\omega$ we denote the restrictions of $\pi_{C^o}^\omega$ and $\tilde{\pi}_{C^o}^\omega$ respectively to the fundamental crescent $C$:
$$
\pi_C^\omega\equiv \pi_{C^o}^\omega|_C,\quad\text{ and }\quad \tilde{\pi}_C^\omega \equiv\tilde\pi_{C^o}^\omega|_C.
$$
\end{definition}

We note that 
the maps $\pi_C^\omega$ and $\tilde{\pi}_C^\omega$ depend only on the fundamental crescent $C$, and are independent of the choice of its fattening $C^o$.

\begin{remark}
Whenever the base point is chosen to be equal to zero, $\omega=0$, we will simplify the notation by writing $\pi_{C^o}$, $\tilde{\pi}_{C^o}$, $\pi_C$ and $\tilde{\pi}_C$ instead of $\pi_{C^o}^\omega$, $\tilde{\pi}_{C^o}^\omega$, $\pi_C^\omega$ and $\tilde{\pi}_C^\omega$ respectively.
\end{remark}

For a domain $\Omega\subset\bbC$, we denote by $\mathfrak D(\Omega)$ the Banach space of analytic functions in $\Omega$, continuous up to the boundary of $\Omega$ and equipped with the sup norm. The following key lemma mirrors the statements  proved in \cite{Ya3}.

\begin{lemma}\label{Nearby_crescent_lemma}
Let $C$ be a fundamental crescent of a map $h\in\mathfrak D(\Omega)$ with the boundary consisting of the curves $l$ and $h(l)$ whose common end points are the points $a,b\in\Omega$. Let $C^o=C\cup W_l\cup h(W_l)$ be a fattening of $C$, such that some other fattening $\hat C^o$ of $C$ contains $C^o$ and $\partial C^o\cap\partial\hat C^o=\{a,b\}$, and let $\omega\in C$ be a fixed base point. Then there exists an open neighborhood $\mathcal W(h)\subset\mathfrak D(\Omega)$ of $h$, such that every map $g\in\mathcal W(h)$ has a fundamental crescent $C_g$ and a fattening $C_g^o$ of $C_g$ with the following properties:

(i) For $g=h$, we have $C_g=C$ and $C_g^o=C^o$, and the mappings
$$
g\mapsto C_g\quad \text{ and }\quad g\mapsto C_g^o
$$
are continuous with respect to the Hausdorff metric on the image;

(ii) If the map $h$, the fundamental crescent $C$ and its fattening $C^o$ are real symmetric, then $C_g$ and $C_g^o$ are also real-symmetric for all real-symmetric maps $g\in\mathcal W(h)$;

(iii) For all $g\in\mathcal W(h)$, the base point $\omega$ is contained in $C_g$. Moreover, for every $g_0\in\mathcal W(h)$ and $z\in C_{g_0}^o$, the number $\pi_{C_g^o}^\omega(z)$ is defined for all $g\in\mathcal W(h)$ that are sufficiently close to $g_0$, and the dependence 
$$
g\mapsto\pi_{C_g^o}^\omega(z)
$$
is locally analytic in $g\in\mathfrak D(\Omega)$.
\end{lemma}
\begin{proof}
We construct a neighborhood $\mathcal W(h)$ and the family of fattenings $C_g^o$ that move holomorphically over $g\in\mathcal W(h)$. 
More precisely, we construct a holomorphic motion of a fattening $C^o$ of the fundamental crescent
\begin{equation}\label{holo_motion_of_Co}
\chi_g\colon C^o\mapsto C_g^o
\end{equation}
over a neighborhood $\mathcal W(h)$, so that 
\begin{equation}\label{equivariance_rel}
\chi_g(h(z))=h(\chi_g(z)),
\end{equation}
whenever both sides are defined. Then $C_g$ is defined as 
$$
C_g=\chi_g(C).
$$
The construction is made in two steps.

The first step consists of constructing the holomorphic motion $\chi_g$ of $\overline{W_l\cup h(W_l)}$ over some neighborhood $\mathcal W_1(h)\subset\mathfrak D(\Omega)$, satisfying the equivariance relation~(\ref{equivariance_rel}). This construction is analogous to the construction made in Proposition~7.3 of~\cite{Ya3}. We refer the reader to the proof of that statement and omit further details. In the second step, applying the Theorem of Bers and Royden \cite{BR}, 
we extend the previously constructed holomorphic motion $\chi_g$ to a holomorphic
motion~(\ref{holo_motion_of_Co}) over a possibly smaller open neighborhood $\cW(h)$ with an additional property that
$$
\chi_g(\omega)=\omega.
$$

The holomorphic motion~(\ref{holo_motion_of_Co}) defines the family of fundamental crescents $C_g$ and their fattenings $C_g^o$ that continuously depend on the map $g\in\mathcal W(h)$, and hence satisfy property (i). It also follows from the construction that these families satisfies property~(ii) (c.f. Proposition~7.3 of~\cite{Ya3}).

To prove property (iii) for the constructed family of fundamental crescents, we use a version of the argument given in Proposition~7.5 of~\cite{Ya3}. 
Namely, by the Lambda-lemma~\cite{MSS} all maps $\chi_g$ are quasiconformal on $C^o$. Even though the map $\tilde\pi_{C^o}^\omega$ is not univalent on $C^o$, the equivariance relation~(\ref{equivariance_rel}) implies that the pullback of the standard conformal structure on $C_g^o$ by the map $\chi_g\circ(\tilde\pi_{C^o}^\omega)^{-1}$ is well defined. This pullback provides a conformal structure on the cylinder $\bbC\slash\bbZ$, and the dependence of this conformal structure on $g$ is analytic. By the Measurable Riemann Mapping Theorem, there exists a unique map $f_g\colon\bbC\slash\bbZ\to\bbC\slash\bbZ$ that straightens this conformal structure and fixes the point $0$ and the two boundary points at infinity. The composition
$
f_g\circ\tilde{\pi}_{C^o}^\omega\circ\chi_g^{-1}
$
induces a conformal isomorphism between the Riemann surfaces $C_g^o\slash g$ and $\bbC\slash\bbZ$, and $f_g\circ\tilde{\pi}_{C^o}^\omega\circ\chi_g^{-1}(\omega)=0$. The uniqueness of such an isomorphism implies that
\begin{equation}\label{tile_pi_eq}
\tilde{\pi}_{C_g^o}^\omega=f_g\circ\tilde{\pi}_{C^o}^\omega\circ\chi_g^{-1}.
\end{equation}
Finally, analytic dependence on the parameter in the Measurable Riemann Mapping Theorem implies that $f_g$ depends analytically on $g\in\mathcal W(h)$, hence~(\ref{tile_pi_eq}) implies that for a fixed $z$, the point $\tilde{\pi}_{C_g^o}^\omega(z)$ depends analytically on $g$. Now analytic dependence of $\pi_{C_g^o}^\omega(z)$ on $g$ follows from~(\ref{tilde_pi_to_pi_eq}).
\end{proof}

\subsection{Fundamental crescents of maps in $\kC_r$}
We let $V_r$ be as in (\ref{v-annulus}).

\begin{definition}\label{b_arc_def}
A simple piecewise smooth arc $l\subset V_r$ will be called a \textit{separating arc}, if $l$ is an image of an arc $\tilde l\subset A_r\setminus((-\infty,0)\cup(1,+\infty))$ under the projection $\pi$, and $l$ connects two points strictly above and below the equator.

\end{definition}
Note that since a separating arc $l$ has to intersect the equator at least at one point, its preimage $\tilde l\subset A_r\setminus((-\infty,0)\cup(1,+\infty))$ is uniquely determined.

Let $f\in\mathfrak C_r$ be a generalized cylinder map. We denote by $\tilde f\colon U_r\to\bbC$  some fixed lift of $f$ via
 the relation~(\ref{proj_rel}). If $l\subset V_r$ is a separating arc, then by $f_l$ we denote the analytic continuation of $f$ to a neighborhood of $l$, obtained as the projection of $\tilde f$ restricted to a neighborhood of the curve $\tilde l$. This specifies a branch of $f$ in the case it is multiple valued in $V_r$. In particular, this analytic continuation 
satisfies the relation $\pi\circ\tilde f=f_l\circ\pi$.
If both $l$ and $f_l(l)$ are separating arcs, then in order to simplify the notation, by $f^2_l$ we denote the composition $f_{f_l(l)}\circ f_l$. Similarly, by induction we can define $f^n_l=f_{f_l^{n-1}(l)}\circ f_l^{n-1}$, for any positive integer~$n$, provided that $f_l^{n-1}(l)$ is a separating arc.

\noindent
\begin{definition}
\label{def-crescent2}
Let $f\in\mathfrak C_r$. We say that a domain $C\subset V_r$ is {\it a fundamental crescent of $f$ with period $n$} if the boundary $\partial C$ contains a separating arc $l$ with end points $a$, $b$, the map $f_l^n$ is defined and univalent in a neighborhood of $l$,  $\partial C=l\cup f_l^n(l)$, $l\cap f_l^n(l)=\{a,b\}$, and $C$ is a fundamental crescent for the map $f_l^n$.

In the remaining part of the paper we will always require that $0\in C$, and we set the basepoint $\omega=0$.
\end{definition}

\subsection{Fundamental crescents of holomorphic pairs}
In \cite{Ya3} the existence of fundamental crescents was stated for all maps in the Epstein class. In this paper it will be more convenient for
us to prove the corresponding statement for holomorphic commuting pairs. We start with some notation.
Let $\mathcal H=(\eta,\xi,\nu)$ be a holomorphic pair 
 with domains 
$U$, $V$, $D$ and range $\Delta$, as above.
By Denjoy-Wolff theorem, there exists a unique point $p_\eta^+\in\overline{\bbH\cap U}$, such that the iterates of $\eta^{-1}$ converge to $p_\eta^+$ uniformly on compact subsets of $\bbH\cap \Delta$. We define the point $p_\eta^-\in\bbC$ as $p_\eta^-=\overline{p_\eta^+}$.
If $p_\eta^+\neq p_\eta^-$, then both $p_\eta^+$ and $p_\eta^-$ are repelling fixed points of the map $\eta$. Otherwise, if $p_\eta^+= p_\eta^-$, then $p_\eta^+$ lies on the real line and is a parabolic fixed point of the map $\eta$.

\begin{definition}
In the above notation, we say that a simply connected domain $C_{\mathcal H}\Subset\Delta$ is a fundamental crescent for the holomorphic pair $\mathcal H$, if the domain $C_{\mathcal H}$ is a real-symmetric fundamental crescent for the map $\eta$, such that $p_\eta^+,p_\eta^-\in\partial C_{\mathcal H}$, and $0\in C_{\mathcal H}$ (see Figure \ref{fig-crescent}).
\end{definition}

\begin{figure}[ht]
\centerline{\includegraphics[width=\textwidth]{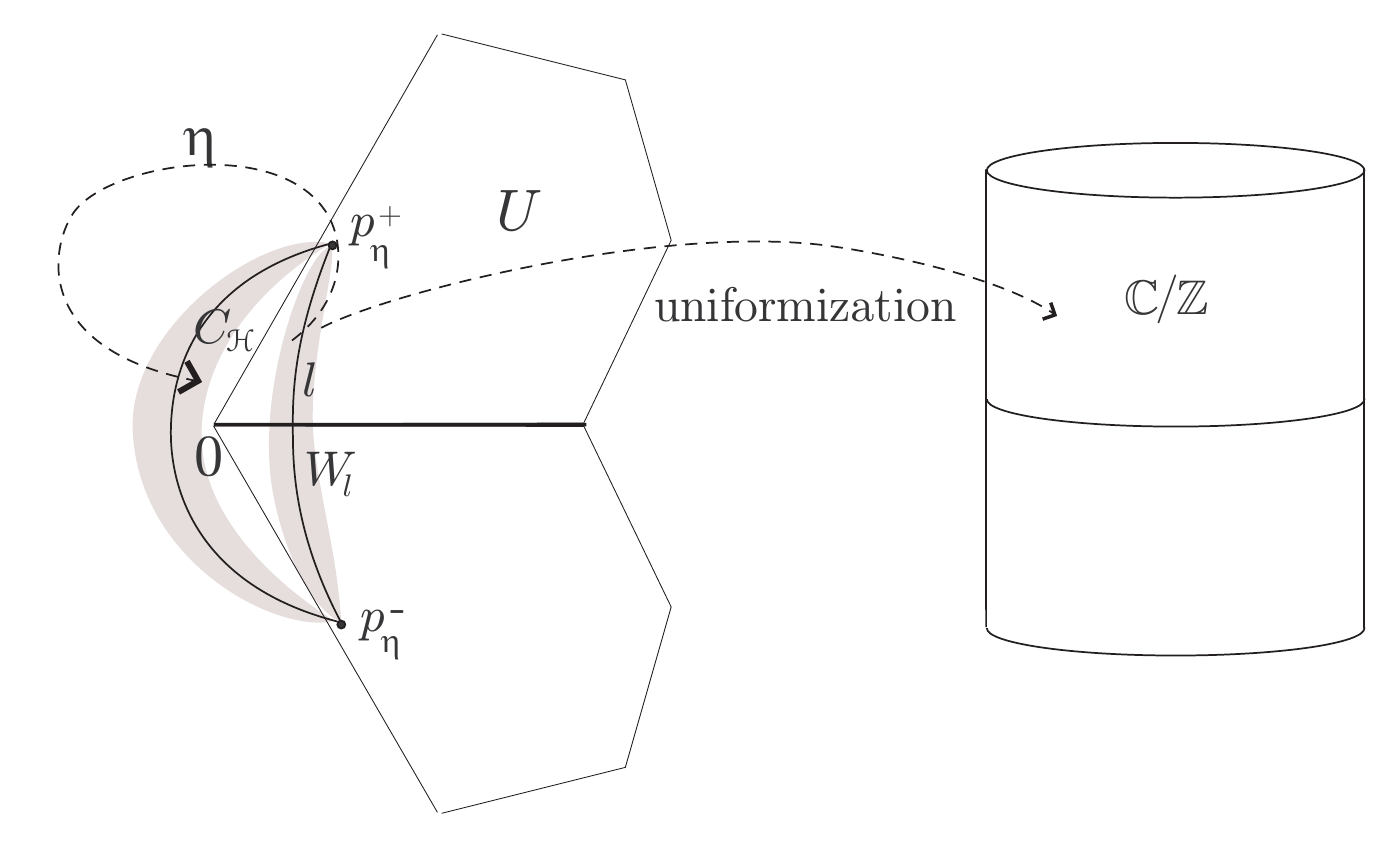}}
\caption{\label{fig-crescent}A fundamental crescent of a holomorphic pair $\mathcal H$.}
\end{figure}

\noindent
For the remaining part of the paper we again set the basepoint $\omega=0$, so that
the projection
$$
\pi_{C_{\mathcal H}^o}\colon C_{\mathcal H}^o \to\bbC
$$
is the unique holomorphic map that conjugates $\eta$ with the unit shift $z\mapsto z+1$, and $\pi_{C_{\mathcal H}^o}(0)=0$. 

We notice that according to our definition of a holomorphic pair, the domains $U\cap\bbH$ and $\Delta\cap\bbH$ have Jordan boundary, hence the map $\eta^{-1}|_{\Delta\cap\bbH}$ continuously extends to the closure $\overline{\Delta\cap\bbH}$. In order to simplify further notation, we will write $\eta^{-1}$ instead of $\eta^{-1}|_{\overline{\Delta\cap\bbH}}$. 

The main result of this subsection is the following lemma:

\begin{lemma}\label{E_crescent_exist_lemma}
\begin{enumerate}[(i)]
 \item \label{exist_part}
Let $\mathcal H=(\eta,\xi,\nu)$ be a holomorphic commuting pair, such that 
the map $\eta^{-1}$ 
does not have fixed points on the real line.
Then there exists a fundamental crescent $C_{\mathcal H}$ for the holomorphic pair $\mathcal H$, 
such that  $\pi_{C_{\mathcal H}}(C_{\mathcal H})$ contains the interval $(-1/2,1/2)\subset\bbR$.

\item \label{domain_part}
For every $s\in (0,1)$, there exist a real number $a_s>0$ and a domain $D_s\subset\bbC$, such that $[-1/2,1/2]\subset D_s$, and if $\mathcal H=(\eta,\xi,\nu)$ is the same as in part~\ref{exist_part} of the Lemma, and $\zeta_\cH\in\mcE_s$, then the fundamental crescent $C_{\mathcal H}$ has a fattening $C_{\mathcal H}^o$, such that $D_s\subset \pi_{C_{\mathcal H}^o}(C_{\mathcal H}^o)$, and $|\pi_{C_{\mathcal H}^o}^{-1}(1/2)-\pi_{C_{\mathcal H}^o}^{-1}(-1/2)|>a_s$.
\end{enumerate}
\end{lemma}

The proof of Lemma~\ref{E_crescent_exist_lemma} is based on the following construction:

\begin{proposition}\label{lin_eta_lemma}
Let $\mathcal H=(\eta,\xi,\nu)$ be a holomorphic commuting pair with range $\Delta$, such that 
the map $\eta^{-1}$ 
does not have fixed points on the real line. 
Then there exists a unique conformal change of coordinates $\theta_\eta\colon\Delta\cap\bbH\to\bbC$, such that $\theta_\eta(p_\eta^+)=0$, the map $\theta_\eta$ conjugates $\eta^{-1}$ with the linear map $z\mapsto\lambda z$, where $\lambda$ is the multiplier of the fixed point $p_\eta^+$ for the map $\eta^{-1}$, and $\theta_\eta$ has a continuous extension to the boundary $\partial(\Delta\cap\bbH)$. 
\end{proposition}
\begin{proof}
Since the map $\eta$ does not have fixed points on the real line, the point $p_\eta^+$ is an interior point of $\Delta\cap\bbH$, hence is a globally attracting fixed point for the map $\eta^{-1}$. Then by Koenigs Theorem, there exists a  conformal change of coordinates $\theta_\eta$ in a neighborhood of $p_\eta^+$ which linearizes the map $\eta^{-1}$. Since the map $\eta^{-1}$ is univalent on $\Delta\cap\bbH$, the local linearizing conformal chart around $p_\eta^+$ can be extended to the whole domain $\Delta\cap\bbH$ by pullbacks under the dynamics of $\eta^{-1}$. Thus we obtain the linearizing conformal chart defined on the whole domain~$\Delta\cap\bbH$.

Since $\eta$ does not have fixed points on the real line, this implies that there exists an integer $k>0$, such that for every $z\in\partial(\Delta\cap\bbH)$ the point $\eta^{-k}(z)$ is contained in $\Delta\cap\bbH$. This means that the linearizing chart $\theta_\eta$ extends to the boundary $\partial(\Delta\cap\bbH)$ 
by finitely many pullbacks, hence $\theta_\eta$ has a continuous extension to $\partial(\Delta\cap\bbH)$.
\end{proof}

Assume that $\mathcal H=(\eta|_U,\xi|_V,\nu|_D)$ is the same as in Proposition~\ref{lin_eta_lemma}. 
Since $\eta^{-1}$ does not have fixed points on the real line, the forward iterates of the interval $I_U=U\cap\bbR$ under the map $\eta^{-1}$ converge to the point $p_\eta^+$. Notice that $\eta(I_U)\cup I_U$ is an interval on the real line, hence, the set
$$
S_\eta=\{p_\eta^+\}\bigcup\left(\bigcup_{j=-1}^\infty\eta^{-j}(\overline{I_U})\right)
$$
is a simple curve, and the point $p_\eta^+$ is one of its end points.

\begin{figure}[ht]
\centerline{\includegraphics[width=\textwidth]{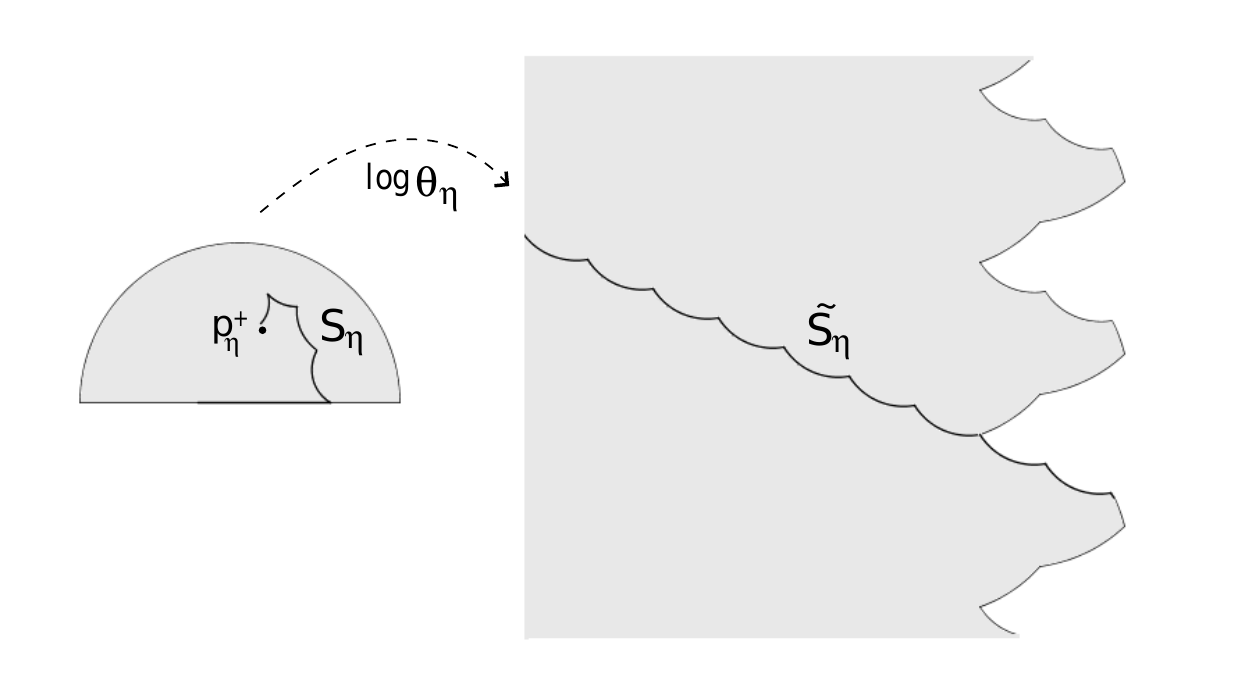}}
\caption{\label{fig-crescent2}The curve $S_\eta$ and its image under the log-linearizing coordinate. }
\end{figure}

Passing to logarithmic coordinates $\log\theta_\eta$ in the image $\theta_\eta(\Delta\cap\bbH)$, we obtain a  domain $W\subset\bbC$ that is invariant under the shifts by $2\pi i$ and contains a halfplane $\{z\in\bbC\mid\Re z<x_0\}$, for some $x_0\in\bbR$ (see Figure \ref{fig-crescent2}). The map $\phi_\eta\colon W\to\Delta\cap\bbH$ defined by $\phi_\eta(z)=\theta_\eta^{-1}(e^z)$ is a covering map from $W$ onto $(\Delta\cap\bbH)\setminus\{p_\eta^+\}$. It follows from Proposition~\ref{lin_eta_lemma} that the map $\phi_\eta$ semi-conjugates the map $\eta^{-1}$ with the shift $z\mapsto z+w_\eta$, where $w_\eta\in\bbC$ is such that $\Re w_\eta<0$ and $-\pi<\Im w_\eta\le\pi$.

Let $\tilde S_\eta\subset\overline{W}$ be an arbitrary fixed lift of the curve $S_\eta$ to $\overline{W}$ under the covering map $\phi_\eta$. In particular, it follows from the previous observation that the simple curve $\tilde S_\eta$ is invariant under the shift by $w_\eta$, and a part of this curve belongs to the boundary of $W$. The curve $\tilde S_\eta$ splits the domain $W$ into two subdomains, and we let $W_{\eta}\subset  W$ be the one of these subdomains, whose boundary contains $\tilde S_\eta$ (see Figure~\ref{fig-crescent2}). Further, let $p_0\in\tilde S_\eta$ be the point that projects to zero under the covering map: $\phi_\eta(p_0)=0$.

\begin{proposition}\label{Unique_straightening_chart_lemma}
Let $\mathcal H=(\eta,\xi,\nu)$ be the same as in Proposition~\ref{lin_eta_lemma}. 
Then there exists a unique conformal map $\tau\colon W_{\eta}\to\bbC$, such that $\tau(z+w_\eta)=\tau(z)-1$, for all $z\in W_{\eta}$, the map $\tau$ continuously extends to the boundary curve $\tilde S_\eta$ so that $\tau(\tilde S_\eta)\subset\bbR$ and $\tau(p_0)=0$, and there exists a simple curve $\gamma\subset\tau(W_{\eta})$ that does not have common points with its shift by $1$ and whose projection onto the imaginary axis is either $\{ix\mid x>0\}$ or $\{ix\mid x<0\}$. Furthermore, the curve $\phi_\eta(\tau^{-1}(\gamma))$ is simple and has an endpoint at $p_\eta^+$.
\end{proposition}
\begin{proof}
Since the domain $W$ contains a halfplane $\{z\in\bbC\mid\Re z<x_0\}$, for some $x_0\in\bbR$, this means that there exists an infinite ray $l\subset W_{\eta}$ 
with the end point in $\tilde S_\eta$, such that $l$ is not parallel to the vector $w_\eta$ and the projection of $l$ onto $\bbR$ is an infinite interval of the form $(-\infty,a]$. Let $l'\subset W_{\eta}$ be the shift of $l$ by $w_\eta$. Since $l$ is not parallel to $w_\eta$, the rays $l$ and $l'$ have no common points. Let $C_l\subset W_{\eta}$ be the subdomain of $W_{\eta}$ bounded by the curves $l$, $l'$ and $\tilde S_\eta$. Again, since $l$ is not parallel to $w_\eta$, the set $C_l\cup l\cup l'$ factored by the action of the shift by $w_\eta$, is a Riemann surface conformally isomorphic to a punctured disk. 

Let $p_0'\in\tilde S_\eta$ be the unique point such that $p_0'\in\partial C_{l}$ and $p_0'-p_0=kw_\eta$, where $k\in\bbZ$. Then there exists a unique conformal map $\tau$ defined in the relative neighborhood of $C_l$ in $W_{\eta}$, such that $\tau(z+w_\eta)=\tau(z)-1$, whenever both sides are defined, $\tau$ continuously extends to $\tilde S_\eta\cap \partial W_{\eta}$, $\tau(\tilde S_\eta\cap \partial W_{\eta})\subset\bbR$, and $\tau(p_0')=-k$. Finally, we extend the map $\tau$ to the whole domain $W_{\eta}$ by pullbacks or push forwards under the dynamics.

Since the rays $l$ and $l'$ have an empty intersection, the curve $\tau(l)$ does not have common points with its shift by $1$. Since the set $C_l\cup l\cup l'$ factored by the action of the shift by $w_\eta$, is a Riemann surface conformally isomorphic to a punctured disk, the simple curve $\tau(l)$ has an unbounded projection onto the imaginary axis. Finally, we set $\gamma=\tau(l)$, and we notice that since the projection of $l$ onto $\bbR$ is an interval of the form $(-\infty,a]$, the curve $\phi_\eta(\tau^{-1}(\gamma))=\phi_\eta(l)$ is simple and has an endpoint at $p_\eta^+$.
\end{proof}

\begin{proof}[Proof of Lemma~\ref{E_crescent_exist_lemma}]
\ref{exist_part} Let $\tau\colon W_{\eta}\to\bbC$ be the conformal map provided by Proposition~\ref{Unique_straightening_chart_lemma} and let $\gamma\subset\tau(W_{\eta})$ be the curve from the same Proposition. Then there exists a straight line segment in $\tau(W_{\eta})$ that connects the point $1/2$ with the curve $\gamma$. The point at which this line segment intersects the curve $\gamma$, splits $\gamma$ into two pieces. Let $l\subset\tau(W_{\eta})$ be the curve obtained as the union of the unbounded piece of $\gamma$ and the line segment. We define $l'\subset\tau(W_{\eta})$ as the shift of $l$ by $-1$. It follows from Proposition~\ref{Unique_straightening_chart_lemma} that the image under the map $\phi_\eta\circ\tau^{-1}$ of the subdomain of $\tau(W_{\eta})$ bounded by the curves $l$ and $l'$, is the upper or lower half of a fundamental crescent for $\mathcal H$. The other half of this fundamental crescent is obtained by reflection about the real axis.

\ref{domain_part} We fix the real number $s\in(0,1)$. All constants in the following proof will depend only on the number $s$ and not on a particular commuting pair $\zeta_\cH\in\mcE_s$, unles otherwise stated. We split the proof into the following simple steps:

\textit{Step 1:} For every $\zeta_\cH=(\eta,\xi)\in\mcE_s$, we have $\eta\in\mcE_{[0,1],s}$, which implies that there exist two constants $C_1>c_1>0$, such that $c_1\le|\eta(0)|\le C_1$, (c.f. Definition~\ref{E_I_s_def}).

\textit{Step 2:} Lemma~\ref{E_s_bounded_distortion_lemma} together with Step~1 imply that there exists a constant $c_2>0$, such that $c_2\le|\eta^{-1}(0)|\le 1$.

\textit{Step 3:} There exists a constant $r>0$, such that the map $\pi_{C_{\cH}}$ analytically extends to $N_r([r+\eta(0),\eta^{-1}(0)])$, where $N_r(S)$ is the complex $r$-neighborhood of the set $S\subset\bbC$.

\textit{Proof of Step 3:} It follows from Proposition~\ref{Unique_straightening_chart_lemma} that the map $\pi_{C_{\cH}}$ analytically extends to $\Delta\setminus S_\eta$. Now Lemma~\ref{E_s_bounded_distortion_lemma}, Step~2 and injectivity of $\eta^{-1}$ imply that $\eta^{-3}(\Delta\cap\bbH)$ has an empty intersection with $N_r([r+\eta(0),\eta^{-1}(0)])$, for some fixed $r$. In particular, this means that $S_\eta\cap N_r([r+\eta(0),\eta^{-1}(0)])= (\eta(0),\eta^{-1}(0)+r)$, so the map $\pi_{C_{\cH}}$ analytically extends to $N_r([r+\eta(0),\eta^{-1}(0)])\cap\bbH$. Finally, $\pi_{C_{\cH}}$ extends to the whole neighborhood $N_r([r+\eta(0),\eta^{-1}(0)])$ according to the reflection principle.

\textit{Step 4:} Let $z_{-1/2}\in [0,\eta^{-1}(0)]$ be the point, for which $\pi_{C_{\cH}}(z_{-1/2})=-1/2$. Then there exists a constant $c_3>0$, such that $z_{-1/2}\ge c_3$. 

\textit{Proof of Step 4:} Since the interval $\pi_{C_{\cH}}([0,z_{-1/2}])$ 
has a fixed length equal to $1/2$, Step~2, Step~3 and the Koebe Distortion Theorem imply that if 
the interval $[0, z_{-1/2}]$ is too small, then the interval $\pi_{C_{\cH}}([0,\eta^{-1}(0)])$ is larger than the unit interval, which is a contradiction.

\textit{Step 5:} Consider the point $z_{1/2}=\eta(z_{-1/2})$. It follows from Lemma~\ref{E_s_bounded_distortion_lemma} and Step~4 that there exists a constant $c_4>0$, such that 
$z_{1/2}-\eta(0)>c_4$.

\textit{Step 6:} It follows from Step~1, Step~2 and Step~4 that the interval $[z_{1/2},z_{-1/2}]$ is commensurable with a unit interval. Now since the interval $\pi_{C_{\cH}}([z_{1/2},z_{-1/2}])$ has a fixed (unit) length, Step~5, Step~3 and the Koebe Distortion Theorem imply that $\pi_{C_{\cH}}(N_r([r+\eta(0),\eta^{-1}(0)]))$ contains a fixed domain $D_s$, such that $[-1/2,1/2]\subset D_s$. If $D_s$ is sufficiently small, we can allways choose a fattening $C_{\cH}^o$ of $C_{\cH}$, so that $D_s\subset \pi_{C_{\cH}^o}(C_{\cH}^o)$.

Existence of the constant $a_s$ immediately follows from Step~4. Namely, we can put $a_s=c_3$.
\end{proof}

\subsection{Holomorphic commuting pairs and critical circle maps. Choice of constants.}
For a positive integer $B\in\bbN$, consider the space $\Sigma_B\subset\Sigma$ of all bi-infinite sequences of positive integers that are not greater than $B$.

\begin{definition}
For a positive integer $B>0$, let $\mathcal I_B\subset\mathbf E$ be the set of all commuting pairs from $\mathcal I$ that are images of $\Sigma_B$ under the map $\iota$ (c.f. Theorem~\ref{Comm_pair_attractor_theorem}). 
Let us denote by $\mathcal K\subset\mathcal I$ the union
$$
\mathcal K=\bigcup_{B\in\bbN} \mathcal I_B.
$$
\end{definition}

It follows from Theorem~\ref{Comm_pair_attractor_theorem} that for every integer $B>0$, the set $\mathcal I_B$ is sequentially compact and has a topological structure of a Cantor set.

\begin{lemma}\label{Cont_holo_pair_lemma}
Let $\mu>0$ be the universal constant from Theorem~\ref{complex bounds}. Then there exists a continuous embedding $\zeta\mapsto\cH_\zeta$ from $\mathcal K$ to $\mathbf H(\mu)$, such that for every critical commuting pair $\zeta\in\mathcal K$, the image $\cH_\zeta=\mathcal H=(\eta,\xi,\nu)$ is its holomorphic pair extension, and that the inverse map $\eta^{-1}=\eta^{-1}|_{\overline{\Delta_\cH\cap\bbH}}$ does not have fixed points on the real line.
\end{lemma}
\begin{proof}
Let $s>0$ be the same as in Theorem~\ref{Comm_pair_attractor_theorem}. It follows from Theorem~\ref{Comm_pair_attractor_theorem} that for every $\zeta\in\mathcal K$ there exists an infinite sequence of commuting pairs $\zeta_{-1},\zeta_{-2},\ldots\in\mathbf E_s$, such that $\mathcal R^k\zeta_{-k}=\zeta$, for all $k\in\bbN$. It follows from Lemma~\ref{E_s_bounded_distortion_lemma} and Proposition~\ref{E_s_compactness_proposition} that $\mcE_s$ is a compact subset of $\mathcal A_r$, for some $r>0$. Then Theorem~\ref{complex bounds} implies that
$\zeta$ extends to a holomorphic pair $\mathcal H=(\eta,\xi,\nu)$ from $\mathbf H(\mu)$. Theorem~\ref{complex bounds} also implies that if $\eta^{-1}$ has a fixed point on the real line, then this point is a periodic point of a critical commuting pair $\zeta_{-k}\in\mathcal K$, for some $k\in\bbN$. However, the later is not possible, since $\zeta_{-k}$ has an irrational rotation number.

Now, for each $\zeta\in\mathcal K$, the constructed correspondence $\zeta\mapsto\mathcal H$ can be extended to a continuous map from a sufficiently small neighborhood of $\zeta$ in $\mathbf E$ to the space $\mathbf H(\mu)$ preserving the property that the map $\eta^{-1}$ does not have fixed points on $\bbR$. Finally, the set ${\mathcal K}\subset\mcE$ is a countable union of nested topological Cantor sets. Hence, it can be covered by a countable collection of small enough open disjoint neighborhoods in each of which we can choose the holomorphic pair extension  
$\cH_\zeta$ in such a way that it depends continuously on $\zeta$. Continuity of the inverse map $\cH_\zeta\mapsto \zeta$ follows easily.
\end{proof}

\begin{definition}
For every positive integer $B>0$, we denote by $\tilde{\mathcal I}_B\subset \mathbf H(\mu)$ the image of $\mathcal I_B$ under the map
$$
\zeta\mapsto \cH_\zeta.
$$
Similarly, we denote by $\tilde{\mathcal K}\subset\mathbf H(\mu)$ the image of $\mathcal K$ under the same map.
\end{definition}

\begin{lemma}\label{Canon_holo_cres_lemma}\label{Choose_U_B_lemma}
For every holomorphic pair $\mathcal H\in\tilde{\mathcal K}$ there exists a fundamental crescent $C_{\mathcal H}$ with a fattening $C_{\mathcal H}^o$, such that  
the mapping $\mathcal H\mapsto C_{\mathcal H}^o$ is continuous on $\tilde{\mathcal K}$. Moreover, there exist two Jordan domains $\tilde U\Subset U_1\subset\bbC$, such that 
$$[-1/2,1/2]\subset \tilde U,$$
and for every $\mathcal H\in\tilde{\mathcal K}$ and the corresponding fattening $C_{\mathcal H}^o$, the map $\pi_{C_{\mathcal H}^o}^{-1}$ is defined and univalent on the domain $U_1$, and 
\begin{equation}\label{finite_diff_condition}
|\pi_{C_{\mathcal H}^o}^{-1}(1/2)-\pi_{C_{\mathcal H}^o}^{-1}(-1/2)|>a>0,
\end{equation}
for some universal constant $a\in\bbR$.
\end{lemma}
\begin{proof}
First, we notice that $\mathcal K\subset\mcE_s$, for the fixed value of $s>0$ from Theorem~\ref{Comm_pair_attractor_theorem}. Now it follows from Lemma~\ref{Cont_holo_pair_lemma} and Lemma~\ref{E_crescent_exist_lemma} that
for every holomorphic pair $\mathcal H\subset\tilde{\mathcal K}$ we can choose a fundamental crescent $C_{\mathcal H}$ that has a fattening $C_{\mathcal H}^o$, such that a certain fixed domain $U_1\supset [-1/2,1/2]$ is compactly contained in $\pi_{C_{\mathcal H}^o}(C_{\mathcal H}^o)$ and condition~(\ref{finite_diff_condition}) holds. 
According to Lemma~\ref{Nearby_crescent_lemma} 
this fundamental crescent and its fattening may be chosen to depend locally continuously on $\mathcal H\in\mathbf H(\mu)$.
Since $\tilde{\mathcal K}\subset \mathbf H(\mu)$ is a countable union of nested Cantor sets, the global continuous dependence $\mathcal H\mapsto C_{\mathcal H}^o$ is obtained in the same way as in Lemma~\ref{Cont_holo_pair_lemma}.

Finally, we choose the domain $\tilde U$, so that $\tilde U\Subset U_1$ and $[-1/2,1/2]\subset\tilde U$, which completes the proof.
\end{proof}

Let $\mathcal H\in\tilde{\mathcal K}$ be a holomorphic commuting pair with a fundamental crescent $C_{\mathcal H}$ provided by Lemma~\ref{Canon_holo_cres_lemma}. For a point $z\in C_{\mathcal H}$, whose iterates under ${\mathcal H}$ eventually return to $C_{\mathcal H}$, we define $R_{C_{\mathcal H}}(z)$ to be the first return of $z$ to $C_{\mathcal H}$ under the dynamics of ${\mathcal H}$. Then according to Proposition~7.9 from~\cite{Ya3}, there exists an equatorial neighborhood in the cylinder $\bbC\slash\bbZ$, such that for every $\mathcal H\in\tilde{\mathcal K}$, the map
\begin{equation}\label{f_zeta_formula}
f_\cH\equiv \tilde\pi_{C_{\mathcal H}}\circ R_{C_{\mathcal H}}\circ\tilde\pi^{-1}_{C_{\mathcal H}}
\end{equation}
is defined and analytic in that neighborhood.

\begin{lemma}\label{Choose_t_lemma}
There exists a positive real number $\tilde t>0$, such that for every holomorphic pair $\mathcal H\in\tilde{\mathcal K}$, the map $\phi_{\mathcal H}=\pi_{C_{\mathcal H}}'(0)\cdot\pi_{C_{\mathcal H}}^{-1}$ is contained in the space $\mathfrak B_{\tilde U,\tilde t}^3$, where $\tilde U$ is the same as in Lemma~\ref{Choose_U_B_lemma}.
\end{lemma}
\begin{proof}
According to Lemma~\ref{Canon_holo_cres_lemma}, the family of maps 
$$
\{\left.\pi_{C_{\mathcal H}^o}^{-1}\right|_{U_1}\colon \mathcal H\in\tilde{\mathcal K}\}
$$
is defined, and since by Definition~\ref{H_mu_def}, the image $\pi_{C_{\mathcal H}^o}^{-1}(U_1)\subset\Delta_\cH$ is uniformly bounded, this family is normal. According to~(\ref{finite_diff_condition}), every map $\pi^{-1}$ from the closure of this family is non-constant, hence univalent, and there exists a positive number $t>0$, such that $\pi'(0)\cdot\pi^{-1}\in \mathfrak B_{\tilde U, t}^3$. Now, the existence of $\tilde t>0$, satisfying the conditions of Lemma~\ref{Choose_t_lemma}, follows from compactness arguments.
\end{proof}

\begin{lemma}\label{Choose_r_h_lemma}
Let $\tilde U$ and $\tilde t$ be the same as in Lemma~\ref{Choose_U_B_lemma} and Lemma~\ref{Choose_t_lemma}, and fix a family of maps $\pi_{\phi,\alpha,\tilde t}\colon Q_{\phi,\alpha,\tilde t}\to\bbC\slash\bbZ$, $(\alpha,\phi)\in\mathfrak B_{\tilde U,\tilde t}$ provided by Lemma~\ref{pi_phi_lemma}. Then there exist positive real numbers $r_1>0$ and $h_1>0$, such that for every $\mathcal H\in\tilde{\mathcal K}$ and the corresponding map $\phi_\cH$ from Lemma~\ref{Choose_t_lemma}, there exists a unique real-symmetric map $\psi_{\mathcal H}\in\mathfrak D_{h_1}$ with the property that the composition 
\begin{equation}\label{f_tau_composition}
z\mapsto \psi_{\mathcal H}(\pi_{\phi_{\mathcal H},\alpha,\tilde t}((\phi_{\mathcal H}(z))^3))
\end{equation}
is defined in the domain $\{z\in\bbC\mid |\Re z|<1/2; |\Im z|<r_1\}$ and coincides with the map
$$
z\mapsto (\tilde\pi_{C_{\mathcal H}}\circ R_{C_{\mathcal H}}\circ\pi^{-1}_{C_{\mathcal H}})(z)
$$
on that domain. 
\end{lemma}
\begin{proof}
Since for every $n\in\bbN$, every $\zeta\in\mathcal K$ has a preimage $\cR^{-n}\zeta\in\mathcal K$, it follows from Theorem~\ref{Comm_pair_attractor_theorem} that each map $\eta$, $\xi$, $\nu$ from the holomorphic pair $\cH_\zeta=(\eta,\xi,\nu)$ can be represented in the form
$$
z\mapsto g(z^3), 
$$
where $g$ is a conformal map. Since $R_{C_\cH}$ is a composition of such maps and is a first return map to the neighborhood of the origin, it can also be represented in the same form. Thus, we get the following identity:
$$
(\tilde\pi_{C_{\mathcal H}}\circ R_{C_{\mathcal H}}\circ\pi^{-1}_{C_{\mathcal H}})(z)= \tilde\pi_{C_{\mathcal H}}(g((\pi^{-1}_{C_{\mathcal H}}(z))^3)),
$$
where $g$ is a conformal map that depends on $R_{C_\cH}$. On the other hand, we have 
$$
\psi_{\mathcal H}(\pi_{\phi_{\mathcal H},\alpha,\tilde t}((\phi_{\mathcal H}(z))^3))= \psi_{\mathcal H}(\pi_{\phi_{\mathcal H},\alpha,\tilde t}((\pi_{C_\cH}'(0))^3\cdot(\pi^{-1}_{C_{\mathcal H}}(z))^3)).
$$
Now, canceling the identical cubic parts, we get the following equation that should be satisfied by $\psi_\cH$:
\begin{equation}\label{psi_H_eq}
\tilde\pi_{C_{\mathcal H}}(g(z))= \psi_{\mathcal H}(\pi_{\phi_{\mathcal H},\alpha,\tilde t}((\pi_{C_\cH}'(0))^3\cdot z)).
\end{equation}
Since all maps in this equation are conformal, the equation uniquely determines the map $\psi_\cH$ on the domain, where appropriate compositions are defined.

Finally, since the map $f_\cH$ from~(\ref{f_zeta_formula}) is also defined in a fixed equatorial neighborhood, independent of $\cH\in\tilde{\mathcal K}$, the left hand side of~(\ref{psi_H_eq}) is defined in some fixed neighborhood. Now the existence of positive numbers $r_1$ and $h_1$ is proved using a compactness argument, similar to the one, used in the proof of Lemma~\ref{Choose_t_lemma}.
\end{proof}

\begin{remark}\label{r_1_r_remark}
In particular, Lemma~\ref{Choose_r_h_lemma} implies that for every $\cH\in\tilde{\mathcal K}$, we have the inclusion
$$
\tau=(3,\phi_\cH, \psi_\cH)\in\mathfrak P_{\tilde U,\tilde t, h_1}, 
$$
and the map $f_\tau$ defined in~(\ref{f_tau_def}) belongs to the class $\mathfrak C_{r_1}$, and coincides with the map $f_\cH$ from~(\ref{f_zeta_formula}).
\end{remark}

\begin{definition}\label{hat_K_def}
For a positive integer $B>0$, by $\hat{\mathcal I}_B\subset\mathfrak C_{r_1}^3$ we denote the image of the set $\tilde{\mathcal I}_B\subset\mathbf H(\mu)$ under the map $\mathcal H\mapsto f_{\mathcal H}$ given by~(\ref{f_zeta_formula}). Similarly, by $\hat{\mathcal K}\subset\mathfrak C_{r_1}^3$ we denote the image of the set $\tilde{\mathcal K}\subset\mathbf H(\mu)$ under the same map.
\end{definition}

\begin{definition}\label{I_check_def}
Let the constants $\tilde t$, $h_1$ and the domain $\tilde U$ be the same as in Lemma~\ref{Choose_t_lemma}, Lemma~\ref{Choose_r_h_lemma} and Lemma~\ref{Choose_U_B_lemma}. For a positive integer $B>0$, by $\check{\mathcal I}_B\subset\mathfrak P_{\tilde U,\tilde t,h_1}$ we denote the image of $\tilde{\mathcal I}_B$ under the map 
$$
\mathcal H\mapsto (3,\phi_{\mathcal H},\psi_{\mathcal H}),
$$
where $\phi_{\mathcal H}$ and $\psi_{\mathcal H}$ are the same as in Lemma~\ref{Choose_t_lemma} and Lemma~\ref{Choose_r_h_lemma} respectively.
Similarly, by $\check{\mathcal K}\subset\mathfrak P_{\tilde U,\tilde t,h_1}$ we denote the image of the set $\tilde{\mathcal K}\subset\mathbf H(\mu)$ under the same map.
\end{definition}

We recall that if positive numbers $r$ and $h$ are such that $r<r_1$ and $h<h_1$, then there are natural inclusions $\mathfrak C_{r_1}\subset\mathfrak C_r$ and $\mathfrak P_{\tl U,\tl t, h_1}\subset \mathfrak P_{\tl U,\tl t, h}$, hence we can view the sets $\hat{\mathcal K}$ and $\check{\mathcal K}$ as subsets of $\mathfrak C_r$ and $\mathfrak P_{\tl U,\tl t,h}$ respectively.

\begin{lemma}\label{Choose_hr_tilde_lemma}
There exist positive constants $\tl h$, $\tl r$ and an open set $\check{\mathcal U}\subset\mathfrak P_{\tl U,\tl t,\tl h}$, such that $0<\tl h<h_1$, $0<\tl r<r_1$, $\check{\mathcal K}\Subset\check{\mathcal U}$ and the mapping
$$
\tau\mapsto f_\tau
$$
is an analytic map from $\check{\mathcal U}\subset\mathfrak P_{\tl U,\tl t,\tl h}$ to $\mathfrak C_{\tl r}$.
\end{lemma}
\begin{proof}
Fix a constant $\tl h<h_1$. According to Lemma~\ref{Choose_r_h_lemma}, for every $\tau=(3,\phi_\cH,\psi_\cH)$ from the closure of $\check{\mathcal K}$, the map $\psi_\cH$ is defined on the annulus $V_{h_1}$. If we view $\psi_\cH$ as the same map but restricted to a smaller annulus $V_{\tl h}$, then according to Lemma~\ref{Choose_r_h_lemma} and Koebe Distortion Theorem, the composition~(\ref{f_tau_composition}) is guaranteed to be defined in the set $\{z\in\bbC\mid |\Re z|<1/2; |\Im z|<\hat r\}$, where $\hat r$ is some constant, such that $0<\hat r\le r_1$. Finally, we set $\tl r=\hat r/2$. Then Lemma~\ref{f_tau_analyt_lemma} implies the existence of an open set $\check{\mathcal U}\subset\mathfrak P_{\tl U,\tl t,\tl h}$ with required properties.
\end{proof}

For future reference we formulate the following lemma: 

\begin{lemma}\label{Choose_K_lemma}
There exists a positive real number $r>0$, such that every critical commuting pair  $\zeta=(\eta|_{I_\eta},\xi|_{I_\xi})\in\mathcal K$ belongs to the class $\mathcal A_r$ from Definition~\ref{A_r_com_pair_def}, and if $\cH=\cH_\zeta\in\tilde{\mathcal K}$, then 
$$
\tilde{\pi}_{C_\cH^o}(C_\cH^o\cap (N_r(I_\eta)\cup N_{r|\eta(0)|}(I_\xi)))\subset V_{\tl r}.
$$
Furthermore, there exists $K\in\bbN$ with the property that for every positive integer $k\ge K$, and for every $\zeta\in\mathcal K$, 
the critical commuting pair $p\cR^k\zeta$ restricts (in the sense of Theorem~\ref{complex bounds}) to a holomorphic commuting pair $\cH_k\colon\Omega_k\to\Delta_k$, such that $\Delta_k\Subset C_\cH\cap (N_r(I_\eta)\cup N_{r|\eta(0)|}(I_\xi))$, and $\cH_k$ is an affine rescaling of a holomorphic pair from $\tilde{\mathcal K}$.
\end{lemma}
\begin{proof}
Existence of the number $r$ that satisfies the first part of the lemma, follows easily from Lemma~\ref{E_s_bounded_distortion_lemma}, Lemma~\ref{Choose_U_B_lemma} and the Koebe Distortion Theorem.

Let $s$ and $\mu$ be the same as in Theorem~\ref{Comm_pair_attractor_theorem} and Theorem~\ref{complex bounds} respectively. Since $\mathcal K\subset\mcE_s$ and $\mathcal K\subset\mathcal A_r$, the set $\mathcal K$ is pre-compact in $\mathcal A_r$, hence by Theorem~\ref{complex bounds}, there exists $N=N(r,{\mathcal K})$, such that for every $\zeta\in\mathcal K$, the prerenormalization $p\cR^N\zeta$ restricts (in the sense of Theorem~\ref{complex bounds}) to some holomorphic pair $\mathcal G_\zeta$ whose appropriate affine rescaling lies in $\mathbf H(\mu)$ and the range of $\mathcal G_\zeta$ is compactly contained in $C_\cH\cap (N_r(I_\eta)\cup N_{r|\eta(0)|}(I_\xi))$. Then we note that the range of holomorphic pairs $\cR^k(\mathcal G_\zeta)$ is the same for all $k=0,1,2,\dots$, hence Lemma~\ref{E_s_bounded_size_lemma} and properties~\ref{hol_mu_2},~\ref{hol_mu_5} of Definition~\ref{H_mu_def} imply that there exists a positive integer $L\in\bbN$ with the property that for every $k\ge L$ and $\zeta\in \mathcal K$, if $\mathcal G$ is a holomorphic pair extension of $p\cR^{k+N}\zeta$, such that an appropriate affine rescaling of $\mathcal G$ lies in $\hol(\mu)$, 
then $\mathcal G$ is a restriction of the holomorphic pair $p\cR^k(\mathcal G_\zeta)$. This implies that the range of $\mathcal G$ is compactly contained in $C_\cH\cap (N_r(I_\eta)\cup N_{r|\eta(0)|}(I_\xi))$, and $\mathcal G$ is a restriction of $p\cR^{k+N}\zeta$ in the sense of Theorem~\ref{complex bounds}.

We finish the proof by putting $K=N+L$. Now the choice of holomorphic pairs $\cH_k$ is possible since $\cR(\mathcal K)=\mathcal K$.
\end{proof}

\section{Renormalization in $\mathfrak C_r$}
\label{sec:renorm1}

\subsection{Renormalization with respect to a fundamental crescent}
We let $V_r$ be as in (\ref{v-annulus}). 
Assume, $f\in\mathfrak C_r$, for some real number $r\in\bbR$, $0<r<0.5$. Then we can restrict $f$ to an analytic (single-valued) map $\breve f\colon V_r\setminus\{\Re z=0\}\to \bbC\slash\bbZ$. 
\begin{definition}
	Let $f\in\mathfrak C_r$ be a generalized critical cylinder map and let $C_f\subset V_r$ be its fundamental crescent of period $n$. For a point $z\in C_f$, such that $\{\breve f^j(f(z))\}_{j\in\bbN}\cap C_f\neq\varnothing$, we define $R_{C_f}(z)$ to be the first return of $f(z)$ to $C_f$ under the map $\breve f$. In general, $R_{C_f}$ is a multiple-valued map since $f$ is also multiple-valued.
\end{definition}

\begin{definition}\label{Renorm_cres_gen_cyl_def}
	Given a generalized critical cylinder map $f\in\mathfrak C_r$ and its fundamental crescent $C_f\subset V_r$ of period $n$, let us say that $f$ is renormalizable with respect to the fundamental crescent $C_f$, if there exists a positive real number $r'>0$, such that 
	 the composition
	\begin{equation*}\label{Renorm_real_eq}
	\hat f= \tilde\pi_{C_f}\circ R_{C_f}\circ\tilde\pi_{C_f}^{-1}
	\end{equation*}
	belongs to  $\mathfrak C_{r'}$.

	We will say that the  generalized critical cylinder map $\hat f\in\mathfrak C_{r'}$ is the {\it renormalization of $f$ with respect to the fundamental crescent $C_f$} and we will denote it by $$\hat f=\mathcal R_{C_f} f.$$
\end{definition}

The following lemma is easy to verify:
\begin{lemma}\label{alpha_preserve_lemma}
	If $f\in\mathfrak C_r^\alpha$, for some $\alpha\in\bbC$, and $f$ is renormalizable with respect to a fundamental crescent $C_f$, then $\cR_{C_f}f\in\mathfrak C_{r'}^\alpha$. In other words, renormalization does not change the critical exponent at zero.
\end{lemma}

\subsection{Cylinder renormalization operator $\cren$}\label{Cyl_ren_sec}
Let  $\mu$ be as in Theorem \ref{complex bounds}, and let $\hat{\mathcal K}\subset\mathfrak C_{\tilde r}^3$ be the same as in Definition~\ref{hat_K_def}. We also fix $r_1>\tilde r$ as in Lemma~\ref{Choose_r_h_lemma} and Lemma~\ref{Choose_hr_tilde_lemma}, and let $\tilde U$ be as in Lemma~\ref{Choose_U_B_lemma}. 

\begin{remark}\label{M_real_bounds_remark}
It follows from real a priori bounds (c.f. Lemma~\ref{E_s_bounded_size_lemma}) and the Koebe Distortion Theorem that there exists a positive integer $L>0$ with the property that for every $f\in\hat{\mathcal K}$ and for every real-symmetric conformal map $h\colon\tl U\to\bbC$, such that $h([-1/2,1/2])$ is contained in the domain of the commuting pair $p\cR^Lf$, we have 
$$
|h'(0)|<1/3.
$$
\end{remark}
Set $N$ to be the smallest even integer greater than $\max(K,L)$, where $K$ is the same as in Lemma~\ref{Choose_K_lemma} and $L$ is the same as in Remark~\ref{M_real_bounds_remark}.

\begin{remark}
The fact that the number $N$ is even, will be important in Section~\ref{sec:renorm2}.
\end{remark}

\begin{lemma}\label{U_lemma}
There exists a positive real number $r_2>\tilde r$, an open neighborhood $U_2\subset\bbC$ with $\tilde U\Subset U_2$, and an open set $\mathcal U\subset\mathfrak C_{\tilde r}$, such that $\hat{\mathcal K}\subset\mathcal U$, and for each $f\in\mathcal U$ there exists a choice of a fundamental crescent $C_f\subset V_{\tilde r}$ and its fattening $C_f^o\subset V_{\tilde r}$, such that both $C_f$ and $C_f^o$ depend continuously on $f\subset\mathcal U$ and the following holds:

(i) $C_f$ has period $q_{N-1}$, where $q_{N-1}$ is the denominator of the $(N-1)$-st convergent of $\rho(f)$, written in the irreducible form;

(ii) the dependence $f\mapsto \pi_{C_f^o}(z)$ is analytic for any fixed $z\in C_f^o$;

(iii) 
$$
\mathcal R_{C_f}f\in\mathfrak C_{r_2},
$$
the map $\pi_{C_f^o}^{-1}$ is defined and univalent on $U_2$, and 
$|(\pi_{C_f^o}^{-1})'(0)|<1/2$.

(iv) for $\zeta\in\mathcal K$, $\cH=\cH_\zeta\in\tilde{\mathcal K}$ and $f=f_\cH\in\hat{\mathcal K}$, the fundamental crescent $C_f$ and its fattening $C_f^o$ satisfy the relations
$$
C_f=\tilde{\pi}_{C_\cH}(C_{\cH_1}),\qquad C_f^o=\tilde{\pi}_{C_\cH}(C_{\cH_1}^o),
$$
where $\cH_1$ is a linear rescaling of the holomorphic pair $\cH_{\cR^N\zeta}$, and $C_{\cH_1}$ and $C_{\cH_1}^o$ are linear rescalings of $C_{\cH_{\cR^N\zeta}}$ and $C_{\cH_{\cR^N\zeta}}^o$ respectively.
\end{lemma}
\begin{proof}
For an arbitrary $f=f_\mathcal H\in\hat{\mathcal K}$ with the corresponding holomorphic commuting pair $\mathcal H\in\tilde{\mathcal K}$ and the underlying critical commuting pair $\zeta\in\mathcal K$, the renormalization $\mathcal R^N\zeta$ belongs to $\mathcal K$, hence it extends to a holomorphic commuting pair $\cH_{\mathcal R^N\zeta}\in\mathbf H(\mu)$. Let $\cH_1$ be the holomorphic pair extension of the pre-renormalization $p\cR^N\zeta$, such that $\cH_1$ is a linear rescaling of $\cH_{\mathcal R^N\zeta}$. Let $C_{\cH}$ be the fundamental crescent of $\cH$ from Lemma~\ref{Canon_holo_cres_lemma}.
From Lemma~\ref{Choose_K_lemma} 
we can see that $\Delta_{\cH_1}\Subset C_\cH$, and the projection $\pi_{C_\cH}$ conjugates $\cH_1$ with a holomorphic pair extension $\cH_2$ of $p\cR^{N-1}(f)$, such that $\Delta_{\cH_2}\Subset V_{\tilde r}$ and the correspondence $f\mapsto \cH_2$ extends to a continuous map from a neighborhood of $f\in\hat{\mathfrak C}^3_r$ to $\hol$.

Let $C_{\cH_1}$ and $C_{\cH_1}^o$ be the fundamental crescent and its fattening that are linear rescalings of the fundamental crescent $C_{\cH_{\cR^N\zeta}}$ and the corresponding fattening $C_{\cH_{\cR^N\zeta}}^o$, provided by Lemma~\ref{Canon_holo_cres_lemma}. Then we set
$$
C_{f}=\tilde{\pi}_{C_\mathcal H}(C_{\mathcal H_1})\qquad\text{and }\qquad C_{f}^o=\tilde{\pi}_{C_\mathcal H}(C_{\mathcal H_1}^o).
$$
We notice that our construction satisfies property (iv) of the lemma.
Now it follows from the above argument that $C_f$ is a fundamental crescent for $\cH_2$, hence according to Definition~\ref{def-crescent2}, it is also a fundamental crescent for $f\in\hat{\mathcal K}$ of period $q_{N-1}$, and $C_f^o$ is a fattening of $C_f$. Moreover, since $\pi_{C_\cH}$ is a conjugacy between $\cH_1$ and $\cH_2$, we have $\cR_{C_f}f=f_{\cH_1}\in\mathfrak C_{r_1}$, where $r_1$ is the same as in Lemma~\ref{Choose_r_h_lemma}.

Now Lemma~\ref{Nearby_crescent_lemma} implies that the correspondence $f\mapsto C_f^o$ extends to a continuous map on some neighborhood of $f\in\mathfrak C_{\tilde r}$, satisfying property~(ii). Choose $r_2\in\bbR$ and $U_2\subset\bbC$ so that $\tilde r<r_2<r_1$ and $\tilde U\Subset U_2\Subset U_1$. Then according to Lemma~\ref{Choose_U_B_lemma}, Remark~\ref{r_1_r_remark}, Remark~\ref{M_real_bounds_remark} and continuity argument, by possibly shrinking the neighborhood of $f$, we can ensure that property~(iii) also holds. Finally, in order to construct the open set ${\mathcal U}$, we use the fact that the set $\hat{\mathcal K}$ is a countable union of nested Cantor sets, and the rest of the proof goes in the same way as the proof of Lemma~\ref{Cont_holo_pair_lemma}.
\end{proof}

\begin{definition}\label{cren_def}
Let $\mathcal U\subset\mathfrak C_{\tilde r}$ be the open set constructed in Lemma~\ref{U_lemma}. Then we define the {\it cylinder renormalization operator} $\cren\colon\mathcal U\to\mathfrak C_{\tilde r}$ as
$$
\cren(f)=\mathcal R_{C_f}f,
$$
where $C_f$ is the canonical fundamental crescent for $f\in\mathcal U$, constructed in Lemma~\ref{U_lemma}.
\end{definition}
Our definition extends the definition given by the second author in \cite{Ya3} to a wider class of analytic maps. 

\begin{definition}
By $\hren$ we denote the restriction of the cylinder renormalization operator $\cren$ to the subset $\mathcal U\cap\hat{\mathfrak C}_{\tilde r}^3$, where $\hat{\mathfrak C}_{\tilde r}^3\subset\mathfrak C_{\tilde r}^3$ is the set from Definition~\ref{C_hat3_def}.

\end{definition}

In the following proposition we summarize some basic properties of the cylinder renormalization operator:

\begin{proposition}\label{cren_basic_proposition}
\begin{itemize}
\item[(i)] The cylinder renormalization operator $\cren$ is a real-symmetric analytic operator $\cU\to\kC_{\tilde r}$.

\item[(ii)] For every critical commuting pair $\zeta\in\mathcal K$, we have
$$
\mathcal \cren f_{\cH_\zeta}=f_{\cH_{(\mathcal R^N\zeta)}},\text{ where }f_{\cH_{(\cR^N\zeta)}}\text{ is as in (\ref{f_zeta_formula})}.
$$

\item[(iii)] $\cren(\hat{\mathcal I}_B)=\hat{\mathcal I}_B$, for every $B\in\bbN$.

\item[(iv)] For every complex number $\alpha\in\bbC$ sufficiently close to $3$, we have
$$
\mathcal \cren(\mathcal U\cap\mathfrak C_{\tilde r}^\alpha)\subset\mathfrak C_{\tilde r}^\alpha.
$$

\item[(v)] $\hren(\mathcal U\cap\hat{\mathfrak C}_{\tilde r}^3)\subset\hat{\mathfrak C}_{\tilde r}^3$.

\item[(vi)] At every point 
$f\in \mathcal U\cap\hat{\mathfrak C}_{\tilde r}^3$ 
the differential $D_f\hren\colon T_f{\hat\kC}_{\tl r}^3\to T_{\cren f}{\hat\kC}_{\tl r}^3$ is a compact operator. 
\end{itemize}
\end{proposition}
\begin{proof}
Analyticity part of the statement (i) follows from Lemma \ref{U_lemma}. 
Real symmetry is evident from 
the construction. Properties (ii),(iii), and (v) were established in \cite{Ya3}. Finally, (iv) follows from Lemma~\ref{U_lemma} and Lemma~\ref{alpha_preserve_lemma}.

By Montel's theorem, a bounded set in the tangent space $T_{\hren f} \hat{\kC}_{r_2}^3$ is normal, hence pre-compact in $T_{\hren f} {\hat\kC}_{\tl r}^3$. 
Since differential is a bounded operator, property (vi) follows from this and property (iii) of Lemma~\ref{U_lemma} (c.f. Proposition~9.1 from \cite{Ya3}).
\end{proof}

\section{Renormalization in $\mathfrak P_{U,t,h}$}
\label{sec:renorm2}

Let the maps $p_{\phi,\alpha+}$ and $p_{\phi,\alpha-}$ be defined in the same way as in Section~\ref{sec:spaces} after Definition~\ref{U_r_pm_def}. 
The following technical statement is straightforward to verify:
\begin{proposition}\label{commut_prop}
Let the maps $p_{\phi,\alpha+}$ and $p_{\phi,\alpha-}$ be as above, and let $c\in\bbC$ be a non-zero complex number, such that $\Re c>0$. Assume that $\Re \phi'(0)>0$. Then
$$
p_{\phi,\alpha+} \equiv \frac{1}{p_{\alpha+}(c)}p_{c\phi,\alpha+}\qquad\text{and}\qquad
p_{\phi,\alpha-} \equiv \frac{1}{p_{\alpha+}(c)}p_{c\phi,\alpha-}.
$$
\end{proposition}

\begin{lemma}\label{composition_lemma}
Let $r_2$ and $U_2$ be the same as in Lemma~\ref{U_lemma}. 
There exists a positive real number $h_2>\tilde h$, 
and an open set $\tilde{\mathcal U}\subset\mathfrak P_{\tilde U,\tilde t,\tilde h}$, such that $\check{\mathcal K}\subset\tilde{\mathcal U}$, and for every $\tau=(\alpha,\phi,\psi)\in\tilde{\mathcal U}$, the following properties hold:

\begin{enumerate}[(i)]
\item the generalized critical cylinder map $f_\tau$ from~(\ref{f_tau_def}) is contained in the open set $\mathcal U$ from Definition~\ref{cren_def}, and hence, has the canonical fundamental crescent $C_{f_\tau}$ with the fattening $C_{f_\tau}^o$, constructed in Lemma~\ref{U_lemma};
\item $\Re \pi_{C_{f_\tau}^o}'(0)>0$; 
\item\label{phi_tilde_condit} the map $\tilde{\phi}=\pi_{C_{f_\tau}}'(0)\cdot\phi\circ\pi_{C_{f_\tau}^o}^{-1}$ is defined in $U_2$;
\item $\tilde\phi\in\mathfrak B_{U_3,\tilde t}^\alpha$, for every domain $U_3$, such that $\tl U\subseteq U_3\Subset U_2$;
\item\label{psi_tilde_condit} there exists a unique map $\tilde{\psi}\in\mathfrak D_{h_2}$, such that the following identity holds in the domain $\{z\in\bbC\mid |\Re z|<1/2; |\Im z|<r_2\}$: 
$$
\tilde{\pi}_{C_{f_\tau}}\circ R_{C_{f_\tau}}\circ\pi_{C_{f_\tau}}^{-1} \equiv \tilde{\psi}\circ\pi_{\tilde{\phi},\alpha,\tl t}\circ p_{\tilde{\phi},\alpha\pm},
$$
where $\{\pi_{\phi,\alpha,\tl t}\mid (\alpha,\phi)\in\mathfrak B_{\tilde U,\tilde t}\}$ is the same family of maps as in Lemma~\ref{Choose_r_h_lemma}, and both sides of the identity are multiple-valued functions.
\end{enumerate}
\end{lemma}
\begin{proof}
Let $h_1>\tilde h$ be the same as in Lemma~\ref{Choose_r_h_lemma} and Lemma~\ref{Choose_hr_tilde_lemma}, and Choose $h_2\in\bbR$ so that $\tilde h<h_2<h_1$.
According to Lemma~\ref{Choose_hr_tilde_lemma}, we can choose $\tilde{\mathcal U}$ so that $\check{\mathcal K}\subset\tilde{\mathcal U}$, and (i) holds. For $\tau\in\tilde{\mathcal U}$, let $C_{f_\tau}$ and $C_{f_\tau}^o$ be the canonical fundamental crescent and its fattening, constructed in Lemma~\ref{U_lemma}. 
Then it follows from Lemma~\ref{U_lemma} that for all $\tau\in\tilde{\mathcal U}$, the map $\tl\phi=\pi_{C_{f_\tau}}'(0)\cdot\phi\circ\pi_{C_{f_\tau}^o}^{-1}$ is defined in the domain $U_2$, which proves property~(iii). 

For $\tau=(3,\phi_\cH,\psi_\cH)\in\check{\mathcal K}$, where $\phi_\cH$ and $\psi_\cH$ are the same as in Lemma~\ref{Choose_t_lemma} and Lemma~\ref{Choose_r_h_lemma}, 
it follows from part (iv) of Lemma~\ref{U_lemma} that
\begin{equation}\label{phi_computation}
\pi_{C_{f_\tau}}'(0)\cdot\phi_\cH\circ\pi_{C_{f_\tau}^o}^{-1}=\pi_{C_{\cH_1}}'(0)\cdot\pi_{C_{\cH_1}^o}^{-1},
\end{equation}
where $\cH_1$ is a linear rescaling of a holomorphic pair from $\tilde{\mathcal K}$. Now 
Lemma~\ref{Choose_t_lemma} implies that $\pi_{C_{\cH_1}}'(0)\cdot\pi_{C_{\cH_1}^o}^{-1}\in\mathfrak B_{\tilde U,\tilde t}^3$, and because of property~(iii), by continuity considerations, shrinking $\tilde{\mathcal U}$, if necessary, we can make property (iv) hold for all $\tau\in\tilde{\mathcal U}$.

Since the constant $N$ is even, this implies that for $\tau\in\check{\mathcal K}$ the number $\pi_{C_{f_\tau}}'(0)$ 
is positive. Again, by continuity, shrinking $\tilde{\mathcal U}$, if necessary, we guarantee that property (ii) holds.

Now the composition $\tilde{\pi}_{C_{f_\tau}}\circ R_{C_{f_\tau}}\circ\pi_{C_{f_\tau}}^{-1}$ can be naturally represented as
$$
\tilde{\pi}_{C_{f_\tau}}\circ R_{C_{f_\tau}}\circ\pi_{C_{f_\tau}}^{-1} = g\circ p_{\phi\circ\pi_{C_{f_\tau}}^{-1},\alpha\pm},
$$
where $g$ is a conformal map on some sufficiently large domain. By Proposition~\ref{commut_prop} and property~(ii), we have 
$$
p_{\phi\circ\pi_{C_{f_\tau}}^{-1},\alpha\pm}= \frac{1}{p_{\alpha+}(\pi'_{C_{f_\tau}}(0))}\cdot p_{\tilde{\phi},\alpha\pm},
$$
hence
$$
\tilde{\pi}_{C_{f_\tau}}\circ R_{C_{f_\tau}}\circ\pi_{C_{f_\tau}}^{-1} = g_1\circ p_{\tilde{\phi},\alpha\pm},
$$
for some conformal map $g_1$. Now, if $f_\tau\in\hat{\mathcal K}$, then by Lemma~\ref{Choose_r_h_lemma}, we know that both sides of the above identity are defined in the domain $\{z\in\bbC\mid |\Re z|<1/2; |\Im z|<r_1\}$, and the map $g_1$ can be represented as
$$
g_1=\tilde{\psi}\circ\pi_{\tilde{\phi},\alpha,\tilde t},
$$
where $\tilde{\psi}\in\mathcal D_{h_1}$, for $h_1$ from Lemma~\ref{Choose_r_h_lemma}. Using continuity argument and possibly shrinking the neighborhood $\tilde{\mathcal U}$ again, we guarantee that the above representation of the map $g_1$ with $\tilde{\psi}\in\mathcal D_{h_2}$ holds for every $\tau\in\tilde{\mathcal U}$.
\end{proof}

\begin{definition}
We define the renormalization operator $\mathcal R_{\mathfrak P}\colon\tilde{\mathcal U}\to\mathfrak P_{\tilde U,\tilde t,\tilde h}$ by the relation 
$$\mathcal R_{\mathfrak P}(\alpha,\phi,\psi)=(\alpha,\tilde{\phi},\tilde{\psi})$$
where $\tilde{\phi}$ and $\tilde{\psi}$ are the same as in properties~\ref{phi_tilde_condit} and~\ref{psi_tilde_condit} of Lemma~\ref{composition_lemma}.
\end{definition}

In the following proposition we summarize some basic properties of the renormalization operator $\mathcal R_{\mathfrak P}$:

\begin{proposition}\label{basic_R_P_prop}
\begin{enumerate}[(i)]
\item The operator $\mathcal R_{\mathfrak P}$ is real-symmetric and analytic.

\item\label{Psi_map_property2} For every $\tau\in\tilde{\mathcal U}$, we have
$$
\mathcal \cren f_\tau=f_{\mathcal R_{\mathfrak P}\tau}, 
$$

\item\label{Psi_map_property}
For every $B\in\bbN$, let the map $\Psi\colon\mathcal I_B\to \check{\mathcal I}_B$ be defined by the relation
$$
\Psi(\zeta)=(3,\phi_{\cH_\zeta},\psi_{\cH_\zeta}),
$$
where $\phi_{\cH_\zeta}$ and $\psi_{\cH_\zeta}$ are the same as in Lemma~\ref{Choose_t_lemma} and Lemma~\ref{Choose_r_h_lemma} respectively. Then $\Psi$ is a homeomorphism between $\mathcal I_B$ and $\check{\mathcal I}_B$ that conjugates $\cR^N|_{\mathcal I_B}$ and $\mathcal R_{\mathfrak P}|_{\check{\mathcal I}_B}$. In particular, $\check{\mathcal I}_B$ is an invariant horseshoe for the operator $\cR_{\mathfrak P}$.

\item For every $\tau\in\tl\cU$ the differential $D_\tau {\cR}_{\bP}$ is a compact operator.
\end{enumerate}
\end{proposition}
\begin{proof}
Properties (i) and (ii) are evident from the construction of the operator $\cR_{\mathfrak P}$. 

In property (iii) we will first prove that $\Psi$ is a homeomorphism. Indeed, if $\zeta\in\mathcal I_B$ and $\tau=(3,\phi_{\cH_\zeta},\psi_{\cH_\zeta})=\Psi(\zeta)$, then $f_\tau$ together with the local coordinate $\phi_{\cH_\zeta}$ around the origin, completely determines the critical commuting pair $\cR\zeta\in\mathcal I_B$. Since according to Theorem~\ref{Comm_pair_attractor_theorem}, the operator $\cR$ is bijective on $\mathcal I_B$, the inverse map $\Psi^{-1}$ is defined. Continuity of $\Psi$ and its inverse follows from the definition of $\Psi$. Now, computation~(\ref{phi_computation}) and the fact that $\cH_1$ is a linear rescaling of a holomorphic pair from $\tilde{\mathcal K}$ (c.f. Lemma~\ref{U_lemma}, part~(iv)) yields the conjugacy statement of property~(iii). 

The proof of property~(iv) is analogous to the proof of (vi) from Proposition~\ref{cren_basic_proposition} and is based on the fact that according to Lemma~\ref{composition_lemma}, $\cR_{\mathfrak P}(\tilde{\mathcal U})\subset\mathfrak P_{U_3,\tilde t,h_2}$, for some domain $U_3\Supset\tl U$, and by 
Montel's theorem,
a bounded set in $T_{\cR_{\mathfrak P}\tau}\mathfrak P_{U_2,\tilde t,h_2}$ is pre-compact in $T_{\cR_{\mathfrak P}\tau}\mathfrak P_{\tilde U,\tilde t,\tilde h}$.
\end{proof}

Now we are ready to prove the hyperbolicity results for the renormalization operator $\mathcal R_{\mathfrak P}$.

\begin{definition}
For a complex number $\alpha\in\bbC$ close to $3$, we denote by $\mathcal R_{\mathfrak P,\alpha}$ the restriction of the operator $\mathcal R_{\mathfrak P}$ on the set $\tilde{\mathcal U}\cap\mathfrak P_{\tilde U,\tilde t,\tilde h}^\alpha$.
\end{definition}

\section{Hyperbolicity of renormalization for odd integer $\alpha$}
\label{sec:hyperb}

The goal of this section is to establish the following theorem:
\begin{theorem}\label{hyperb_R_P_theorem}
For every $B\in\bbN$, the set $\check{\mathcal I}_B$ is a uniformly hyperbolic invariant set for the operator $\mathcal R_{\mathfrak P,3}$ with a complex one-dimensional unstable direction. Moreover, for every $\tau\in\check{\mathcal I}_B$ and the corresponding local stable manifold $W^s(\tau)$, the set $W^s(\tau)\cap (\mathfrak P_{\tilde U,\tilde t,\tilde h}^3)^\bbR$ consists of all $\omega\in (\mathfrak P_{\tilde U,\tilde t,\tilde h}^3)^\bbR$ that are sufficiently close to $\tau$ and such that $\rho(f_\tau)=\rho(f_\omega)$.
\end{theorem}

Before giving a proof of Theorem~\ref{hyperb_R_P_theorem}, we formulate its immediate corollary:

\begin{corollary}\label{main3_extended}
 For every positive integer $B>0$, there exists an open interval $I=I(B)\subset\bbR$, such that $3\in I$ and for every $\alpha\in I$, the operator $\mathcal R_{\mathfrak P,\alpha}\colon \tilde{\mathcal U}\cap\mathfrak P_{\tilde U,\tilde t,\tilde h}^\alpha \to \mathfrak P_{\tilde U,\tilde t,\tilde h}^\alpha$ has a hyperbolic horseshoe attractor $\check{\mathcal I}_B^\alpha\subset (\mathfrak P_{\tilde U,\tilde t,\tilde h}^\alpha)^\bbR$ of type bounded by $B$. 
 The action of $\mathcal R_{\mathfrak P,\alpha}$ on $\check{\mathcal I}_B^\alpha$ is topologically conjugate to the shift $\sigma^N\colon\Sigma_B\to\Sigma_B$:
 $$
 \kappa_\alpha\circ\mathcal R_{\mathfrak P,\alpha}\circ \kappa_\alpha^{-1}=\sigma^N,
 $$
 and if 
 $$
 \tau=\kappa^{-1}_\alpha(\dots,r_{-k},\dots,r_{-1},r_0,r_1,\dots,r_k,\dots),
 $$
 then 
 $$
 \rho(f_\tau)=[r_0,r_1,\dots,r_k,\dots].
 $$
 Moreover, for every $\tau\in\check{\mathcal I}_B^\alpha$, the corresponding local stable manifold $W^s(\tau)$ has complex codimension~$1$, and if $\omega\in W^s(\tau)\cap (\mathfrak P_{\tilde U,\tilde t,\tilde h}^\alpha)^\bbR$, then $\rho(f_\tau)=\rho(f_\omega)$.
\end{corollary}
\begin{remark}
In Theorem~\ref{Global_attractor_theorem_expanded} we will strengthen Corollary~\ref{main3_extended} by showing that for every $\tau\in\check{\mathcal I}_B^\alpha$, the set $W^s(\tau)\cap (\mathfrak P_{\tilde U,\tilde t,\tilde h}^\alpha)^\bbR$ consists of all $\omega\in (\mathfrak P_{\tilde U,\tilde t,\tilde h}^\alpha)^\bbR$ that are sufficiently close to $\tau$ and such that $\rho(f_\tau)=\rho(f_\omega)$.
\end{remark}

\begin{proof}[Proof of Corollary~\ref{main3_extended}]
For all $\alpha$ sufficiently close to $3$, the operator $\mathcal R_{\mathfrak P,\alpha}$ can be thought of as an operator acting on a neighborhood of $\mathfrak P_{\tilde U,\tilde t,\tilde h}^3$ by the correspondence $(3,\phi,\psi)\mapsto(3,\tl\phi,\tl\psi)$, where $\tl\phi$ and $\tl\psi$ are such that $\mathcal R_{\mathfrak P,\alpha}(\alpha,\phi,\psi)=(\alpha,\tl\phi,\tl\psi)$. Then this operator is a real-symmetric analytic perturbation of $\mathcal R_{\mathfrak P,3}$ in a small neighborhood of the attractor $\check{\mathcal I}_B$ of type bounded by $B$. Since this is a hyperbolic horseshoe attractor, it survives for $\alpha$ close to $3$.

For every $\tau\in\check{\mathcal I}_B^\alpha$, the corresponding local stable manifold $W^s(\tau)$ must have complex codimension~$1$, and if $\omega\in (\mathfrak P_{\tilde U,\tilde t,\tilde h}^\alpha)^\bbR$ is such that the combinatorics of $f_\tau$ and $f_\omega$ are asymptotically different, then because of part~\ref{Psi_map_property2} of Proposition~\ref{basic_R_P_prop} the point $\omega$ does not belong to $W^s(\tau)$. 
\end{proof}

We will prove Theorem~\ref{hyperb_R_P_theorem} in the end of this section. Now let us start by formulating a theorem which is a special case of Theorem~\ref{cren_hyperb_theorem} and Theorem~\ref{hyperb_1_theorem}:

\begin{theorem}\label{hyperb_special_case_theorem} 
For any positive integer $B>0$, the set $\hat{\mathcal I}_B\subset\hat{\mathfrak C}_r^3$ is a uniformly hyperbolic invariant set for the operator $\hren$ with a one-dimensional unstable direction. Moreover, 
for any $f\in\hat{\mathcal I}$ and the corresponding local stable manifold $W^s(f)\subset \hat{\mathfrak C}_r^3$, the intersection $W^s(f)\cap(\hat{\mathfrak C}_r^3)^\bbR$ consists of all critical circle maps $g$ that are sufficiently close to $f$ and such that $\rho(g)=\rho(f)$.
\end{theorem}

We proceed with the following construction: according to part (iii) of Lemma~\ref{U_lemma}, there exists a real number $R>0$, such that for every $g\in\mathcal U$, the map $\pi_{C_g^o}^{-1}$ is defined in the disk $\bbD_{2R}$, and 
\begin{equation}\label{D_R_inclusion_eq}
\pi_{C_g^o}^{-1}(\bbD_R)\subset\bbD_{2R/3}.
\end{equation}

We denote by $\kL$ the space that consists of all bounded analytic functions $\phi\colon\bbD_{R}\to\bbC$ that are continuous up to the boundary, and satisfy the properties $\phi(0)=0$, and $\phi'(0)=1$. The space $\kL$ equipped with the sup-norm forms an affine complex Banach space ($\kL-\text{Id}$ is a complex Banach space).

\begin{definition}
For each cylinder renormalizable map $g\in\mathcal U$, we define an affine operator $A_g\colon\kL\to\kL$ by the formula
$$
A_g(\phi)=\pi_{C_g}'(0)\cdot\phi\circ\pi_{C_g^o}^{-1},
$$
where $C_g$ is the canonical fundamental crescent provided by Lemma~\ref{U_lemma}.
\end{definition}
Condition~(\ref{D_R_inclusion_eq}) implies that the operator $A_g$ is well-defined for every $g\in\mathcal U$, and it follows from Lemma \ref{U_lemma} that $A_g$ analytically depends on the parameter $g$.

\begin{definition}\label{R1_def}
We set $\cV=\cU\cap\hat{\mathfrak C}_{\tilde r}^3$ and we define the operator ${\tl\cR}\colon\kL\times\cV\to\kL\times\hat{\mathfrak C}_{\tilde r}^3$ as the skew product
$$
{\tl\cR}(\phi,g)=(A_g(\phi),\hren\, g).
$$
\end{definition}

The following lemma is an immediately corollary from~(\ref{D_R_inclusion_eq}) and Proposition~\ref{cren_basic_proposition} part (vi): 
\begin{lemma}
For every $\kappa\in \kL\times\cV$ the differential $D_\kappa \tl\cR$ is a compact operator.
\end{lemma}

We denote by $\mathcal J\subset \kL\times\cV$ the image of the set $\tilde{\mathcal K}$ under the map
\begin{equation}\label{include_eq}
\mathcal H\mapsto (\pi_{C_{\mathcal H}}'(0)\cdot\pi_{C_{\mathcal H}^o}^{-1}, f_{\mathcal H}).
\end{equation}
Note that the projection of the set $\mathcal J$ onto the second coordinate is the set $\hat{\mathcal K}$. Similarly, for every positive integer $B>0$, we denote by $\mathcal J_B\subset \kL\times\cV$ the image of $\tilde{\mathcal I}_B$ under the map~(\ref{include_eq}). 

\noindent
\begin{lemma}
For every $B\in\bbN$, the set $\mathcal J_B$ is a forward invariant set for the operator $\tl\cR$.
\end{lemma}
\begin{proof}
For $\zeta\in\mathcal I_B$ and the corresponding $\cH=\cH_\zeta\in\tl{\mathcal I}_B$, let $\cH_1\subset\tilde{\mathcal I}_B$ be the holomorphic pair $\cH_1=\cH_{\cR^N\zeta}$. Then for $(\pi_{C_{\mathcal H}}'(0)\cdot\pi_{C_{\mathcal H}^o}^{-1}, f_{\mathcal H})\in\mathcal J_B$, we have 
$$
\tl\cR(\pi_{C_{\mathcal H}}'(0)\cdot\pi_{C_{\mathcal H}^o}^{-1}, f_{\mathcal H}) = (A_{f_\cH}(\pi_{C_{\mathcal H}}'(0)\cdot\pi_{C_{\mathcal H}^o}^{-1}),f_{\cH_1}) = 
$$
$$
=(\pi_{C_{\mathcal H_1}}'(0)\cdot\pi_{C_{\mathcal H_1}^o}^{-1}, f_{\mathcal H_1})\subset\mathcal J_B.
$$
(We notice that the last equality follows from~(\ref{phi_computation}).) 
\end{proof}

For maps $\phi_0\in\kL$ and $g\in\cV$, one can consider a sequence of iterates $\{\tl\cR^k(\phi_0,g)\}$. Restricting to the first coordinate, this corresponds to the sequence of functions $\phi_0$, $\phi_1=A_g(\phi_0)$, $\phi_2=A_{\hren g}(\phi_1)$, $\phi_3= A_{\hren^2 g}(\phi_2),\dots, \phi_{k+1}= A_{\hren^k g}(\phi_k),\ldots$.
For simplicity of notation we will write that $\phi_k=A_g^k(\phi_0)$, provided that $\phi_k$ is defined.

\begin{lemma}\label{Koebe_appl_lemma}
There exists a positive real number $\nu<1$ 
and a positive integer $M>0$, such that for any $m\ge M$ times cylinder renormalizable $g\in\cV$ and $\phi_1,\phi_2\in\kL$,  we have $\|A_g^m(\phi_1)-A_g^m(\phi_2)\|\le\nu\|\phi_1-\phi_2\|$, where $\|\phi\|=\sup_{z\in\bbD_R}|\phi(z)|$. 
\end{lemma}
\begin{proof}
Assume that $\|\phi_1-\phi_2\|=cR^2$, for some positive real number $c>0$. Then for all $z\in\bbD_{R}$ we have the inequality
$$
|\phi_1(z)-\phi_2(z)|\le c|z|^2.
$$

\noindent
We notice that if $g\in\cV$ is $n$-times cylinder renormalizable, then there exists a sequence $C_n\Subset\ldots\Subset C_1\subset V_{\tl r}$ of $n$ nested fundamental crescents for $g$, such that $\cren^m g= \cR_{C_m}g$, and
$$
A_g^m(\phi)=\pi_{C_m}'(0)\cdot\phi\circ\pi_{C_m^o}^{-1},
$$
for $m\le n$. By part (iii) of Lemma~\ref{U_lemma} there exists a positive integer $M>0$, such that for all $m\ge M$, if $g\in\cV$ is $m$-times cylinder renormalizable, then for the corresponding fundamental crescent $C_m$ we have $|(\pi_{C_m}^{-1})'(0)|<1/17$.

Since $\pi_{C_m}^{-1}$ is defined in the disk $\bbD_{2R}$, then again from Koebe Distortion Theorem it follows that
$$
|\pi_{C_m^o}^{-1}(z)|\le 4|(\pi_{C_m}^{-1})'(0)|\cdot|z|,
$$
for all $z\in\bbD_{R}$. Combining the above estimates, for all $z\in\bbD_{R}$ we obtain

$$
|[A_g^m(\phi_1)](z)-[A_g^m(\phi_2)](z)|\le |\pi_{C_m}'(0)|\cdot|\phi_1(\pi_{C_m^o}^{-1}(z))- \phi_2(\pi_{C_m^o}^{-1}(z))|
$$
$$
\le |\pi_{C_m}'(0)|\cdot c\cdot 16 |(\pi_{C_m}^{-1})'(0)|^2\cdot|z|^2\le \frac{16}{17}c(R)^2.
$$
Thus,
$$
\|A_g^m(\phi_1)-A_g^m(\phi_2)\|\le\frac{16}{17}\|\phi_1-\phi_2\|.
$$
\end{proof}

\begin{lemma}\label{tlR_hyperb_lemma}
For every $B\in\bbN$, the set $\mathcal J_B$ is a uniformly hyperbolic invariant set for the operator $\tl\cR$ with a one-dimensional unstable direction. Moreover, the stable manifold of every point $(\phi,f)\in\mathcal J_B$ is $\mathbf L\times W^s(f)$, where $W^s(f)$ is the stable manifold of $f$ for the cylinder renormalization $\hren$.
\end{lemma}
\begin{proof}

Since $\mathbf L$ is an affine Banach manifold, the tangent space to every its point can be naturally identified with the same Banach space, which we denote by $\mathbf K$.
Let $E^s\subset T_{\hat{\mathcal I}_B}\hat{\mathfrak C}_{\tl r}^3$ be the stable bundle of the hyperbolic set $\hat{\mathcal I}_B$ for the operator $D\hren$. Then the set 
$$
\tl E^s=\mathbf K\times E^s
$$ 
is an invariant subbundle of the tangent bundle $T_{\mathcal J_B}(\mathbf L\times\hat{\mathfrak C}_{\tl r}^3)$ for the operator $D\tl\cR$. It follows from Definition~\ref{R1_def} that the action of $D\tl\cR$ on the second factor of $\tl E^s$ is independent from the first factor. Hence, the differential $D\tl\cR$ uniformly contracts in the second factor of $\tl E^s$.

Now by compactness of $\mathcal J_B$, the derivative of $A_g(\phi)$ with respect to $g$ is uniformly bounded over all $(\phi,g)\in\mathcal J_B$, and by Lemma~\ref{Koebe_appl_lemma}, the derivative of $A_g(\phi)$ with respect to $\phi$ uniformly contracts. Together with uniform contraction in the second factor, this implies that the differential $D\tl\cR$ contracts in the first factor of $\tl E^s$ as well. Thus, $\tl E^s$ is the stable subbundle for the differential $D\tl\cR$.

In a similar way one shows that if $E^u$ 
is the unstable subbundle (in the sense of Definition~\ref{hyperb_def}) of the hyperbolic set $\hat{\mathcal I}_B$ for the operator $D\hren$, then $\tl E^u=\mathbf 0\times E^u$ 
is the unstable bundle for the operator $D\tl\cR$. Since according to Theorem~\ref{hyperb_special_case_theorem}, the fibers of $E^u$ are one-dimensional, so are the fibers of $\tl E^u$. Finally, by dimension count we see that there are no neutral directions, hence the set $\mathcal J_B$ is uniformly hyperbolic.
\end{proof}

\begin{proof}[Proof of Theorem~\ref{hyperb_R_P_theorem}]
The analytic map $\Phi\colon\mathfrak P_{\tl U,\tl t,\tl h}^3 \to \mathbf L\times\hat{\mathfrak C}_{\tl r}^3$ defined by the relation
$$
\Phi\colon\tau=(3,\phi,\psi)\mapsto (\phi, f_\tau),
$$
is injective, since the maps $\phi$ and $f_\tau$ determine the map $\psi$ (c.f. Lemma~\ref{Choose_r_h_lemma}).
Now we notice that the map $\Phi$ 
maps $\check{\mathcal I}_B$ to ${\mathcal J}_B$ and conjugates the operator $\mathcal R_{\mathfrak P,3}$ with the operator ${\tl\cR}|_{\Phi(\tl{\mathcal U})}$ on the neighborhood of the invariant set $\check{\mathcal I}_B$. Then Lemma~\ref{tlR_hyperb_lemma} and Theorem~\ref{hyperb_special_case_theorem} imply the statement of the theorem.
\end{proof}

\section{Proofs of Theorems \ref{main1}-\ref{main3}}
\label{sec:main_proofs}

Now we are ready to prove Theorems~\ref{main1},~\ref{main2} and~\ref{main3}. As mentioned in the Introduction, we prove these theorems with the parameter $k$ set to $k=1$. For $k>1$ the proofs are identical.

\begin{proof}[Proof of Theorem~\ref{main1}]
We set $\mathfrak N^{3}=\mathfrak C_{\tilde r}$ and define $\cren$ as in Definition~\ref{cren_def}. 
Then $\mathfrak C^\alpha=\mathfrak C_{\tilde r}^\alpha$, for all complex $\alpha$ in a suitably small neighborhood of $3$, and Theorem~\ref{main1} follows from Proposition~\ref{cren_basic_proposition}.
\end{proof}

\begin{proof}[Proof of Theorem~\ref{main2}]
We set $\mathfrak M^3$ to be an open subset of $\mathfrak P_{\tilde U,\tilde t,\tilde h}^3\cap\check{\mathcal U}$, where $\check{\mathcal U}$ is the same as in Lemma~\ref{Choose_hr_tilde_lemma}, such that the closure of $\mathfrak M^3$ in $\mathfrak P_{\tilde U,\tilde t,\tilde h}^3$ is contained in $\check{\mathcal U}$ and $\check{\mathcal K}\subset\mathfrak M^3$. 
Then for all  $\alpha\in \CC$ which lie in a small open neighborhood of $3$, the map $i_\alpha$ can be defined by
$$
i_\alpha\colon(\phi,\psi)\mapsto f_\tau,
$$
where $\tau=(\alpha,\phi,\psi)\in\mathfrak P_{\tilde U,\tilde t,\tilde h}^\alpha$. It follows from Lemma~\ref{Choose_hr_tilde_lemma} that the map $i_\alpha$ is analytic and analytically depends on $\alpha$. Finally, for all $\alpha$ close to $3$, we define $\cR_\alpha(\phi,\psi)=(\tilde{\phi},\tilde{\psi})$, where $\tilde{\phi}$ and $\tilde{\psi}$ are such that $\cR_{\mathfrak P}(\alpha,\phi,\psi)=(\alpha,\tilde{\phi},\tilde{\psi})$, whenever $\cR_{\mathfrak P}(\alpha,\phi,\psi)$ is defined.

The required properties of the operator $\cR_\alpha$ follow from Proposition~\ref{basic_R_P_prop} and Theorem~\ref{hyperb_R_P_theorem}. Analyticity of the map $i_\alpha$ follows from Lemma~\ref{f_tau_analyt_lemma}.
\end{proof}

\begin{proof}[Proof of Theorem~\ref{main3}] 
The proof immediately follows from Corollary~\ref{main3_extended}.
\end{proof}

\section{Global attractor, universality and rigidity}
\label{sec:global}

\subsection{Generalized critical commuting pairs}
We start by generalizing the notion of a critical commuting pair to the case of an arbitrary real critical exponent. In order to make this generalization, we need to substitute the commutation relation with a more general condition. The following definition generalizes the class $\mathcal A_r$ from Definition~\ref{A_r_com_pair_def}.

\begin{definition}\label{A_r_alpha_def}
For positive real numbers $r,\alpha>0$, the class $\mathcal A_r^\alpha$ consists of all pairs of real-symmetric analytic maps $\zeta=(\eta,\xi)$, such that the following holds:

\begin{enumerate}[(i)]
\item the maps $\eta$ and $\xi$ are analytic on the sets $N_r([0,1])\setminus(-\infty,0]$ and $N_{r|\eta(0)|}(I_\xi)\setminus[0,+\infty)$ respectively, where 
$I_\xi=[b,0]\subset(-\infty,0]$, and $N_r(S)\subset\bbC$ stands for the $r$-neighborhood of a set $S$ in $\bbC$; furthermore, $\eta$ and $\xi$ extend continuously to the boundary of their corresponding domains, (here we distinguish the upper and lower boundaries of slits);

\item the maps $\eta$ and $\xi$ can be locally represented as
\begin{equation*}\label{eta_xi_decomposition_eq2}
\eta=\psi_\eta\circ p_{\alpha+}\circ\phi_\zeta\quad\text{ and }\quad \xi=\psi_\xi\circ p_{\alpha-}\circ\phi_\zeta,
\end{equation*}
where $\psi_\eta$, $\psi_\xi$ and $\phi_\zeta$ are real-symmetric conformal maps in a neighborhood of the origin and $\phi_\zeta(0)=0$, $\phi_\zeta'(0)=1$;

\item\label{Generalized_cr_com_p_property_rho2} 
the compositions $\nu^+\equiv\xi\circ\eta$ and $\nu^-\equiv\eta\circ\xi$ are defined in some neighborhood of the origin and can be represented in it as 
$$
\nu^+=\rho_\zeta\circ p_{\alpha+}\circ\phi_\zeta\quad\text{ and }\quad \nu^-=\rho_\zeta\circ p_{\alpha-}\circ\phi_\zeta,
$$
where $\rho_\zeta$ is a real-symmetric conformal map in a neighborhood of the origin;

\item for every $z\in(0,1]$, we have $\eta'(z)>0$ and for every $z\in I_\xi\setminus\{0\}$, we have $\xi'(z)>0$; furthermore, $\eta(0)=b<0$, $\xi(0)=1$ and $\eta(\xi(0))=\xi(\eta(0))\in I_\eta$.
\end{enumerate}
\end{definition}

 If a pair of maps $\zeta=(\eta,\xi)$ belongs to $\mathcal A_r^\alpha$, for some $r,\alpha>0$, then we say that $\zeta$ is a \textit{critical commuting pair with critical exponent} $\alpha$.
 We also note that $\mathcal A_r\subset\mathcal A_r^3$. 
 
\begin{definition}
For a positive real $r>0$ and a set $J\subset (0,+\infty)$, we define
$$
\mathcal A_r^J=\bigcup_{\alpha\in J}\mathcal A_r^\alpha.
$$
\end{definition}

Similarly to the case with the class $\mathcal A_r$, the map~(\ref{norm_inclusion_eq}) induces a metric on $\mathcal A_r^J$ from the sup-norm on $$\mathfrak D(N_r([0,1])\setminus(-\infty,0]))\times \mathfrak D(N_r([0,1])\setminus(-\infty,0])).$$ We will denote this metric by $\dist_r(\cdot,\cdot)$, since it coincides with the previously defined metric on $\mathcal A_r$.

Renormalization of pairs from $\mathcal A_r^{J}$ is defined in exactly the same way as renormalization of ordinary critical commuting pairs (c.f. Definition~\ref{renormalization_def}) and is determined by the dynamics on the real line. In particular, the $n$-th prerenormalization $p\cR^n\zeta$ of $\zeta=(\eta,\xi)\in\mathcal A_r^{J}$ can be viewed as a pair whose elements are certain compositions of $\eta$ and $\xi$. We will say that the domains of the maps that constitute $p\cR^n\zeta$, are the maximal domains in $\bbC$, where the corresponding compositions are defined.

\begin{definition}\label{A_alpha_def}
We let $\mathcal A^\alpha$ (or $\mathcal A^J$) denote the set of all real-analytic pairs $(\eta,\xi)$, such that an analytic extension of $(\eta,\xi)$ belongs to $\mathcal A_r^\alpha$ (or $\mathcal A_r^J$), for some $r>0$. 
\end{definition}
We think of critical commuting pairs $\zeta$ from $\mathcal A^\alpha$ or $\mathcal A^J$ and their renormalizations $\cR\zeta$ as pairs of maps defined only on the real line, though having analytic extensions to some domains in $\bbC$.

It is convenient for our purposes to introduce a smaller class of critical commuting pairs:
\begin{definition}\label{Gccp_def}
For positive real numbers $r,\alpha>0$, the class $\mathcal B_r^\alpha\subset \mathcal A_r^\alpha$ consists of all $\zeta=(\eta,\xi)\in\mathcal A_r^\alpha$, such that the following holds:
\begin{enumerate}[(i)]
\item the maps $\eta$ and $\xi$ can be represented as
\begin{equation*} 
\eta=\psi_\eta\circ p_{\alpha+}\quad\text{ and }\quad \xi=\psi_\xi\circ p_{\alpha-},
\end{equation*}
where $\psi_\eta$ and $\psi_\xi$ are conformal diffeomorphisms of $p_{\alpha+}(N_r([0,1])\setminus(-\infty,0))$ and $p_{\alpha-}(N_{r|\eta(0)|}(I_\xi)\setminus(0,+\infty))$ respectively;

\item\label{Generalized_cr_com_p_property_rho} the compositions $\nu^+\equiv\xi\circ\eta$ and $\nu^-\equiv\eta\circ\xi$ are defined in the domains $\bbD_{\hat r}\setminus(-\infty,0)$ and $\bbD_{\hat r}\setminus(0,+\infty)$ respectively, where $\hat r=\min(r,r|\eta(0)|)$, and there exists a conformal diffeomorphism $\rho_\zeta\colon\bbD_{\hat r^\alpha}\to\bbC$, such that $\nu^+=\rho_\zeta\circ p_{\alpha+}$ and $\nu^-=\rho_\zeta\circ p_{\alpha-}$;
\end{enumerate}
\end{definition}

\begin{definition}\label{B_alpha_def}
We let $\mathcal B^\alpha$ (or $\mathcal B^J$) denote the set of all real-analytic pairs $(\eta,\xi)$, such that an analytic extension of $(\eta,\xi)$ belongs to $\mathcal B_r^\alpha$ (or $\mathcal B_r^J$), for some $r>0$. 
\end{definition}

The following theorem is an expanded version of Theorem~\ref{main4}. As before, we define $\Sigma_B=\{1,\dots,B\}^\bbZ$.

\begin{theorem}[{\bf Global renormalization attractor}]\label{Global_attractor_theorem_expanded}
For every $k,B\in\bbN$, there exists an open interval $J=J(k,B)\subset\bbR$ and a positive real number $r=r(k)$, such that $2k+1\in J$, and for every $\alpha\in J$, 
there exists an $\cR$-invariant set $\mathcal I_B^\alpha\subset\mathcal B_r^\alpha$ with the following properties. 
\begin{enumerate}[(i)]
\item The action of $\cR$ on $\mathcal I_B^\alpha$ is topologically conjugate to the shift $\sigma\colon\Sigma_B\to\Sigma_B$:
$$
\iota_\alpha\circ\mathcal R\circ \iota_\alpha^{-1}=\sigma,
$$
and if 
$$
\zeta=\iota^{-1}_\alpha(\dots,r_{-k},\dots,r_{-1},r_0,r_1,\dots,r_k,\dots),
$$
then 
$$
\rho(\zeta)=[r_0,r_1,\dots,r_k,\dots].
$$

\item For every $\zeta\in\mathcal A^\alpha$ with an irrational rotation number of type bounded by $B$, there exists $M\in\bbN$, such that for all $m\ge M$ the renormalizations $\cR^m\zeta$ belong to $\mathcal A_r^\alpha$ and for every $\zeta'\in\mathcal I_B^\alpha$ with $\rho(\zeta)=\rho(\zeta')$, we have
\begin{equation}\label{pair_exp_convergence_eq}
\dist_r(\cR^m\zeta,\cR^m\zeta')\le C\lambda^m,
\end{equation}
for some constants $C>0$, $\lambda\in(0,1)$ that depend only on $B$ and $\alpha$.
\end{enumerate}
\end{theorem}

We give a proof of Theorem~\ref{Global_attractor_theorem_expanded} in Section~\ref{Glob_Attr_proof_sec}. We note that Theorem~\ref{main4} is an immediate corollary of Theorem~\ref{Global_attractor_theorem_expanded}.

\subsection{Generalized holomorphic commuting pairs and complex bounds}

\begin{definition}\label{Holo_pairs_new_def}
Let $\zeta=(\eta,\xi)$ be either an element of some class $\mathcal B^{\alpha}$ or a linear rescaling of such an element. 
We say that $\zeta$ extends to a \textit{holomorphic commuting pair with critical exponent} $\alpha$, if there exist domains $U, V, D,\Delta\subset\bbC$ that satisfy the same conditions as the corresponding domains in the definition of a holomorphic commuting pair (c.f. Section~\ref{holo_pairs_sec}), and the following properties hold:

\begin{enumerate}[(i)]
\item\label{holo_pair_prop_1} $\eta$ and $\xi$ extend to analytic maps $\hat{\eta}\colon (U\cup D)\setminus(-\infty,0]\to\bbC$ and $\hat{\xi}\colon (V\cup D)\setminus[0,+\infty)\to\bbC$ respectively and 
\begin{equation*} 
\hat\eta=\psi_\eta\circ p_{\alpha+}\quad\text{ and }\quad \hat\xi=\psi_\xi\circ p_{\alpha-},
\end{equation*}
where $\psi_\eta$ and $\psi_\xi$ are real-symmetric conformal diffeomorphisms of $p_{\alpha+}(U\cup D\setminus(-\infty,0))$ and $p_{\alpha-}(V\cup D\setminus(0,+\infty))$ respectively;

\item $\hat{\eta}|_U$ is a conformal diffeomorphism of $U$ onto $(\Delta\setminus\bbR)\cup\eta(I_U)$, where $I_U=U\cap\bbR$, and $\hat{\xi}|_V$ is a conformal diffeomorphism of $V$ onto $(\Delta\setminus\bbR)\cup\eta(I_V)$, where $I_V=V\cap\bbR$;

\item the compositions $\nu^+=\hat\xi\circ\hat\eta$ and $\nu^-=\hat\eta\circ\hat\xi$ are defined on the sets $D\setminus(-\infty,0)$ and $D\setminus(0,+\infty)$ respectively and can be represented in these sets as $\nu^+=\psi\circ p_{\alpha+}$ and $\nu^-=\psi\circ p_{\alpha-}$, where $\psi$ is some conformal real-symmetric map; furthermore, 
we have 
$\bbH\cap\nu^+(D\setminus(-\infty,0))=\bbH\cap\nu^-(D\setminus(0,+\infty)) = \bbH\cap\Delta$;
\end{enumerate}
\end{definition}

We let $\mathbf H^\alpha$ denote the space of all holomorphic commuting pairs with critical exponent $\alpha$. It follows from the above definition that $\mathbf H^3\subset\mathbf H$, where $\mathbf H$ is the same as in Section~\ref{holo_pairs_sec}. 

For a subset $J\subset (2,+\infty)$,  
we let $\mathbf H^{J}$ be the space of all holomorphic commuting pairs with critical exponent $\alpha$ in the set $J$. This space can be equipped with the Carath\'eodory convergence in exactly the same way as the space of ordinary holomorphic pairs $\mathbf H$, namely, by viewing an element of $\mathbf H^J$ as three triples $(U,\xi(0),\eta), (V,\eta(0),\xi), (D,0,\nu^+)$.
Similarly, we define the modulus of a holomorphic pair $\cH=(\eta,\xi)\in \mathbf H^{J}$ in the same way as for an ordinary holomorphic pair from $\mathbf H$. We denote the modulus of $\cH$ by $\mod(\cH)$, and we say that the domain $\Delta$ in the above definition is the range of a holomorphic pair $\cH$, and $\zeta_\cH\in\mathcal B^J$ is the commuting pair underlying $\cH$. Finally, we define the prerenormalization $p\cR(\cH)$ and the renormalization $\cR(\cH)$ of a holomorphic pair $\cH$ with critical exponent $\alpha$ in the same way as we defined renormalization of ordinary holomorphic pairs. In particular, we note that the ranges of $\cH$ and $p\cR(\cH)$ are the same.

\begin{definition}
For $\mu\in (0,1)$ and a closed set $J\subset (2,+\infty)$, 
we let $\mathbf H^{J}(\mu)\subset \mathbf H^{J}$ be the space of all holomorphic commuting pairs $\cH=(\eta,\xi)\in \mathbf H^{J}$ with corresponding domains $D_\cH, U_\cH, V_\cH, \Delta_\cH$, such that all properties (i)-(v) from Definition~\ref{H_mu_def} hold.
\end{definition}

\begin{lemma}\label{H_mu_compactness_lemma_extended}
For each $\mu\in(0,1)$ and a closed set $J\subset (2,+\infty)$, 
the space $\hol^{J}(\mu)$ is sequentially compact. Furthermore, if a sequence of holomorphic pairs $\{\cH_k\mid k\in\bbN\}\subset\hol^J(\mu)$ with corresponding critical exponents $\alpha_k$ converges to a holomorphic pair with critical exponent $\alpha$, then $\lim_{k\to\infty}\alpha_k=\alpha$.
\end{lemma}
\begin{proof}
Exactly the same argument as in the proof of Lemma~\ref{bounds compactness} (c.f. Lemma 2.17 from~\cite{Yamp-towers}) shows that every sequence of holomorphic pairs from $\hol^J(\mu)$ has a subsequence that converges to a holomorphic pair that satisfies properties (i)-(v) from Definition~\ref{H_mu_def}. In order to complete the proof, we need to show the convergence of critical exponents.

Consider a sequence $\cH_1,\cH_2,\dots\in\hol^J(\mu)$ of holomorphic pairs with corresponding critical exponents $\alpha_1,\alpha_2,\dots$ and let $U_k,V_k,D_k$ denote the corresponding domains of the maps that constitute the holomorphic pair $\cH_k$. We note that because of property~\ref{hol_mu_4} from Definition~\ref{H_mu_def}, Carath\'eodory convergence of these domains is equivalent to Hausforff convergence of $\overline{U_k\cap\bbH}$, $\overline{V_k\cap\bbH}$ and $\overline{D_k\cap\bbH}$. Furthermore, properties~\ref{hol_mu_2} and~\ref{hol_mu_4} from Definition~\ref{H_mu_def} imply that all domains $D_k$ contain a disk of some fixed positive radius, centered at zero. Then it follows from property~\ref{holo_pair_prop_1} of Definition~\ref{Holo_pairs_new_def} that all boundaries $\partial U_k$ contain a straight line segment of fixed positive length with one endpoint at $0$, forming angle $\pi/\alpha_k$ with the ray $[0,+\infty)$. Now convergence of the domains $U_k$ implies convergence of the critical exponents $\alpha_k$.
\end{proof}

The following Theorem was proved by the second author in \cite{Ya1} for critical commuting pairs in the Epstein class with $\alpha=3$, the proof was later generalized in \cite{FM2} to remove the Epstein class condition. The proof extends to the case of a general $\alpha>2$ {\it mutatis mutandis}.

\begin{theorem}[{\bf Complex bounds}]
\label{complex bounds_extended}
For every bounded set $J\subset[2,+\infty)$, there exists a constant $\mu=\mu(J)>0$  such that
the following holds. For every positive real number $r>0$ and every pre-compact family $S\subset\mathcal B_r^{J}$ of critical commuting pairs, there exists $K=K(r, S)\in\bbN$ such that if $\zeta\in S$ is
a $2n$ times renormalizable commuting pair, where $n\geq K$, then
$p\cR^n\zeta$ restricts to a holomorphic commuting pair $\cH_n$ with range $\Delta_n\subset N_r(I_\eta)\cup N_r(I_\xi)$. 
Furthermore, the range $\Delta_n$ is a Euclidean disk,  and the appropriate affine rescaling of $\cH_n$ is in $\hol^J(\mu)$.
\end{theorem}

\subsection{The maps between $\mathcal B_r^\alpha$ and $\mathfrak P_{\tl U,\tl t,\tl h}^\alpha$}

Let $\tilde{\mathcal U}_0$ denote the set of all $\tau\in\tilde{\mathcal U}\cup (\mathfrak P_{\tilde U,\tilde t,\tilde h})^\bbR$ for which $\cR_{\mathfrak P}^2(\tau)$ is defined. Then this implies that the rotation number of the generalized critical circle map $f_\tau$ has at least $2N$ elements in its continued fraction expansion. Let $p\cR^m f_\tau=(f_\tau^{q_{m+1}},f_\tau^{q_m})$ be the $m$-th prerenormalization of the circle map $f_\tau$ considered in the sense of commuting pairs (c.f. Definition~\ref{renorm_f_def}). Then one can consider a map $\Phi\colon \tilde{\mathcal U}_0 \to\mathcal B^{(0,+\infty)}$ defined by
$$
\Phi\colon \tau=(\alpha,\phi,\psi)\mapsto (h\circ\phi\circ f_\tau^{q_{2N}}\circ\phi^{-1}\circ h^{-1},h\circ\phi\circ f_\tau^{q_{2N-1}}\circ\phi^{-1}\circ h^{-1}),
$$
where $h(x)=x/\phi(f_\tau^{q_{2N-1}}(\phi^{-1}(0)))$. In other words, $\Phi(\tau)$ is the $(2N-1)$-st prerenormalization $p\cR^{2N-1}f_\tau$, taken in $\phi$-coordinate and then rescaled so that the first map in the pair acts on the interval $[0,1]$.

For $\tau_1=(\alpha_1,\phi_1,\psi_1),\tau_2=(\alpha_2,\phi_2,\psi_2)\in\mathfrak P_{\tl U,\tl t,\tl h}$, we define 
\begin{multline*}
\dist_{\mathfrak P}(\tau_1,\tau_2)=|\alpha_1-\alpha_2|+\sup_{z\in\tl U}|\phi_1(z)-\phi_2(z)| \\
+ \inf_{\tl\psi_1,\tl\psi_2}\left(\sup_{|\Im z|<\tl h}|\tl\psi_1(z)-\tl\psi_2(z)|\right),
\end{multline*}
where the infinum is taken over all $\tl\psi_1,\tl\psi_2$ that are lifts of $\psi_1$ and $\psi_2$ respectively via the projection~(\ref{proj_rel}). 
We note that $\dist_{\mathfrak P}(\cdot,\cdot)$ is a metric on the Banach manifold $\mathfrak P_{\tl U,\tl t,\tl h}$, induced by the norm in the Banach space that $\mathfrak P_{\tl U,\tl t,\tl h}$ is modeled on.

The following two lemmas establish a relation between the two metrics $\dist_r(\cdot,\cdot)$ and $\dist_{\mathfrak P}(\cdot,\cdot)$ in a neighborhood of the attractors $\mathcal I_B$ and $\check{\mathcal I}_B$ respectively.

\begin{lemma}\label{Phi_norm_bound_lemma}
For any $B\in\bbN$ and any closed interval $J\subset I(B)$, where $I(B)$ is the interval from Corollary~\ref{main3_extended}, there exist positive constants $C=C(B,J)>0$, $r=r(B,J)>0$ and an open set $\tilde{\mathcal U}_{B,J}\subset\tilde{\mathcal U}_0$, such that the following holds:
\begin{enumerate}[(i)]
\item 
$$
\bigcup_{\alpha\in J}\check{\mathcal I}_B^\alpha\subset\tilde{\mathcal U}_{B,J};
$$
\item\label{dist_ineq_Phi_cond} 
for every $\tau\in\tilde{\mathcal U}_{B,J}$, the critical commuting pair $\Phi(\tau)$ analytically extends to a pair from $\mathcal B_{r}^J$ and for any $\tau_1,\tau_2\in\tilde{\mathcal U}_{B,J}$ we have $\dist_{\mathfrak P}(\tau_1,\tau_2)\ge C\dist_{r}(\Phi(\tau_1),\Phi(\tau_2))$.
\item\label{capture_region_cond}
for every $\tau=(\alpha,\phi,\psi)\in\tilde{\mathcal U}_{B,J}$ with $\alpha\in J$, the sequence of iterates 
$$\cR_{\mathfrak P}(\tau),\cR_{\mathfrak P}^2(\tau),\cR_{\mathfrak P}^3(\tau),\dots$$
either eventually leaves the set $\tilde{\mathcal U}_{B,J}$, or stays in it forever and converges to the attractor $\check{\mathcal I}_B^\alpha$.
\end{enumerate}
\end{lemma}
\begin{proof}

The set $\cup_{\alpha\in J}\check{\mathcal I}_B^\alpha$ is compact, hence there exists a positive real number $r>0$, such that for every $\tau\in\cup_{\alpha\in J}\check{\mathcal I}_B^\alpha$, the critical commuting pair $\Phi(\tau)$ analytically extends to a pair from $\mathcal B_{2r}^J$, and we can choose an open set $\mathcal V_{B,J}\subset \tilde{\mathcal U}_0$, so that $\cup_{\alpha\in J}\check{\mathcal I}_B^\alpha\subset\mathcal V_{B,J}$, and for any $\tau\in\mathcal V_{B,J}$, the pair $\Phi(\tau)$ analytically extends to a pair from $\mathcal B_{r}^J$. 

Let $i$ be the map defined in~(\ref{norm_inclusion_eq}).
Since according to Lemma~\ref{Choose_hr_tilde_lemma}, the map $f_\tau$ analytically depends on $\tau$, 
it follows from the definition of the map $\Phi$ that the composition $i\circ\Phi$ is an analytic map from the neighborhood $\mathcal V_{B,J}\subset\mathfrak P_{\tl U,\tl t,\tl h}$ to the Banach space $\mathfrak D(N_r([0,1])\setminus(-\infty,0])\times \mathfrak D(N_r([0,1])\setminus(-\infty,0])$.

The derivative of an analytic operator is locally bounded, which means that there exists a constant $c>0$ and an open set $\tilde{\mathcal V}_{B,J}\subset\mathcal V_{B,J}$, such that $\cup_{\alpha\in J}\check{\mathcal I}_B^\alpha\subset\tilde{\mathcal V}_{B,J}$, and $\|D_\tau(i\circ\Phi)\|<c$, for all $\tau\in\tilde{\mathcal V}_{B,J}$.

Since the set $\cup_{\alpha\in J}\check{\mathcal I}_B^\alpha$ is compact, the set $\Phi(\cup_{\alpha\in J}\check{\mathcal I}_B^\alpha)$ is also compact, hence it has a finite diameter with respect to the Banach norm. Finally, we choose an open set $\tilde{\mathcal U}_{B,J}\subset\tilde{\mathcal V}_{B,J}$ so that $\cup_{\alpha\in J}\check{\mathcal I}_B^\alpha\subset\tilde{\mathcal U}_{B,J}$, the set $\Phi(\tilde{\mathcal U}_{B,J})$ has a finite diameter $D>0$, and if $\tau_1,\tau_2\in\tilde{\mathcal U}_{B,J}$, $\dist_{\mathfrak P}(\tau_1,\tau_2)<\varepsilon$, for some sufficiently small $\varepsilon>0$, then $\tau_1$ and $\tau_2$ can be connected by the shortest geodesic contained in $\tilde{\mathcal V}_{B,J}$. Now it is easy to check that for any $\tau_1,\tau_2\in\tilde{\mathcal U}_{B,J}$ and $C=\min(c, \varepsilon/D)$, the inequality
$$
\dist_{\mathfrak P}(\tau_1,\tau_2)\ge C\|i(\Phi(\tau_1))-i(\Phi(\tau_2))\|
$$
holds, which immediately implies property~\ref{dist_ineq_Phi_cond} of the lemma.

Finally, it follows from Corollary~\ref{main3_extended} that shrinking the open set $\tilde{\mathcal U}_{B,J}$ if necessary, we can guarantee that property~\ref{capture_region_cond} of the lemma also holds.
\end{proof}

\begin{lemma}\label{Psi_lemma}
Fix $B\in\bbN$ and $\mu=\mu(I)$, where $I=I(B)$ is the interval from Corollary~\ref{main3_extended} and $\mu(I)$ is the same as in Theorem~\ref{complex bounds_extended}. 
Then there exists a set $\mathcal W_B\subset \mathcal B^I$ and a map $\Psi\colon{\mathcal W_B}\to \mathfrak P_{\tilde U,\tilde t,\tilde h}$, such that the following properties hold:

\begin{enumerate}[(i)]
\item\label{norm_lemma_prop_1} $\mathcal W_B$ is the projection of an open neighborhood $\tilde{\mathcal W}_B\subset\hol^I(\mu)$ under the map $\cH\mapsto\zeta_\cH$, and $\tilde{\mathcal I}_B\subset\tilde{\mathcal W}_B$;
\item\label{norm_lemma_prop_2} for every integer $m\ge 0$, the following holds whenever both sides are defined:
$$
\Psi(\cR^{mN}(\Phi(\tau)))=\cR_{\mathfrak P}^{m+2}\tau;
$$
\item\label{norm_lemma_prop_3} for every integer $m\ge 0$, the following holds whenever both sides are defined:
$$
\Phi(\cR_{\mathfrak P}^m(\Psi(\zeta)))= \cR^{N(m+2)}\zeta;
$$
\item\label{norm_lemma_prop_4} there exists a closed interval $J\subset I$, such that $3\in J$ lies in the interior of $J$ and 
$$\Psi(\mathcal W_B)\subset\tilde{\mathcal U}_{B,J},$$
where $\tilde{\mathcal U}_{B,J}$ is the set from Lemma~\ref{Phi_norm_bound_lemma};

\item\label{norm_lemma_prop_5} there exist two real constants $s>0$ and $C=C(B,\alpha)>0$, such that every $\zeta\in\mathcal W_B$ analytically extends to a pair from $\mathcal B^I_s$, and for any $\zeta_1,\zeta_2\in\mathcal W_B$ with critical exponent $\alpha$, we have $\dist_s(\zeta_1,\zeta_2)\ge C\dist_{\mathfrak P}(\Psi(\zeta_1),\Psi(\zeta_2))$.

\end{enumerate}
\end{lemma}
\begin{proof}
Lemmas~\ref{Choose_U_B_lemma},~\ref{Choose_t_lemma} and~\ref{Choose_r_h_lemma} provide a continuous mapping 
$$
\tilde\Psi\colon\cH\mapsto (3,\phi_\cH,\psi_\cH)
$$
from $\tilde{\mathcal K}\subset\mathbf H^3(\mu)$ to $\mathfrak P_{\tilde U,\tilde t,\tilde h}$. According to Lemma~\ref{Nearby_crescent_lemma}, the fundamental crescent $C_\cH$ of a holomorphic pair $\cH$ and its fattening $C_\cH^o$ can be chosen to depend locally continuously with respect to $\cH\in\mathbf H^I(\mu)$. Then since $\tilde{\mathcal K}$ is a countable union of nested topological Cantor sets, the standard argument from the proofs of above mentioned lemmas shows that the map $\tilde\Psi$ continuously extends to some open set $\tilde{\mathcal W}\subset\mathbf H^I(\mu)$, such that $\tilde{\mathcal K}\subset\tilde{\mathcal W}$, and for every $\cH\in\tilde{\mathcal W}$ with $\tilde\Psi(\cH)=(\alpha,\phi_\cH,\psi_\cH)$ the identity 
\begin{equation}\label{R_cH_pm_eq}
 \tilde{\pi}_{C_{\cH}}\circ R_{C_{\cH}}^\pm\circ\pi_{C_{\cH}}^{-1} \equiv {\psi_\cH}\circ\pi_{{\phi_\cH},\alpha,\tl t}\circ p_{{\phi_\cH},\alpha\pm}
\end{equation}
holds in the domain $\{z\in\bbC\mid |\Re z|<1/2; |\Im z|<r_2\}$, 
where $r_2>0$ is the same as in Lemma~\ref{U_lemma}, the family of maps $\{\pi_{\phi,\alpha,\tl t}\mid (\alpha,\phi)\in\mathfrak B_{\tilde U,\tilde t}\}$ is the same as in Lemma~\ref{Choose_r_h_lemma}, and $R_{C_{\cH}}^+(z)$, $R_{C_{\cH}}^-(z)$ are the first returns of $\hat\eta(z)$ and $\hat\xi(z)$ respectively to the fundamental crescent $C_\cH$ under the dynamics of $\cH$. (Here $\hat\eta$ and $\hat\xi$ are the same as in Definition~\ref{Holo_pairs_new_def}.)

We note that if $\mathcal W\subset\mathcal B^I$ is the projection of the set $\tilde{\mathcal W}$ under the map $\cH\mapsto\zeta_\cH$ and $h\colon\mathcal W\to\tilde{\mathcal W}$ is some continuous lift of $\mathcal W$ with respect to this projection, then the map 
$$
\Psi=\tilde\Psi\circ h
$$
defined on the set $\mathcal W$, satisfies properties~\ref{norm_lemma_prop_2} and~\ref{norm_lemma_prop_3} of the lemma. Furthermore, for every $B\in\bbN$ we can choose an open subset $\tilde{\mathcal W}_B\subset\tilde{\mathcal W}$, such that $\tilde{\mathcal I}_B\subset\tilde{\mathcal W}_B$, and define $\mathcal W_B\subset\mathcal W$ as the projection of this set under the map $\cH\mapsto\zeta_\cH$, so that the restriction of $\Psi$ to $\mathcal W_B$ satisfies property~\ref{norm_lemma_prop_4}.

Now we will show, that possibly, after shrinking the set $\tilde{\mathcal W}_B$ and correspondingly its projection $\mathcal W_B$, we can establish property~\ref{norm_lemma_prop_5}.

First of all, since every holomorphic pair $\cH\subset\tilde{\mathcal K}$ restricts to a pair from $\mathcal B_{2s}^3$, for some constant $s>0$ independent of $\cH$, by shrinking the open set $\tilde{\mathcal W}_B$ if necessary, we can guarantee that every $\cH\subset\tilde{\mathcal W}_B$ restricts to a pair from $\mathcal B_{s}^I$.

Now we note that for every $\cH=(\eta,\xi)\in\tilde{\mathcal W}_B$ and the corresponding triple $(\alpha,\phi_\cH,\psi_\cH)=\tilde\Psi(\cH)$, the map $\phi_\cH$ is completely determined by the map $\eta$. Furthermore, according to Lemma~\ref{Nearby_crescent_lemma}, at every such $\eta\colon U_\cH\to\Delta_\cH$, the correspondence $\eta\mapsto\phi_\cH$ extends to a locally analytic map from a neighborhood in $\mathbf D(\hat U_\cH)$ to $\mathbf D(\tl U)$, where $\hat U_\cH$ is some open set, compactly contained in $U_\cH$. Analyticity of this map implies local boundedness of its derivative, and the Koebe distortion theorem implies that for all $\cH_0=(\eta_0,\xi_0)$ from some neighborhood of $\cH$ in $\tilde{\mathcal W}_B$, we have 
$$
\sup_{z\in N_s([0,1])}|\eta_0(z)|\ge c_1 \sup_{z\in \hat U_\cH}|\eta_0(z)|, 
$$
for some universal constant $c_1>0$.
Now, arguing in the same way as in the end of the proof of Lemma~\ref{Phi_norm_bound_lemma}, using the above inequality and analyticity of the map $\eta\mapsto\phi_\cH$ from $\mathbf D(\hat U_\cH)$ to $\mathbf D(\tl U)$,
we conclude that possibly after shrinking the neighborhood $\tilde{\mathcal W}_B$, the following conditions are satisfied: $\tilde{\mathcal I}_B\subset\tilde{\mathcal W}_B$ and for evrey $\cH_1,\cH_2\in\tilde{\mathcal W}_B$, the inequality 
$$
\dist_s(\cH_1,\cH_2)\ge c_2 \sup_{z\in\tl U}|\phi_{\cH_1}(z)-\phi_{\cH_2}(z)|
$$
holds for some constant $c_2>0$, independent from $\cH_1$ and $\cH_2$.

Similarly to how it is done in the proof of Lemma~\ref{Choose_r_h_lemma}, we can represent the left-hand side of~(\ref{R_cH_pm_eq}) as 
$$
(\tilde\pi_{C_{\mathcal H}}\circ R_{C_{\mathcal H}}^\pm\circ\pi^{-1}_{C_{\mathcal H}})(z)= \tilde\pi_{C_{\mathcal H}}(g( p_{\phi_\cH,\alpha\pm}(z) )),
$$
where $g$ is a conformal map that is a composition of a linear rescaling by $p_{\alpha-}(1/\pi_{C_\cH}'(0))$, the map $\rho_{\zeta_\cH}$ from property~\ref{Generalized_cr_com_p_property_rho} of Definition~\ref{Gccp_def} and a finite number of the maps $\eta$. Combining this with~(\ref{R_cH_pm_eq}), we get that 
$$
{\psi_\cH}\equiv \tilde\pi_{C_{\mathcal H}}\circ g \circ\pi_{{\phi_\cH},\alpha,\tl t}^{-1},
$$
so $\psi_\cH$ analytically depends on $\phi_\cH\in \mathfrak B_{\tilde U,\tilde t}^\alpha$, $\rho_{\zeta_\cH}\in\mathbf D(\bbD_{\hat r^\alpha})$ and $\eta\in\mathbf D(N_s([0,1])\setminus(-\infty,0])$, hence has a locally bounded derivative with respect to each of them.
Finally, applying the argument from the end of the proof of Lemma~\ref{Phi_norm_bound_lemma}, shrinking the open set $\tilde{\mathcal W}_B$ and its projection $\mathcal W_B$ again if necessary, we obtain property~\ref{norm_lemma_prop_5} of the lemma while property~\ref{norm_lemma_prop_1} remains true.
\end{proof}

\subsection{Global attractor}\label{Glob_Attr_proof_sec}
Now we are ready to give a proof of Theorem~\ref{Global_attractor_theorem_expanded}.

\begin{proof}[Proof of Theorem~\ref{Global_attractor_theorem_expanded}]
We will give a proof for the case $k=1$. For other $k\in\bbN$ the proof is identical.

Let $I\subset\bbR$ be the interval from Corollary~\ref{main3_extended}. Then for every $\alpha\in I$, we define 
$$
\mathcal I_B^\alpha=\Phi(\check{\mathcal I}_B^\alpha).
$$
It follows from Lemma~\ref{Choose_K_lemma} and the choice of the constant $N$, made in Section~\ref{Cyl_ren_sec}, that for all $\alpha$ from some open interval $J_1\subset\bbR$ with $3\in J_1$, every element of $\Phi(\check{\mathcal I}_B^\alpha)$ extends to a holomorphic commuting pair from $\tilde{\mathcal W}_B$, where $\tilde{\mathcal W}_B$ is the same as in property~\ref{norm_lemma_prop_1} of Lemma~\ref{Psi_lemma}. Thus, for all $\alpha\in J_1$, we have
$$
\mathcal I_B^\alpha\subset\mathcal W_B.
$$
Furthermore, since according to Corollary~\ref{main3_extended}, the operator $\cR_{\mathfrak P}$ is bijective on $\check{\mathcal I}_B^\alpha$, property~\ref{norm_lemma_prop_2} of Lemma~\ref{Psi_lemma} implies that the map $\Phi$ provides a bijection, hence a homeomorphism, between $\check{\mathcal I}_B^\alpha$ and $\mathcal I_B^\alpha$. This homeomorphism, composed with the map $\kappa_\alpha$ from Corollary~\ref{main3_extended}, induces a homeomorphism 
$$\iota_\alpha\colon\mathcal I_B^\alpha\to\Sigma_B.$$

First, we will prove part (ii) of Theorem~\ref{Global_attractor_theorem_expanded} only for $\zeta\in\mathcal B^\alpha$. We start with a proposition:

\begin{proposition}\label{iterate_into_W_B_proposition}
Let $\mu=\mu(\overline I)$ be the same as in Theorem~\ref{complex bounds_extended}. Then there exists an open interval $J_2\subset I$ and a positive integer $L\in\bbN$, such that $3\in J_2$ and if a commuting pair $\zeta\in\mathcal B^{J_2}$ with an irrational rotation number of type bounded by $B$ extends to a holomorphic commuting pair from $\mathbf H^{J_2}(\mu)$, then $\cR^L\zeta\in\mathcal W_B$.
\end{proposition}
\begin{proof}
It follows from compactness of $\mathbf H^3(\mu)$ and real {\it a priori bounds} \cite{Herm,YocDen} that there exists a positive integer $K>0$, such that for every $\cH\in\mathbf H^{3}(\mu)$ with $\rho(\zeta_\cH)$ of type bounded by $B$ and every $k\ge K$, the range of the renormalization $\cR^k(\cH)$ contains a disk of radius $2/\mu$.
Then, since the set $\mathcal I_B^3$ is a global attractor of type bounded by $B$ (c.f. Remark~\ref{global_attractor_remark}), and the class $\mathbf H^3(\mu)$ is sequentially compact, there exists a positive integer $L>K$, such that for every $\cH\in\mathbf H^3(\mu)$ with $\rho(\zeta_\cH)$ of type bounded by $B$, the renormalization $\cR^{L}(\cH)$ restricts to a holomorphic pair $\mathcal G\in\tilde{\mathcal W}_B$, where $\tilde{\mathcal W}_B\subset\mathbf H^I(\mu)$ is an open set from property~\ref{norm_lemma_prop_1} of Lemma~\ref{Psi_lemma}.

Finally, by continuity and sequential compactness of $\mathbf H^{\overline I}(\mu)$ we conclude the existence of an open interval $J_2$ with $3\in J_2$, such that for every $\cH\in\mathbf H^{J_2}(\mu)$ with $\rho(\zeta_\cH)$ of type bounded by $B$, the renormalization $\cR^{L}(\cH)$ restricts to a holomorphic pair from $\tilde{\mathcal W}_B$. This completes the proof of the proposition.
\end{proof}

We define $J=J_1\cap J_2$ and we set $\hat r=\min\{s, r(B,\overline J)\}$, where $s$ is the same as in property~\ref{norm_lemma_prop_5} of Lemma~\ref{Psi_lemma} and $r(B,\overline J)$ is the same as in Lemma~\ref{Phi_norm_bound_lemma}. Assume that $\zeta\in\mathcal B^\alpha$, for some $\alpha\in J$, and $\rho(\zeta)$ is of type bounded by $B$. Then, according to the complex bounds (Theorem~\ref{complex bounds_extended}), there exists an integer $K>0$, such that for all $k\ge K$, the renormalization $\cR^k\zeta$ extends to a holomorphic commuting pair from $\mathbf H^J(\mu)$. Then Proposition~\ref{iterate_into_W_B_proposition} implies that $\cR^{k+L}\zeta\in\mathcal W_B$. Now it follows from property~\ref{norm_lemma_prop_5} of Lemma~\ref{Psi_lemma} that $\cR^{k+L}\zeta$ extends to a commuting pair from $\mathcal B_{\hat r}^\alpha$, and, finally, exponential convergence~(\ref{pair_exp_convergence_eq}) for any $r\le\hat r$ follows from properties~\ref{norm_lemma_prop_3}-\ref{norm_lemma_prop_5} of Lemma~\ref{Psi_lemma} together with Lemma~\ref{Phi_norm_bound_lemma} and Corollary~\ref{main3_extended}.

Now we will prove part (ii) of Theorem~\ref{Global_attractor_theorem_expanded} for all $\zeta\in\mathcal A^\alpha$. Assume, $\zeta\in\mathcal A^\alpha$ has an irrational rotation number of type bounded by $B$. Then according to the real a priori bounds \cite{Herm,YocDen}, there exists a positive integer $K>0$, such that the dynamical intervals of $p\cR^K\zeta$ are contained in the domain of the map $\phi_\zeta$ from Definition~\ref{A_r_alpha_def}. Thus, there exists a conformal map $\phi_0$ with $\phi_0(0)=0$ and $\phi_0'(0)=1$, such that $\cR^K\zeta=(\phi_0^{-1}\circ\eta_0\circ\phi_0, \phi_0^{-1}\circ\xi_0\circ\phi_0)$, where $\zeta_0=(\eta_0,\xi_0)\in\mathcal B^\alpha$.

Now it is easy to check that for every $n\in\bbN$, we have
$$
\cR^{K+n}\zeta= (\phi_n^{-1}\circ\eta_n\circ\phi_n, \phi_n^{-1}\circ\xi_n\circ\phi_n),
$$
where $\zeta_n=(\eta_n,\xi_n)=\cR^n\zeta_0$, and $\phi_n$ can be defined recurrently by
\begin{equation}\label{phi_n_recurrent_eq}
\phi_n(z)=\frac{1}{\eta_{n-1}(0)}\phi_{n-1}(\eta_{n-1}(0)z).
\end{equation}

Since $\zeta_0\in\mathcal B^\alpha$, the above argument implies that all sufficiently high renormalizations of $\zeta_0$ lie in $\mathcal B_{\hat r}^\alpha$, and exponential convergence~(\ref{pair_exp_convergence_eq}) with $r\le\hat r$ holds for these renormalizations. At the same time, real bounds and~(\ref{phi_n_recurrent_eq}) imply that the domains of the maps $\phi_n$ increase exponentially, and the maps $\phi_n$ converge to the identity map exponentially fast in sup-norm on any compact set. This implies that all sufficiently high renormalizations of $\cR^K\zeta$ belong to the class $\mathcal A_{\hat r/2}^\alpha$, and satisfy the exponential convergence condition~(\ref{pair_exp_convergence_eq}) for all $r\le \hat r/2$. This completes the proof of part (ii) of Theorem~\ref{Global_attractor_theorem_expanded}.

Finally, we give a prove of part (i) of Theorem~\ref{Global_attractor_theorem_expanded}. It follows from the construction of the map $\Phi$, that this map conjugates the operators $\cR_{\mathfrak P,\alpha}$ and $\cR^N$ on the sets $\check{\mathcal I}_B^\alpha$ and $\mathcal I_B^\alpha$ respectively. This implies that the set $\cR(\mathcal I_B^\alpha)$ is an invariant set for the operator $\cR^N$. Then it follows from part (ii) of Theorem~\ref{Global_attractor_theorem_expanded} that $\cR(\mathcal I_B^\alpha)\subset\mathcal I_B^\alpha$. 
This means that the composition $\iota_\alpha\circ\mathcal R\circ \iota_\alpha^{-1}$ is defined for all $\alpha\in J$ and depends continuously on $\alpha$. Then since $\Sigma_B$ is a totally disconnected space, this composition must be independent from $\alpha$. Now part (i) of Theorem~\ref{Global_attractor_theorem_expanded} follows from the fact that for $\alpha=3$, this composition is the shift $\sigma$.
\end{proof}

\subsection{Proof of Universality (Theorem \ref{main5})}
Let  $B$, $N$ be as before. Denote $\theta\in \cI^\alpha_B$ the periodic orbit of $\cR$ given by
$$\cR^{lp}f_0\to \theta$$ (Theorem~\ref{Global_attractor_theorem_expanded}). 
Let $l$ be sufficiently large, so that for some real number $t_1>0$ and for all $t\in[-t_1,t_1]$, the domain of definition of the prerenormalization $p\cR^lf_t\equiv(\tilde{\eta}_t,\tilde{\xi}_t)$ is contained in $V_2$. It follows from complex bounds (Theorem~\ref{complex bounds_extended}) that possibly, after increasing $l$ and decreasing $t_1>0$, the commuting pairs 
$$
\zeta_t\equiv(h_t\circ\phi_t\circ\tilde{\eta}_t\circ\phi_t^{-1}\circ h_t^{-1},\,\, h_t\circ\phi_t\circ\tilde{\xi}_t\circ\phi_t^{-1}\circ h_t^{-1}),
$$
where $h_t(z)=z/\phi_t(\tilde{\xi}_t(\phi_t^{-1}(0)))$, belong to $\mathcal B^\alpha$ and extend to holomorphic pairs from $\mathbf H^\alpha(\mu)$, for all $t\in[-t_1,t_1]$. Then according to Proposition~\ref{iterate_into_W_B_proposition}, after further increasing $l$ and decreasing $t_1>0$, we can make sure that $\zeta_t\in\mathcal W_B$, for all $t\in[-t_1,t_1]$. Then $\tau_t\equiv\Psi(\zeta_t)$ is defined for all $t\in[-t_1,t_1]$ and $\tau_0$ belongs to the local stable manifold of $\Psi(\theta)$.

It was shown in~\cite{Ya3} (Lemma~9.3) that the condition $$\inf_{x\in\bbT}\frac{\partial}{\partial t}f_t(x)>0$$ implies that the tangent vector to the family $\{f_{\tau_t}\}$ at $t=0$ lies inside an invariant cone field $\mathcal C$ defined on the tangent bundle to the set of all infinitely cylinder renormalizable $f\in \mathfrak C_{\tl r}^\alpha$. Furthermore, it was shown in~\cite{Ya3} that this cone field is expanded by $D\cren$. Pulling back $\mathcal C$ by the analytic mapping $\tau\mapsto f_\tau$, gives a cone field $\mathcal C'$ defined on a subset of the tangent bundle of $\mathfrak P_{\tl U,\tl t,\tl h}^\alpha$. It follows from Proposition~\ref{basic_R_P_prop} that $\mathcal C'$ is invariant and expanding under the map $D\cR_{\mathfrak P,\alpha}$. Since the tangent vector to the family $\{\tau_t\}$ at $t=0$ lies in $\mathcal C'$, the expanding property of $\mathcal C'$ implies that the family $\{\tau_t\}$ intersects the local stable manifold $W_{\text{loc}}^s(\Psi(\theta))$ transversely. Now the statement of Theorem~\ref{main5} follows immediately, with $\delta$ being equal to the absolute value of the unstable eigenvalue of the linearization of  $\cR_{{\mathfrak P},\alpha}^p$ at $\Psi(\theta)$.

\subsection{Proof of $C^{1+\beta}$-rigidity (Theorem \ref{main6})}
The proof repeats identically the argument from~\cite{FM1}. We show that exponential convergence of renormalizations in the case of bounded combinatorics implies $C^{1+\beta}$-rigidity.

For a critical circle map $f\colon\bbR\slash\bbZ\to\bbR\slash\bbZ$ with an irrational rotation number, let $I_m$ be as in~(\ref{I_m_def}), and define $I_m^j=f^j(I_m)$. The dynamical partition of level $n$, associated to $f$, is defined as
$$
\mathcal P_n(f)=\{I_{n-1},I_{n-1}^1, \dots,I_{n-1}^{q_n-1}\}\cup \{I_n,I_n^1,\dots,I_N^{q_{n-1}-1}\}.
$$
The intervals from $\mathcal P_n(f)$ partition the circle modulo the endpoints, and $\mathcal P_{n+1}(f)$ is a refinement of the partition $\mathcal P_n(f)$. 

The following proposition is a reformulated version of Proposition~4.3 from~\cite{FM1}.
\begin{proposition}\label{dFdM43_prop}
Let $f$ be a critical circle map with an irrational rotation number of type bounded by $B$, and let $h\colon\bbR\slash\bbZ\to\bbR\slash\bbZ$ be a homeomorphism. If there exist constants $C>0$ and $0<\lambda<1$ such that
$$
\left|\frac{|I|}{|J|}-\frac{|h(I)|}{|h(J)|}\right|\le C\lambda^n,
$$
for each pair of adjacent atoms $I,J\in\mathcal P_n(f)$, for all $n\ge 0$, then $h$ is a $C^{1+\beta}$-diffeomorphism for some $\beta>0$ that depends only on $\lambda$ and $B$.
\end{proposition}

\begin{proof}[Proof of Theorem~\ref{main6}]
According to the real a priori bounds~\cite{Herm,YocDen}, the length of a maximal interval in $\mathcal P_n(f)$ goes to zero as $n\to\infty$. This implies that, given two critical circle maps $f_1,f_2$ with the same irrational rotation number, there exists a unique homeomorphism $h\colon\bbR\slash\bbZ\to\bbR\slash\bbZ$ that maps  $0\mapsto 0$ and conjugates $f_1$ with $f_2$. Now if the rotation number of $f_1$ and $f_2$ is of combinatorial type bounded by $B$, and the critical exponent of $f_1$ and $f_2$ lies in the union $\cup_{k\in\bbN} J(k,B)$, then Theorem~\ref{Global_attractor_theorem_expanded} together with the Koebe Distortion Theorem imply that the homeomorphism $h$ satisfies the conditions of Proposition~\ref{dFdM43_prop}, hence is a $C^{1+\beta}$-diffeomorphism.
\end{proof}

\bibliographystyle{amsalpha}
\bibliography{biblio}
\end{document}